\documentclass[10pt]{amsart}
\usepackage{files/style}
 
\pdfoutput=1 % for arxiv compatibility; dvi has issues with url line breaks

\begin{document}
\title{Equivariant operads, symmetric sequences, and Boardman-Vogt tensor products}
\author{Natalie Stewart}
\date{\today}
 
\begin{abstract}
  We advance the foundational study of be Nardin-Shah's $\infty$-category of $G$-operads and their associated $\infty$-categories of algebras.
  In particular, we construct the \emph{underlying $G$-symmetric sequence} of a (one color) $G$-operad, yielding a monadic functor; 
  we use this to lift Bonventre's genuine operadic nerve to a conservative functor of $\infty$-categories, restricting to an equivalence between categories of discrete $G$-operads.
  Using this, we extend Blumberg-Hill's program concerning \emph{$\cN_\infty$-operads} to arbitrary sub-operads of the terminal $G$-operad, which we show are equivalent to weak indexing systems.

  We then go on to define and characterize a homotopy-commutative and closed \emph{Boardman-Vogt tensor product} on $\Op_G$;
  in particular, this specializes to a $G$-symmetric monoidal $\infty$-category of $\cO$-algebras in a $G$-symmetric monoidal $\infty$-category whose $\cP$-algebras are objects with interchanging $\cO$-algebra and $\cP$-algebra structures.
\end{abstract}

\maketitle
 
\toc

\section*{Introduction}  
Within the burgeoning study of algebraic structures in $G$-equivariant homotopy theory, tensor products are generalized to \emph{indexed tensor products}, leading to the notion of \emph{$G$-symmetric monoidal $\infty$-categories} \cite{Hill_SMC,Bachmann}.
Naturally, $G$-equivariant algebraic theories are represented by \emph{$G$-operads}, including the equivariant little cubes/Seiner operads of \cite{Guillou-May} and $\cN_\infty$-operads of \cite{Blumberg-op}.
In this paper, we use $\infty$-categorical foundations to advance the homotopy theory of $G$-operads, both structurally on Nardin-Shah's $\infty$-category of $G$-$\infty$-operads $\Op_G$ (henceforth just $G$-operads) and individually on the $\infty$-categories of algebras $\Alg_{\cO}(\cC)$ of $\cO$-algebras for various examples of interest.\footnote{
  In this paper we will call $\infty$-categories \emph{$\infty$-categories} and $\infty$-categories with discrete mapping spaces \emph{1-categories}, as their theory is equivalent to the traditional theory of categories.
More generally, we will call $\infty$-categories whose mapping spaces are $(d-1)$-truncated \emph{$d$-categories}.}

Our first contribution generalizes the rudimentary theory of $G$-symmetric monoidal $\infty$-categories to \emph{$I$-symmetric monoidal $\infty$-categories}, for $I$ a \emph{weak} indexing category in the sense of \cite{Windex};
these posses indexed tensor products over a collection of arities only under the assumptions that they can be restricted and composed.

We go on to generalize $G$-operads to \emph{$I$-operads}, which occur as a full subcategory $\Op_I \subset \Op_G$ with a terminal object $\cN_{I \infty}^{\otimes}$, which we refer to as a \emph{weak $\cN_\infty$-operad};
in particular, an $I$-symmetric monoidal $\infty$-category $\cC^{\otimes}$ has an underlying (colored) $I$-operad of the same name, and $\cO$-algebras in $\cC^{\otimes}$ correspond with maps of $G$-operads $\cO^{\otimes} \rightarrow \cC^{\otimes}$.  
We combinatorially classify the weak $\cN_\infty$-operads as \emph{weak indexing systems}, generalizing \cite{Rubin,Gutierrez,Bonventre,Nardin}.

One of our central constructions is a monadic \emph{underlying $G$-symmetric sequence} functor 
\[
  \sseq\cln \Op_G^{\oc} \rightarrow \Fun(\tot \uSigma_G, \cS),
\]
the former being the \emph{one-colored} $G$-operads.
The objects of $\tot \uSigma_G$ are identified with pairs $(H,S)$ where $H \subset G$ is a subgroup and $S \in \FF_H$ is a finite $H$-set;
given this data, we write $\cO(S) \deq \sseq \cO^{\otimes}(S)$, which we call the \emph{$S$-ary structure space of $\cO^{\otimes}$}.
This intertwines with Bonventre's genuine operadic nerve, so the nerve lifts to a conservative functor of $\infty$-categories.

We use this data to characterize the compatible $(d+1)$-categories of \emph{$G$-symmetric monoidal $d$-categories} and \emph{$G$-$d$-operads}:
a $G$-operad $\cO^{\otimes}$ is a \emph{$G$-$d$-operad} if the $S$-ary structure space $\cO(S)$ is $(d-1)$-truncated for all subgroups $H \subset G$ and finite $H$-sets $S \in \FF_H$.
These are a localizing subcategory, and the corresponding \emph{homotopy $G$-$d$-operad} functor $h_d\colon \Op_G \rightarrow \Op_{G,d}$ acts on structure spaces as $(d-1)$-truncation.
We characterize the free $\cO$-algebra monad, showing that the functor $\Alg_{(-)}(\ucS_{G, \leq (d-1)})$ of algebras in $(d-1)$-truncated $G$-spaces detects $h_d$-equivalences between one color $G$-operads;
in particular, taking algebras in $G$-spaces is conservative.

When $d \leq 1$, we show that the restriction of Bonventre's nerve to genuine $G$-operads with $(d-1)$-truncated structure spaces maps equivalently onto $G$-$d$-operads, and we classify the $G$-0-operads as the weak $\cN_\infty$-operads. 
Using this, we classify the $d$-connected $I$-operads as those whose algebras in $d$-truncated $G$-spaces lift canonically to weak $\cN_\infty$-spaces.

Having done this, we define a homotopy-commutative tensor product on $\Op_G$ called the \emph{Boardman-Vogt tensor product} .
We show that this tensor product is \emph{closed}, i.e. it has an associated \emph{(colored) $G$-operad of algebras} $\uAlg_{\cO}^{\otimes}(\cC)$. 
When $\cC^{\otimes}$ is an $I$-symmetric monoidal $\infty$-category, we show that $\uAlg_{\cO}^{\otimes}(\cC)$ underlies an $I$-symmetric monoidal $\infty$-category, which we give the same name;
in particular, $\uAlg_{\cO}^{\otimes}(\cC)$ is an $I$-symmetric monoidal $\infty$-category whose $\cP$-algebras are characterized by the formula
\[
  \Alg_{\cP} \uAlg_{\cO}^{\otimes}(\cC) \simeq \Alg_{\cP \otimes \cO}(\cC).
\]
We thus interpret $\cP \otimes \cO$-algebras as \emph{homotopy coherently interchanging pairs of $\cP$-algebras and $\cO$-algebras};
indeed we give a ``bifunctor'' presentation generalizing \cite[\S~2.2.5.3]{HA}.

We end by developing an ``inflation and fixed points'' adjunction $\Infl_e^G\colon  \Op \rightleftarrows \Op_G\cln \Gamma^{G}$ and showing that it is compatible with Boardman-Vogt tensor products. 
We now move on to a more careful accounting of the background and main results of this paper.

\stoptocwriting % stop toc from writing 

\subsection*{Background and motivation}
Let $\cC$ be a semiadditive 1-category, i.e. a pointed 1-category whose \emph{norm} map $X \sqcup Y \rightarrow X \times Y$ is an isomorphism for all $X,Y \in \cC$.
Let $G$ be a finite group and let $\cO_G$ be the orbit category of $G$.\footnote{The \emph{orbit category} is the full subcategory of $G$-sets $\cO_G \subset \Set_G$ spanned by the homogeneous $G$-sets $[G/H]$ for $H \subset G$ a subgroup.}
Recall that a \emph{semi-Mackey functor} valued in $\cC$ is the data of:
\begin{itemize}
  \item a contravariant functor $R\cln \cO^{\op}_G \rightarrow \cC$, and
  \item a covariant functor $N\cln \cO_G \rightarrow \cC$
\end{itemize}
subject to the conditions that
\begin{enumerate}[label={(\alph*)}]
  \item for all $H \subset G$, the values $R([G/H])$ and $N([G/H])$ are isomorphic, and
  \item writing $R_K^H\cln R([G/H]) \rightarrow R([G/K])$ for the contravariant functoriality and $N_K^H \cln N([G/K]) \rightarrow N([G/H])$ for the covariant functoriality, $R$ and $N$ satisfy the \emph{double coset formula}
    \[
      R_J^H N_K^H(-) \simeq \sum_{g \in [J\backslash H / K]} N_{H \cap gKg^{-1}}^H \Res_K^H\prn{-}_g
    \]
    where $(-)_g$ denotes the covariant conjugation action and $[J \backslash G / K]$ is the set of \emph{double cosets}.
\end{enumerate}
Let $\Span(\FF_G)$ be the effective Burnside 1-category, whose objects are finite $G$-sets, whose morphisms $R_{XY}\cln X \rightarrow Y$ are given by isomorphism classes of spans $X \leftarrow R_{XY} \rightarrow Y$, and whose composition is given by pullback of spans
    \[\begin{tikzcd}[row sep=tiny]
	&& {R_{XZ}} \\
	& {R_{XY}} && {R_{YZ}} \\
  X && Y && Z.
	\arrow[from=1-3, to=2-2]
	\arrow[from=1-3, to=2-4]
	\arrow["\lrcorner"{anchor=center, pos=0.125, rotate=-45}, draw=none, from=1-3, to=3-3]
	\arrow[from=2-2, to=3-1]
	\arrow[from=2-2, to=3-3]
	\arrow[from=2-4, to=3-3]
	\arrow[from=2-4, to=3-5]
\end{tikzcd}\]
It is an observation due to Lindner \cite{Lindner} that (semi)-Mackey functors valued in $\cC$ are equivalently given by product preserving functors 
\[
  \Span(\FF_G) \rightarrow \cC.
\]
This appears as a straightforward generalization of the Lawvere theory $\Span(\FF)$ for commutative monoids, so we will refer to semi-Mackey functors as \emph{$G$-commutative monoids}.

Moreover, \emph{any} $\cC$ admits a universal map from a semiadditive category, given by the forgetful functor $U\cln \CMon(\cC) \rightarrow \cC$;
since $\Span(\FF_G)$ possesses an identity-on-objects anti-involution, it is semiadditive, and so $U$ induces an equivalence
\[
  \Fun^{\times}(\Span(\FF_G), \CMon(\cC)) \xrightarrow{\;\;\sim\;\;} \Fun^{\oplus}(\Span(\FF_G),\cC);
\]
in fact, replacing $\Span(\FF_G)$ with the effective Burnside \emph{2-}category of \cite{Barwick1} (whose 2-cells are isomorphisms of spans), $\cC$ with an $\infty$-category, and interpreting $\CMon(\cC)$ as $\EE_\infty$-monoids in $\cC$, the semiadditivization result for $\CMon(\cC)$ still holds \cite{Gepner}, and $\Span(\FF_G)$ is still semiadditive.
Thus we are justified in making the following definition.
\begin{definition*}
  The \emph{$\infty$-category of $G$-commutative monoids in $\cC$} is the product-preserving functor $\infty$-category
  \[
    \CMon_G(\cC) \deq \Fun^\times(\Span(\FF_G),\cC);
  \]
  the \emph{$\infty$-category of small $G$-symmetric monoidal $\infty$-categories} is
  \[
    \Cat_G^{\otimes} \deq \CMon_G(\Cat).\qedhere
  \]  
\end{definition*}
This recovers the notion of \cite{Nardin}, which generalizes the notion of \cite{Hill_SMC}.
Recall that we define $G$-$\infty$-categories to be categorical coefficient systems
\[
  \Cat_G \deq \Fun\prn{\cO_G^{\op},\cC};
\]
the $[G/H]$-value of a $G$-$\infty$-category $\cC$ will be written $\cC_H$, and the contravariant functoriality along $[G/K] \rightarrow [G/H]$ will be written $\Res_K^H \cln \cC_H \rightarrow \cC_K$.
$G$-symmetric monoidal $\infty$-categories $\cC^{\otimes}$ have underlying $G$-$\infty$-categories $\cC$ defined by the precomposition
\[
  \cC\cln \cO_G^{\op} \rightarrow \Span(\FF_G) \xrightarrow{\cC^{\otimes}} \Cat.
\]
Given a subgroup $H \subset G$ and a finite $H$-set $S$, we will write the value of $\cC^{\otimes}$ on $\Ind_H^G S$ as $\cC_S$, noting that there is a canonical equivalence $\cC_S \simeq \prod_{[H/K] \in \Orb(S)} \cC_K$.

We may induce the unique map of $H$-sets $S \rightarrow *_H$ to $G$ to construct a structure map $\Ind_H^G S \rightarrow [G/H]$,\footnote{See \cite{Dieck} for a discussion of induced $G$-sets.}
and covariant functoriality yields a natural \emph{$S$-indexed tensor product} operation
\[
  \bigotimes^S\cln \cC_S \rightarrow \cC_H.
\]
We may induce the \emph{orbit set} factorization $S \rightarrow \coprod_{[H/K] \in \Orb(S)} *_H \rightarrow *_H$ to yield a natural equivalence
\[
  \bigotimes^S_K X_K \simeq \bigotimes_{[H/K] \in \Orb(S)} N_K^H X_K.
\]
Similarly, contravariant functoriality yields an \emph{$S$-indexed diagonal} $\Delta^S\cln \cC_H \rightarrow \cC_S$ satisfying
\[
  \Delta^S X \simeq \prn{\Res_K^H X}_{[H/K] \in \Orb(S)}.
\]
This allows us to define \emph{$S$-indexed tensor power} of an object $X_H \in \cC_H$ by
\[
  X_H^{\otimes S} \deq \bigotimes^S \Delta^S X_H \simeq \bigotimes^S_K \Res_K^H X_H \simeq \bigotimes_{[H/K] \in \Orb(S)} N_K^H \Res_K^H X_H.
\]
Akin to the discrete case, these satisfy a double coset formula by functoriality under the composite span
\[
  \begin{tikzcd}[row sep = tiny]
    && {\coprod\limits_{g \in [J \backslash H / K]} G / (K \cap gJg^{-1})} \\
    & {G/J} && {G/K} \\
    {G/J} && {G/H} && {G/K}
    \arrow[from=1-3, to=2-2]
    \arrow[from=1-3, to=2-4]
    \arrow["\lrcorner"{anchor=center, pos=0.125, rotate=-45}, draw=none, from=1-3, to=3-3]
    \arrow[from=2-2, to=3-3]
    \arrow[from=2-4, to=3-3]
    \arrow[Rightarrow, no head, from=2-4, to=3-5]
    \arrow[Rightarrow, no head, from=3-1, to=2-2]
  \end{tikzcd}
\]
\begin{example*}
  Write $\ucS_G$ for the $G$-$\infty$-category with $H$-value $\prn{\ucS_G}_H \deq \cS_H \simeq \Fun(\cO_H^{\op},\cS)$ the $\infty$-category of genuine $H$-equivariant spaces.
  This possesses a $G$-symmetric monoidal structure $\ucS_G^{G-\times}$ whose $S$-ary tensor product is the \emph{$S$-indexed product} \cite{Nardin};
  in particular, $\cS_H$ is a cartesian symmetric monoidal $\infty$-category and $N_K^H \simeq \CoInd_K^H\colon \cS_K \rightarrow \cS_H$ is right adjoint to restriction.
\end{example*}
\begin{example*}
  There is a $G$-symmetric monoidal $\infty$-category $\uSp_G^{\otimes}$ whose $H$-value $\prn{\uSp_G}_H \simeq \Sp_H$ is the $\infty$-category of genuine $H$-spectra with norms given by the Hill-Hopkins-Ravanel norm \cite{Nardin,Bachmann}.
\end{example*}

We are concerned with algebraic structures \emph{inside} $G$-symmetric monoidal $\infty$-categories, which we will control with a version of Nardin-Shah's $\infty$-category $\Op_G$ of $G$-$\infty$-operads, which we simply call \emph{$G$-operads}.
Work of Barkan, Haugseng, and Steinebrunner \cite{Barkan} identifies these with functors of $\infty$-categories $\pi_{\cO}\cln \cO^{\otimes} \rightarrow \Span(\FF_G)$ possessing cocartesian lifts over backwards maps and satisfying a pair Segal conditions, which we may summarize in two cases of interest:
\begin{enumerate}
  \item in the case that the fibers $\pi_{\cO}^{-1}(S)$ are contractible for all $S \in \FF_G$ (i.e. $\cO^{\otimes}$ \emph{has one color}), cocartesian lifts over the backwards maps $(S \leftarrow [G/H] = [G/H])_{[G/H] \in \Orb(S)}$ furnish an equivalence
    \[
      \Map^{T \rightarrow S}_{\pi_{\cO}} \prn{iT,iS} \simeq \prod_{[G/H] \in \Orb(S)} \Map^{T_H \rightarrow [G/H]}_{\pi_{\cO}}(iT_H,i[G/H]),
    \]
    where we set $T_H \deq T \times_S [G/H]$ and we write $iS$ for the unique object of $\pi_{\cO}^{-1}(S)$;\footnote{Given a functor $F\cln \cC \rightarrow \cD$, and $\psi\cln FX \rightarrow FY$ a map in $\cD$, we write $\Map^{\psi}_{F}(X,Y) \subset \Map_{\cC}(X,Y)$ for the disjoint union of the connected components consisting of maps $\varphi\cln X \rightarrow Y$ such that $F \varphi$ is homotopic to $\psi$.}
  \item in the case that $\pi_{\cO}$ is a cocartesian fibration, $\cO^{\otimes}$ is a $G$-operad if and only if it is the unstraightening of a $G$-symmetric monoidal $\infty$-category.
\end{enumerate}
These span a localizing subcategory \cite[Cor~4.2.3]{Barkan}
\begin{equation}\label{LOpG equation}
  \begin{tikzcd}
    {\Op_G} & {\Cat_{/\Span(\FF_G)}^{\Int-\cocart},}
    \arrow[""{name=0, anchor=center, inner sep=0}, curve={height=18pt}, hook, from=1-1, to=1-2]
    \arrow[""{name=1, anchor=center, inner sep=0}, "{L_{\Op_G}}"', curve={height=18pt}, from=1-2, to=1-1]
    \arrow["\dashv"{anchor=center, rotate=-90}, draw=none, from=1, to=0]
  \end{tikzcd}
\end{equation}
the latter denoting the non-full subcategory $\Cat_{/\Span(\FF_G)}^{\Int-\cocart} \subset \Cat_{/\Span(\FF_G)}$ whose objects possess cocartesian lifts over backwards maps and whose morphisms preserve these cocartesian lifts.

Given $\cO^{\otimes}$ a one-color $G$-operad, $H \subset G$ a subgroup, and $S \in \FF_H$ a finite $H$-set, we write 
\[
  \cO(S) \deq \Map^{\Ind_H^G S \rightarrow [G/H]}_{\pi_{\cO}}(i\Ind_H^G S,i[G/H])
\]
for the \emph{$S$-ary structure space of $\cO^{\otimes}$}.
\begin{example*}
  Let $I \subset \FF_{G}$ be a pullback-stable and core-full subcategory.
  In \cref{I operads subsection} we show that the subcategory $\Span_I(\FF_{G}) \subset \Span(\FF_G)$ presents a $G$-operad if and only if $I$ is a weak indexing category in the sense of \cite{Windex}, in which case we refer to the resulting $G$-operad as $\cN_{I \infty}^{\otimes}$.
  We refer to these together as the class of \emph{weak $\cN_\infty$-operads}.
  These are identified by their structure spaces
  \[
    \cO(S) \simeq \begin{cases}
      * & \Ind_H^G S \rightarrow [G/H] \in I; \\ 
      \emptyset & \mathrm{otherwise}. 
    \end{cases}
  \]
\end{example*} 

An $\cO$-algebra in $\cC^{\otimes}$ is defined to be a map of $G$-operads $\cO^{\otimes} \rightarrow \cC^{\otimes}$;
these posses an underlying $G$-object $X_\bullet$ (extending canonically to a cocartesian section of $\cC \rightarrow \cO_G^{\op}$, with canonical equivalences $X_K \simeq \Res^H_K X_H$) together with action maps 
\begin{equation}\label{action map}
  \cO(S) \rightarrow \Map_{\cC_H}\prn{X^{\otimes S}_H, X_H}
\end{equation}
for each subgroup $H \subset G$ and finite $H$-set $S \in \FF_H$, suitably functorial and compatible with cocartesian lifts of backwards maps.
In fact, as in \cite{Nardin}, we may lift these to a \emph{$G$-$\infty$-category} $\uAlg_{\cO}(\cC)$ whose $H$-value consists of algebras over the restricted \emph{$H$-operad}:
\[
  \uAlg_{\cO}(\cC)_H \simeq \Alg_{\Res_H^G \cO}(\Res_H^G \cC).
\]
\begin{example*}
  Let $\cC^{\otimes} \deq \ucS_G^{G-\times}$.
  Note that there is a natural equivalence
  \[
    \prn{\prod^S_K X_K}^H \simeq \prod_{[H/K] \in \Orb(S)} \prn{ \CoInd_K^HX_K}^H \simeq \prod_{[H/K] \in \Orb(S)} X^K_K,
  \]
  for each $S$-equivariant tuple $(X_K) \in \cS_S$, where $X^H = \Map^H(*,X)$ is the $H$-equivariant genuine fixed points functor.
  Thus we may compose \cref{action map} with genuine fixed points to acquire an action map
  \[
    \cO(S) \rightarrow \Map\prn{\prod_{[H/K] \in \Orb(S)} X^K, X^H};
  \]
  in particular, we may view $\cO([H/K])$ as the \emph{space of transfers $X^K \rightarrow X^H$} prescribed to an $\cO$-algebra.
  
  In particular, $\cN_{I \infty}^{\otimes}$ prescribes a contractible space of maps $\prod_{[H/K] \in \Orb(S)} X^K \rightarrow X^H$ for all $S \in \FF_H$ whose structure map $\Ind_H^G S \rightarrow [G/H]$ lies in $I$;
  indeed we will verify in forthcoming work \cite{Tensor} that $\cN_{I \infty}$-algebras in $\ucS_G^{G-\times}$ are (homotopy-coherent) incomplete $G$-commutative monoids.
\end{example*}

\subsection*{Summary of main results}
Write $\uSigma_G$ for the $G$-space core of the $G$-$\infty$-category of finite $G$-sets $\uFF_G$;
write $\tot\cln \Cat_G \rightarrow \Cat$ for the functor taking a $G$-$\infty$-category to the total $\infty$-category of its corresponding cocartesian fibration.
We identify objects with $\tot \uSigma_G$ with pairs $(H,S)$ where  $(H) \subset G$ is a conjugacy class and $S \in \FF_H$ is a finite $H$-set.
\begin{cooltheorem}\label{Sseq cool theorem}
  There exists a monadic functor
  \[
    \sseq\cln \Op_G^{\oc} \rightarrow \Fun(\tot \uSigma_G, \cS)
  \]
  whose composite functor $\Op_G \xrightarrow{\sseq} \Fun(\tot \uSigma_G, \cS) \xrightarrow{\ev_{(H,S)}} \cS$ recovers $\cO(S)$.
\end{cooltheorem}

In parallel, Bonventre-Pereira developed a model category $s\Op_G^{\oc}$ of \emph{genuine $G$-operads} which is right-transferred along a monadic \emph{underlying $G$-symmetric sequence} functor $U\cln s\Op_{G,*_G} \xrightarrow \Fun(\tot \uSigma_G, \sSet_{\mathrm{Quillen}})$ \cite[Thm~II]{Bonventre}.\footnote{When we say a model category $\cC$ is \emph{right-transferred along $F\colon \cC \rightarrow \cD$}, we mean that $F$ preserves and reflects weak equivalences and fibrations.}
We refer to the associated $\infty$-category as $g\Op_{G}^{\oc} \deq s\Op^{\oc}_{G}[\mathrm{weq}^{-1}]$.

Unwinding definitions, we will see that $\sseq$ is total right derived from a functor of 1-categories out of Nardin-Shah's model structure \cite{Nardin} which preserves and reflects weak equivalences between fibrant objects, and Bonventre's \emph{genuine operadic nerve} $N^{\otimes}$ satisfies $\cP(S) \simeq \prn{N^{\otimes} \cO}(S)$.
We conclude by two-out-of-three that $N^{\otimes}$ preserves and reflects weak equivalences between fibrant objects, yielding the following.
\begin{coolcorollary}\label{Nerve cool theorem}
  Bonventre's genuine operadic nerve possesses a conservative total right derived functor of $\infty$-categories.
\end{coolcorollary}
Moreover, in \cref{I operads subsection}, given a $G$-operad $\cO^{\otimes}$ we construct 
\emph{operadic composition maps}
\begin{equation}\label{Operadic composition equation}
  \gamma\cln \cO(S) \otimes \hspace{-5pt} \bigotimes_{[H/K_i] \in \Orb(S)} \hspace{-5pt} \cO(T_i) \rightarrow \cO\prn{\coprod_{[H/K_i] \in \Orb(S)} \Ind_{K_i}^H T_i},
\end{equation}
\emph{operadic restriction maps}
\begin{equation}\label{Operadic restriction equation}
  \Res\cln \cO(S) \rightarrow \cO\prn{\Res_K^H S},
\end{equation}
and \emph{equivariant symmetric group action}
\begin{equation}\label{Operadic symmetric action equation}
    \rho\cln \Aut_H(S) \times \cO(S) \rightarrow \cO(S)
\end{equation}
It is difficult to describe the coherences for these structures directly; nevertheless, in \cref{Discrete genuine nerve subsection}, we will use this structure to show that $N^{\otimes}$ restricts to an equivalence between the full subcategories of $G$-operads with discrete structure spaces.

Moving on, given $\cO^{\otimes}$ a $G$-operad, we define the \emph{arity support} subcategory\footnote{
  Throughout this paper, we say \emph{subobject} to mean monomorphism in the sense of \cite[\S~5.5.6]{HTT} and we write $\Sub_{\cC}(X)$ for the poset of subobjects of $X$ in $\cC$;
    in the case the ambient $\infty$-category is a 1-category, this agrees with the traditional notion.

    In the case our objects are in the $\infty$-category $\Cat$ of small $\infty$-categories, we call this a \emph{subcategory};
    in the case that the containing $\infty$-category is a 1-category, this is canonically expressed as a \emph{core-preserving wide subcategory of a full subcategory}, i.e. it is a \emph{replete subcategory}.
    Hence it is uniquely determined by its morphisms, so we will implicitly identify subcategories of $\cC$ a 1-category with their corresponding subsets of $\Mor(\cC)$.}
$A\cO \subset \FF_G$  by its maps
\[
  A\cO := \cbr{T \rightarrow S \; \middle| \;  \prod_{[H/K] \in \Orb(S)} \cO(T_K) \neq \emptyset} \subset \FF_G.
\]
where we once again use the shorthand $T_K \deq T \times_S [H/K]$.
In essence, $A \cO$ consists of the \emph{equivariant (multi-)arities} over which $\cO^{\otimes}$ prescribes structure on its algebras.

The fact that $\emptyset$ accepts no maps from nonempty spaces obstructs construction of maps matching \cref{Operadic composition equation,Operadic restriction equation}, so $A\cO$ can't be an arbitrary subcategory.
We use this to show the following.
\begin{cooltheorem}\label{Windex main theorem}
    The following posets are each equivalent:
    \begin{enumerate}[label={(\arabic*)}]
      \item \label[poset]{Subcommutative item}
            The poset $\Sub_{\Op_{G}}(\Comm_{G}) \subset \Op_G$ of sub-commutative $G$-operads. 
          \item \label[poset]{Truncated item}
            The poset $\Op_{G, 0} \subset \Op_G$ of $G$-$0$-operads.
          \item \label[poset]{Weak N ininifty item}
          The poset $\Op_G^{weak-\cN_\infty} \subset \Op_G$ of weak $\cN_\infty$ $G$-operads.
        \item \label[poset]{Image item}
            The essential image $A(\Op_{G}) \subset \Sub_{\Cat}(\FF_{G})$
          \item \label[poset]{Windex item}
            The embedded sub-poset $\wIndCat_{G} \subset \Sub_{\Cat}(\FF_{G})$ spanned by subcategories $I \subset \FF_{G}$ which are closed under base change and automorphisms and satisfy the Segal condition that
            \[
              T \rightarrow S \in I \hspace{30pt} \iff \hspace{30pt} \forall [G/H] \in \Orb(S), \;\;\; T \times_S [G/H] \rightarrow [G/H] \in I
            \]
          \item \label[poset]{Self-indexed coproducts item}
          The embedded sub-poset $\wIndSys_G \subset \mathrm{FullSub}_G(\uFF_{G})$ spanned by full $G$-subcategories $\cC \subset \uFF_{G}$ which are closed under self-indexed coproducts and have $*_H \in \cC_H$ whenever $\cC_H \neq \emptyset$.
    \end{enumerate}
    Furthermore, there is an equalities of sub-posets
    \[
      \IndCat_G = A\Op_{G, \geq \EE_\infty},
    \]
    where $\IndCat_G \simeq \IndSys_G$ denotes the \emph{indexing categories} of \cite{Blumberg-op,Bonventre,Rubin,Gutierrez}.
\end{cooltheorem}
\begin{proof}[References]
  In \cref{0-operads subterminal,Ninfty T-0-operad corollary} we show that \cref{Weak N ininifty item,Subcommutative item,Truncated item} are equal full subcategories of $\Op_G$.
    In \cref{Windex image proposition} we characterize the image of $A$, constructing equivalences between \cref{Image item,Windex item}.
  \cref{Weak N ininifty item,Image item} are shown to be equivalent in \cref{Windex image corollary} by realizing $\Op_G^{\mathrm{weak}-\cN_\infty}$ as the essential image of a fully faithful right adjoint $\cN_{(-)\infty}^{\otimes}$ to the essential surjection underlying $A$: 
    \begin{equation}\label{First ninfty adjunction equation}
         \begin{tikzcd}
        {\Op_G} & {\wIndCat_G}
        \arrow[""{name=0, anchor=center, inner sep=0}, "A", curve={height=-15pt}, from=1-1, to=1-2]
        \arrow[""{name=1, anchor=center, inner sep=0}, "{\cN_{I\infty}^{\otimes}}", curve={height=-15pt}, hook', from=1-2, to=1-1]
        \arrow["\dashv"{anchor=center, rotate=-90}, draw=none, from=0, to=1]
        \end{tikzcd}
    \end{equation}

    The equivalence between \cref{Windex item,Self-indexed coproducts item} is handled in \cite[Thm~A]{Windex};
    nevertheless, the composite map from \cref{Subcommutative item} to \cref{Self-indexed coproducts item} is shown to be furnished by the \emph{self-indexed symmetric monoidal envelope} in \cref{Windex windcat example}.
    Finally, the remaining identity follows by \cref{Property P observation}.
\end{proof}

Having done this, we move on to develop a notion of \emph{equivariant homotopy-coherent interchange} via the \emph{Boardman-Vogt tensor product}
\[
  \cO^{\otimes} \obv \cP^{\otimes} \deq L_{\Op_G}\prn{\cO^{\otimes} \times \cP^{\otimes} \rightarrow \Span(\FF_G) \times \Span(\FF_G) \xrightarrow{\;\;\; \wedge \;\;\;} \Span(\FF_G)}.
\]
where $L_{\Op_G}$ is as in \cref{LOpG equation}.
We verify many basic properties of this.
\begin{cooltheorem}\label{BV cool theorem}
  The bifunctor $\obv \cln \Op_G \times \Op_G \rightarrow \Op_G$ enjoys the following properties.
    \begin{enumerate}[ref={(\arabic*)}]
      \item In the case $G = e$ is the trivial group, $\obv$ is naturally equivalent to the Boardman-Vogt tensor product of \cite{HA,Hinich}.\label[statement]{Ope item}
        \item The functor $- \obv \cO\cln \Op_G \rightarrow \Op_G$ possesses a right adjoint $\uAlg^{\otimes}_{\cO}(-)$, whose underlying $G$-$\infty$-category is the $G$-$\infty$-category of algebras $\uAlg_{\cO}(-)$;
          the associated $\infty$-category is the $\infty$-category of algebras $\Alg_{\cO}(-)$.\label[statement]{Alg is adjoint item}
        \item The $\obv$-unit of $\Op_G$ is the $G$-operad $\triv_{G}^{\otimes}$ of \cite{Nardin};
          hence $\uAlg_{\triv_{G}}^{\otimes}(\cO) \simeq \cO^{\otimes}$.\label[statement]{Triv item}
        \item When $\cC^{\otimes}$ is a $G$-symmetric monoidal $\infty$-category, $\uAlg_{\cO}^{\otimes}(\cC)$ is a $G$-symmteric monoidal $\infty$-category; 
          furthermore, when $\cO^{\otimes} \rightarrow \cP^{\otimes}$ is a map of $G$-operads, the pullback lax $G$-symmetric monoidal functor
        \[
          \uAlg_{\cP}^{\otimes}(\cC) \rightarrow \uAlg_{\cO}^{\otimes} \prn{\cC}
        \]
        is $G$-symmetric monoidal;
        in particular, if $\cO^{\otimes}$ has one object, then pullback along the unique map $\triv^{\otimes}_G \rightarrow \cP^{\otimes}$ presents the unique natural transformation of operads
        \[
          \uAlg_{\cP}^{\otimes}(\cC) \rightarrow \cC^{\otimes},
        \]
        and this is $G$-symmetric monoidal when $\cC$ is $G$-symmetric monoidal. 
        \label[statement]{Symmetric monoidal pullback item}
        \item When $\cC^{\otimes} \rightarrow \cD^{\otimes}$ is a $G$-symmetric monoidal functor, the induced lax $G$-symmetric monoidal functor
        \[
            \uAlg_{\cO}^{\otimes}(\cC) \rightarrow \uAlg_{\cO}^{\otimes}(\cD)
        \]
        is $G$-symmetric monoidal.\label[statement]{Symmetric monoidal pushforward item}
      \item \label[statement]{Infl item}
        The adjunction $\Infl_e^G\cln \Op \rightleftarrows \Op_G\cln \Gamma^G$ enjoys the following (natural) equivalences:
        \begin{align*}
          \Infl_e^G \triv^{\otimes} &\simeq \triv_G^{\otimes};\\
          \Gamma^G \uAlg_{\Infl_e^{\cT} \cO}^{\otimes}(\cC) &\simeq \Alg^{\otimes}_{\cO}\prn{\Gamma^{G} \cC};\\
          \Infl_e^G(\cO) \obv \Infl_e^G(\cP) &\simeq \Infl_e^G(\cO \otimes \cP).
        \end{align*}
        Hence, writing $\EE_n$ for the little $n_G$-disks $G$-operad,\footnote{Here, $n_G$ is the $n$-dimensional trivial orthogonal $G$-representation.} the maps $\EE_{n},\EE_{m} \rightarrow \EE_{n+m}$ induce an equivalence
        \[
          \EE_{n}^{\otimes} \obv \EE_{m}^{\otimes} \xrightarrow{\;\;\; \sim \;\;\;} \EE_{n+m}
        \]
      \item The $G$-symmetric monoidal envelope of \cite{Nardin,Barkan} intertwines Day convolution with Boardman-Vogt tensor products, i.e. the following diagram commutes
    \[
      \begin{tikzcd}[column sep = large]
            \Op_G^2 \arrow[rrr,"\obv"] \arrow[d,"\Env^2"]
            &&& \Op_G \arrow[d,"\Env"]\\
            \prn{\Cat_G^{\otimes}}^2  \arrow[r,"\simeq" marking, phantom]
            & \Fun^{\times}(\Span(\FF_G),\Cat)^2 \arrow[r,"\circledast"]
            & \Fun^{\times}(\Span(\FF_G),\Cat)  \arrow[r,"\simeq" marking, phantom]
            & \Cat_G^{\otimes}
        \end{tikzcd}
      \]\label[statement]{Env item}
    \end{enumerate}
\end{cooltheorem}
\begin{proof}[References]
  \cref{Ope item} is \cref{Underlying tensor product}.
  \cref{Alg is adjoint item} is \cref{Alg values,Alg is adjoint prop,Alg underlying corollary}.
  \cref{Triv item} is \cref{Triv prop}.
  \cref{Symmetric monoidal pullback item,Symmetric monoidal pushforward item} are \cref{Symmetric monoidal push-pull corollary}.
  \cref{Infl item} is \cref{Infl BV,Gamma alg,Infl triv,En corollary}.
  \cref{Env item} is \cref{BV Env corollary}. 
\end{proof}

\subsection*{Notation and conventions}
We assume that the reader is familiar with the technology of higher category theory and higher algebra as developed in \cite{HTT} and \cite[\S~2-3]{HA}, though we encourage the reader to engage with such technologies via a ``big picture'' perspective akin to that of \cite[\S~1-2]{Gepner_HA} and \cite[\S~1-3]{Haugseng}. 
We will generally use the term \emph{replete subcategory inclusion} to refer to functors $F\colon \cC \rightarrow \cD$ whose core $F^{\simeq}\colon  \cC^{\simeq} \rightarrow \cD^{\simeq}$ is a summand inclusion and whose effect on mapping spaces $F\colon \Map(X,Y) \rightarrow \Map(FX,FY)$ is a summand inclusion for each $X,Y \in \cC$. 

\subsection*{Acknowledgements}
I would like to thank Jeremy Hahn for suggesting the problem of \emph{constructing equivariant multiplications on $\BPR$}, whose (ongoing) work necessitated many of the results on equivariant Boardman-Vogt tensor products developed in this paper;
Additionally, I would like to thank Clark Barwick, Dhilan Lahoti, Piotr Pstr\k{a}gowski, Maxime Ramzi, and Andy Senger, with whom I had many helpful conversations about equivariant homotopy theory and algebra.
Of course, none of this work would be possible without the help of my advisor, Mike Hopkins, who I'd like to thank for many helpful conversations.

While developing this material, the author was supported by the NSF Grant No. DGE 2140743.

 \resumetocwriting % resume toc writing 
\section{Equivariant symmetric monoidal categories}
In this section, we review and advance the equivariant $\infty$-category theory of \emph{homotopical incomplete (semi)-Mackey functors} for a weak indexing system $I$, which we call \emph{$I$-commutative monoids}.
To that end, we begin in \cref{T-categories subsection} by reviewing our equivariant higher categorical setup.
We go on to cite and prove some basic facts about $I$-commutative monoids in \cref{I-commutative monoids subsection}.
In \cref{Canonical SMC section} we then endow the $\cT$-$\infty$-category of $I$-commutative monoids with its \emph{mode} symmetric monoidal structure, and prove that this is uniquely determined as a presentable symmetric monoidal structure by the free functor from coefficient systems;
we use this to identify the resulting symmetric monoidal structure with the \emph{localized Day convolution structure}. 
Following this, in \cref{Truncations of SMCs subsubsection} we quickly develop a framework for $\cT$-symmetric monoidal $d$-categories. 

\subsection{Recollections on \texorpdfstring{$\cT$-$\infty$}{T-infinity}-categories}\label{T-categories subsection} 
We center on the following definition.
\begin{definition}
  An $\infty$-category $\cT$ is
  \begin{enumerate}
      \item \emph{orbital} if the finite coproduct completion $\FF_{\cT} := \cT^{\coprod}$ has all pullbacks, and
      \item \emph{atomic orbital} if it is orbital and every map in $\cT$ possessing a section is an equivalence.\qedhere
  \end{enumerate}
\end{definition}
We view the setting of atomic orbital $\infty$-categories as a natural axiomatic home for higher algebra centered around the Burnside category (see \cite[\S~4]{Nardin-Stable}), generalizing the orbit categories of a finite group.
The reader who is exclusively interested in equivariant homotopy theory is encouraged to assume every atomic orbital $\infty$-category is the orbit category of a family of subgroups of a finite group.
\begin{definition}
    Let $\cT$ be an $\infty$-category.
    Then, a full subcategory $\cF \subset \cT$ is a \emph{$\cT$-family} if whenever $V \in \cF$ and $W \rightarrow V$ is a map, we have $W \in \cF$.\footnote{These are named \emph{families} after subconjugacy closed families of subgroups, which frequently occur in equivariant homotopy; these are referred to as \emph{sieves} in \cite{Blumberg-op,Nardin} and \emph{upwards-closed subcategories} in \cite{Glasman}.}
    The poset of $\cT$-families under inclusion is denoted $\Fam_{\cT}$.

    Similarly, a full subcategory $\cF \subset \cT$ is a \emph{$\cT$-cofamily} if its opposite $\cF^{\op} \subset \cT^{\op}$ is a $\cT^{\op}$-family.
\end{definition}
Temporarily fix $G$ be a topological group, let $\cS_G$ be the $\infty$-category of $G$-spaces, and let $\cO_G \subset \cS_G$ be the full subcategory spanned by homogeneous $G$-spaces $[G/H]$, where $H \subset G$ is a closed subgroup.
\begin{example}
  The full subcategory $BG \subset \cO_G$ is a family, and the contractible full subcategory $\cbr{[G/G]} \hookrightarrow \cO_G$ is a cofamily.
  More generally, if $\cT$ is an $\infty$-category and $V \in \cT$ an object, then the full subcategory $\cT_{\geq V} \subset \cT$ consisting of objects admitting a map to $V$ is a family and the full subcategory $\cT_{\leq V} \subset \cT$ of objects admitting a map from $V$ is a cofamily.
\end{example}

\begin{example}
  The following are all atomic orbital $\infty$-categories (see \cite{Windex}).
  \begin{enumerate}
    \item The full subcategory $\cO_G^{fin} \subset \cO_G$ spanned by $[G/H]$ for $H$ finite.
    \item The wide subcategory $\cO_G^{f.i.} \subset \cO_G$ whose morphisms are projections $[G/K] \rightarrow [G/H]$ for $K \subset H$ finite index inclusion of closed subgroups.
    \item $X$ a space, considered as an $\infty$-category.
    \item $P$ a meet semilattice.
    \item If $\cT$ is an atomic orbital $\infty$-category, $\ho(T)$.
    \item If $\cT$ is an atomic orbital $\infty$-category, $\cF \subset \cT$ a full subcategory satisfying the following conditions:
      \begin{enumerate}[label={(\alph*)}]
        \item For all $U,W \in \cF$ and paths $U \rightarrow V \rightarrow W$ in $\cT$, $V \in \cF$.
        \item For all $U,W \in \cF$ and cospans $U \rightarrow V \leftarrow W$ in $\cT$, there is a span $U \leftarrow V' \rightarrow V$ in $\cF$.
      \end{enumerate}
      For instance, $\cF$ may be the intersection of a family and a cofamily whose connected components have weakly initial objects, such as $\cT_{\leq V}$ or $\cT_{\geq V}$.
    \item If $\cT$ is an atomic orbital $\infty$-category and $V \in \cT$, the $\infty$-category $\cT_{/V}$.\qedhere
  \end{enumerate}
\end{example}

In this section, we briefly summarize some relevant elements of parameterized and equivariant higher category theory in the setting of atomic orbital $\infty$-categories.
Of course, this theory has advanced far past that which is summarized here; for instance, further details can be found in the work of Barwick-Dotto-Glasman-Nardin-Shah \cite{Barwick_Intro, Barwick-Parameterized, Shah, Shah2, Nardin-Stable}, Cnossen-Lenz-Linskens \cite{Cnossen_stable, Cnossen_partial, Cnossen_semiadditive,Linskens,Linskens_stable}, Hilman \cite{Hilman}, and Martini-Wolf \cite{Martini_yoneda, Martini_cocompletions, Martini_straightening, Martini_presentable, Martini_topos}.

\subsubsection{The $\cT$-$\infty$-category of small $\cT$-$\infty$-categories}
We are motivated by the following.
\begin{example}
  Let $G$ be a finite group, $\cF \subset \cO_G$ a family, and $\cS_{\cF}$ be the $\infty$-category of $\cF$-spaces, constructed e.g. by inverting $\cF$-weak equivalences between topological $G$-spaces.
  Then, a version of Elmendorf's theorem \cite{Elmendorf} for families \cite[Thm~3.1]{Dwyer_Kan} states that the \emph{total $\cF$-fixed points} functor yields an equivalence
  \[
    \cS_{\cF} \simeq \Fun(\cF^{\op}, \cS).\qedhere
  \]
\end{example}
We extend this via the following definition.
\begin{definition}
  The \emph{$\infty$-category of small $\cT$-$\infty$-categories} is
  \[
    \Cat_{\cT} \deq \Fun(\cT^{\op},\Cat),
  \] 
  where $\Cat$ is the $\infty$-category of small $\infty$-categories.
  If $\widehat \Cat$ is the (very large) $\infty$-category of \emph{arbitrary} $\infty$-categories, then the \emph{very large $\infty$-category of $\cT$-$\infty$-categories} is
  \[
    \widehat \Cat_{\cT} \deq \Fun\prn{\cT^{\op}, \widehat \Cat}.\qedhere 
  \]
\end{definition}

\begin{notation}
    Fix $\cC \in \Cat_{\cT}$.
    We refer to the value of $\cC$ at $V \in \cT^{\op}$ as the \emph{$V$-value category of $\cC$}, written as $\cC_V$;
    given $f\colon V \rightarrow W$, we refer to the associated functor as \emph{restriction}
    \[
        \Res_V^W\cln \cC_W \rightarrow \cC_V.\qedhere
    \]
\end{notation}

\begin{remark}
  We show in \cref{Product pattern example} that $\Cat_{\cT}$ is equivalently presented as \emph{complete Segal objects} in the $\infty$-topos 
  \begin{equation}\label{T-spaces equation}
    \cS_{\cT} \deq \Fun(\cT^{\op},\cS).\qedhere
  \end{equation}
\end{remark}

\begin{remark}
    The Grothendieck construction, imported to $\infty$-category theory as the straightening unstraightening equivalence in \cite[Thm~3.2.0.1]{HTT}, produces an equivalence
    \[
        \Cat_{\cT} \simeq \Cat^{\cocart}_{/\cT^{\op}},
    \]
    the latter denoting the (non-full) subcategory of $\Cat_{/\cT^{\op}}$ whose objects are cocartesian fibrations and whose morphisms are functors over $\cT^{\op}$ which preserve cocartesian arrows.
    Under this identification, the fiber of $\Tot \cC \rightarrow \cT^{\op}$ over $V$ is identified with the $V$-value $\cC_V$ and the restriction functors are identified with cocartesian transport, where $\Tot$ denotes the total $\infty$-category of the unstraightening.
\end{remark}

Given $\cC,\cD$ a pair of $\cT$-$\infty$-categories, we may define the \emph{$\cT$-functor category} to be the full subcategory
\[
    \Fun_{\cT}(\cC,\cD) \deq \Fun^{\cocart}_{/\cT^{\op}}\prn{\cC,\cD} \subset  \Fun_{/\cT^{\op}}(\cC,\cD)
\]
consisting of functors over $\cT^{\op}$ which preserve cocartesian lifts of the structure maps.
\begin{example}
  For any object $V \in \cT$, the forgetful functor $\prn{\cT_{/V}}^{\op} \rightarrow \cT^{\op}$ is a cocartesian fibration classified by the representable presheaf $\Map_{\cT}(-,V)$.
    We refer to the associated $\cT$-$\infty$-category as $\uV$.
    This is covariantly functorial in $V$, since postcomposition yields functors $f_!\cln \cT_{/V} \rightarrow \cT_{/W}$ for all maps $f\cln V \rightarrow W$.
\end{example}
The representable $\cT$-categories are particularly nice in the atomic orbital setting.
\begin{proposition}[{\cite[Prop~2.5.1]{Nardin}}]\label{1-category}
    If an atomic orbital $\infty$-category $\cT$ has a terminal object, then it is a 1-category;
    in particular, $\cT_{/V}$ is a 1-category.\footnote{To see this, note that this is equivalent to the condition that the (split) diagonal map $U \rightarrow U \times U$ is an equivalence, which follows from the atomic assumption.}
\end{proposition}
\begin{remark}
  \cref{1-category} provides an easy verification that $\cO_G$ is not atomic orbital when $\dim G > 0$;
  $\cO_G$ has a terminal object $[G/G]$, but it is not a 1-category, as $\End([G/e]) \simeq G$ is not discrete.
\end{remark}
These play an important role in equivariant higher category theory.
\begin{notation}
  Given $\cC$ a $\cT$-$\infty$-category, we define the \emph{restricted} $\cT_{/V}$-category by 
  \[
    \Res_V^{\cT} \deq \cC_{\uV} \deq \cC \times_{\cT^{\op}} \prn{\cT_{/V}}^{\op}.\qedhere 
  \]
\end{notation}
\begin{proposition}[{\cite[Thm~9.7]{Barwick-Parameterized}}]
  $\Cat_{\cT}$ has exponential objects $\uFun_{\cT}(\cC,\cD)$ classified by the functor
  \[
    V \mapsto \Fun_{\cT_{/V}}\prn{\cC_{\uV}, \cD_{\uV}}.
  \]
\end{proposition}

We refer to monomorphisms\footnote{Following \cite[\S~5.5.6]{HTT}, we refer to a morphism $X \rightarrow Y$ in $\cC$ as a \emph{monomorphism} if the canonical map $X \rightarrow X \times_Y X$ is an equivalence, or equivalently, if the pullback functor $f^*\colon \cC_{/Y} \rightarrow \cC_{/Y}$ is fully faithful.} in $\Cat_{\cT}$ as \emph{$\cT$-subcategories}, and $\cT$-functors which are fiberwise-fully faithful as \emph{full $\cT$-subcategories}, or \emph{$\cT$-fully faithful functors}.
\begin{observation}\label{Subcategories are fiberwise observation}
  By the fiberwise expression for limits in functor categories (c.f. \cite[Cor~5.1.2.3]{HTT}), a $\cT$-functor $F\colon \cC \rightarrow \cD$ is a $\cT$-subcategory inclusion if and only if $F_V\colon \cC_V \rightarrow \cD_V$ is a subcategory inclusion for all $V \in \cT$.
\end{observation}

\begin{example}
  The terminal $\cT$-$\infty$-category $\uAst_{\cT}$ is classified by the constant functor $V \mapsto *$.
  The poset of \emph{sub-terminal objects in $\Cat_{\cT}$} (i.e. monomorphisms with codomain $\uAst_{\cT}$) is isomorphic to $\Fam_{\cT}$; the $\cT$-$\infty$-category $\uAst_{\cF}$ associated with $\cF$ is determined by the values
    \[
        \uAst_{\cF, V} \simeq \begin{cases}
            * & V \in \cF;\\
            \emptyset & \mathrm{otherwise}.
        \end{cases}
        \qedhere
    \]
    In fact, the ``$\infty$-groupoid'' inclusion $\cS \hookrightarrow \Cat$ induces an inclusion $\cS_{\cT} \hookrightarrow \Cat_{\cT}$ sending the universal space $E\cF$ to $\uAst_{\cF}$.
\end{example}

The $\infty$-category $\Cat_{\cT}$ participates in an adjunction
\[
  \Tot\cln \Cat_{\cT} \longrightleftarrows \Cat\cln \uCoFr^{\cT}
\]
whose left adjoint $\Tot$ is the total category of cocartesian fibrations, and whose right adjoint has $V$-value
\[
  (\uCoFr^{\cT} \cC)_V \simeq \Fun\prn{\prn{\cT_{/V}}^{\op},\cC}
\]
where the functoriality on $f$ is given by $\prn{f_!}^*$ \cite[Thm~7.8]{Barwick-Parameterized}.
We refer to $\uCoFr^{\cT}$ as the \emph{$\cT$-$\infty$-category of coefficient systems in $\cC$}.\footnote{These are referred to as the \emph{cofree} parameterization $\underline{\mathrm{CoFree}}(\cC)$ in \cite{Hilman} and as the \emph{$\cT$-$\infty$-category of $\cT$-objects} $\underline{\cC}_{\cT}$ in \cite{Nardin_thesis}.
  We avoid the former for clarity (as we do not view $\Tot$ as a forgetful functor), and we avoid the latter as it conflicts with the $\cT$-$\infty$-category of $\cT$-spectra $\uSp_{\cT}$;
  instead, our name is chosen to evoke the \emph{coefficient systems} used in equivariant cohomology.}
\begin{example}
  There is an equivalence $\uAst_{\cT} = \uCoFr^{\cT} * \in \Cat_{\cT}$ since right adjoints preserve terminal objects.
\end{example}
We may additionally construct the \emph{associated $\infty$-category}
\[
    \Gamma^{\cT} \cC \deq \Fun_{\cT}(\uAst,\cC),
\]
whose objects consist of cocartesian sections of the structure functor $\cC \rightarrow \cT^{\op}$.
We refer to this as the \emph{$\infty$-category of $\cT$-objects in $\cC$}.
For instance, if $\cT$ has a terminal object $V$, \cite[Lemma~2.12]{Barwick-Parameterized} shows that we have an equivalence
\[
  \Gamma^{\cT} \cC \simeq \cC_{V};
\]
more generally, this implies that $\Gamma^{\cT} \cC \simeq \lim_{V \in \cT^{\op}} \cC_V$, i.e. it is the \emph{$\cT$-fixed points} (or the limit of $\cC$ viewed as a $\cT^{\op}$ functor). 
Defining the \emph{$\cT$-inflation} to have $V$-values
\[
    \prn{\Infl_e^{\cT} \cD}_V \deq \cD
\]
for any $\cD \in \Cat$ and $V \in \cT$, the adjunction between limits and diagonals immediately yields the following.
\begin{proposition}\label{Gamma adjunction proposition}
  The functor $\Infl_e^{\cT}\cln \Cat \rightarrow \Cat_{\cT}$ is left adjoint to $\Gamma^{\cT}\cln \Cat_{\cT} \rightarrow \Cat$.
\end{proposition}

Using this adjunction, given $\cC \in \Cat$, we define the $\infty$-category
\[
  \CoFr^{\cT} \cC \deq \Gamma^{\cT} \uCoFr^{\cT} \cC \simeq \Fun(\cT^{\op},\cC);
\]
then, we have $\Cat_{\cT} = \CoFr^{\cT} \Cat$, and Elmendorf's theorem states that $\cS_G \simeq \CoFr^{\cO_G} \cS$, motivating the following.

\begin{definition}
    The \emph{$\cT$-$\infty$-category of small $\cT$-$\infty$-categories} is $\uCat_{\cT} \deq \uCoFr^{\cT}(\Cat)$;
    the $\cT$-$\infty$-category of $\cT$-spaces is $\ucS_{\cT} \deq \uCoFr^{\cT}(\cT)$, and the $\infty$-category of $\cT$-spaces is $\cS_{\cT} \deq \CoFr^{\cT}(\cS) \simeq \Gamma^{\cT} \ucS_{\cT}$.
\end{definition}

\begin{observation}
    The \emph{$V$-value of $\uCat_{\cT}$} is $\prn{\uCat_{\cT}}_V = \Cat_{\cT_{/V}}$;
    we henceforth refer to this as $\Cat_V$.
    The restriction functor $\Res_V^W\colon \Cat_W \rightarrow \Cat_V$ is presented from the perspective of cocartesian fibrations by the pullback
    \[
        \begin{tikzcd}
            \Res^V_W \cC \arrow[r] \arrow[d] \arrow[dr,"\lrcorner" very near start, phantom]
            & \cC \arrow[d]\\
            \prn{\cT_{/V}}^{\op} \arrow[r] 
            & \prn{\cT_{/W}}^{\op}
        \end{tikzcd}
    \]
    In particular, given a map $U \rightarrow V \rightarrow W$, abusively referring to $(U \rightarrow V) \in \cT_{/V}$ as $U$, this is characterized by the formula
    \[
      \prn{\Res_W^V \cC}_U \simeq \cC_U.\qedhere
    \]
\end{observation}

\subsubsection{Join, slice, and (co)limits} 
We now summarize some elements of \cite{Shah,Shah2}.
\begin{definition}[{\cite[Def~4.1]{Shah2}}]
  Let $\iota\cln \cT^{\op} \times \partial \Delta^1 \hookrightarrow \cT^{\op} \times \Delta^1$ be the evident inclusion.
  Then, the \emph{$\cT$-join} is the top horizontal functor
  \[
    \begin{tikzcd}
      \Cat_{\cT}^2 \arrow[rr,"- \star_{\cT} -"] \arrow[d,hook]
      && \Cat_{\cT} \arrow[d,hook]\\
      \Cat_{/\cT^{\op} \times \partial \Delta^1} \arrow[r,"\iota^*"]
      & \Cat_{/\cT^{\op} \times I} \arrow[r,"\pi_!"]
      & \Cat_{/\cT^{\op}}
    \end{tikzcd}
  \]
  which exists by \cite[Prop~4.3]{Shah}.
  We write 
  \begin{align*}
    K^{\utright} &\deq K \star_{\cT} \uAst_{\cT};\\
    K^{\utleft} &\deq \uAst_{\cT} \star_{\cT} K.\qedhere
  \end{align*}
\end{definition}
\begin{definition}
  If $\cC,\cD \in \Cat_{\cT, \cE/}$ are $\cT$-$\infty$-categories under $\cE$, the \emph{$\cT$-$\infty$-category of $\cT$-functors under $\cE$} is defined by the pullback of $\cT$-$\infty$-categories
\[\begin{tikzcd}
	{\uFun_{\cT, \cE/}(\cC,\cD)} & {\uFun_{\cT}(\cC,\cD)} \\
  \uAst_{\cT} & {\uFun_{\cT}(\cE,\cD)}
	\arrow[from=1-1, to=1-2]
	\arrow[from=1-1, to=2-1]
	\arrow["\lrcorner"{anchor=center, pos=0.125}, draw=none, from=1-1, to=2-2]
	\arrow["{\prn{\pi_{\cC}}^*}", from=1-2, to=2-2]
  \arrow["{\cbr{\pi_{\cD}}}"', from=2-1, to=2-2]
\end{tikzcd}\]
  If $p\colon K \rightarrow \cC$ is a $\cT$-functor, then the \emph{$\cT$-undercategory and $\cT$-overcategory} are the functor $\cT$-$\infty$-categories
  \begin{align*}
    \cC^{(p,\cT)/} &\deq \uFun_{\cT, K/}\prn{K^{\utright},\cC};\\
    \cC^{/(p,\cT)} &\deq \uFun_{\cT, K/}\prn{K^{\utleft},\cC} \qedhere
  \end{align*}
\end{definition}
In the case $p\colon \uAst_{\cT} \rightarrow \cC$ corresponds with the $\cT$-object $X \in \Gamma^{\cT} \cC$, we simply write $\cC^{X/} \deq \cC^{(p,\cT)/}$ and similar for overcategories.
In general, the categories $\cC^{(p,\cT)/}$ take part in a functor out of $\Cat_{\cT,K/}$.
Of fundamental importance is the adjoint relationship between these functors:
\begin{theorem}[{\cite[Cor~4.27]{Shah2}}]
  The $\cT$-join forms the left adjoint in a pair of adjunctions
  \begin{align*}
    K \star_\cT -\cln \Cat_{\cT} &\longrightleftarrows \Cat_{\cT, K/} \cln(-)^{(-,\cT)/},\\
    - \star_\cT K\cln \Cat_{\cT} &\longrightleftarrows \Cat_{\cT, K/} \cln(-)^{/(-,\cT)}.
  \end{align*}
\end{theorem}

We say a $\cT$-functor $\underline{p}\colon K^{\utl} \rightarrow \cC$ extends $p\colon K \rightarrow \cC$ if the composite $K \rightarrow K^{\utl} \rightarrow \cC$ is homotopic to $p$.
\begin{definition}
  Let $\cC$ be a $\cT$-$\infty$-category.
  A $\cT$-object $X \in \Gamma^{\cT} \cC$ is \emph{final} if for all $V \in \cT$, the object $X_V \in \cC_{V}$ is final;
  a $\cT$-functor $\underline{p}\colon K^{\utl} \rightarrow \cC$ extending $p\colon K \rightarrow \cC$ is a \emph{limit diagram for $p$} if the corresponding cocartesian section $\sigma_{\underline{p}}\colon *_{\cT} \rightarrow \cC^{/(p,\cT)}$ is a final $\cT$-object.
\end{definition}

The \emph{fiberwise opposite} (or vertical opposite) functor $\op\cln \Cat_{\cT} \rightarrow \Cat_{\cT}$ is the $\cT$ functor induced under $\CoFr^{\cT}$ by the \emph{opposite category} functor $\op\cln \Cat \rightarrow \Cat$;
the notions of initial $\cT$-objects and $\cT$-colimits are defined dually as final $\cT$-objects and $\cT$-limits in the fiberwise opposite.

In many cases, these are familiar;
for instance, \emph{trivially indexed} (co)limits are non-equivariant in nature.
\begin{proposition}[{\cite[Thm~8.6]{Shah}}]\label{Trivially indexed limits}
  A diagram $\underline{p}\colon \prn{\Infl_e^{\cT} K}^{\utl} \rightarrow \cC$ is a limit diagram for $p\colon \Infl_e^{\cT} K \rightarrow \cC$ if and only if for all $V$, the associated diagram $\underline{p}_V\colon K^{\triangleleft} \rightarrow \cC_V$ is a limit diagram for $p_V$.
\end{proposition}

Similarly, indexed (co)limits in coefficient systems may be converted into non-equivariant colimits.
\begin{proposition}[{\cite[Prop~5.6-7]{Shah2}}]\label{Fixed points of colimit}
  Let $\cT$ be an atomic orbital $\infty$-category and $F\colon \cC \rightarrow \uCoFr^{\cT}(\cD)$ a $\cT$-functor.
  Then, the indexed limit and colimit of $F$ have values computed by ordinary limits and colimits:
  \begin{align*}
    \prn{\ucolim F}^V &\simeq \colim\prn{\cC_V \rightarrow \Tot^V \cC \rightarrow \CoFr^V(\cD) \xrightarrow{(-)^V} \cD};\\
    \prn{\ulim F}^V &\simeq \lim\prn{\Tot^V \cC \rightarrow \CoFr^V(\cD) \xrightarrow{(-)^V} \cD}.
  \end{align*}
\end{proposition}
\begin{definition}
    Let $\cC$ be a $\cT$-$\infty$-category and let $\ucK_{\cT} = (\cK_V)_{V \in \cT} \subset \uCat_{\cT}$ be a restriction-stable collection of $V$-categories.
    We say that $\cC$ \emph{strongly admits $\cK$-shaped limits} if for each $V \in \cT$, each $V$-category $K \in \cK_V$ and each $V$-functor $p\colon K \rightarrow \cC_{\uV}$, there exists a limit diagram for $p$.
    We say $\cC$ is \emph{$\cT$-complete} if it strongly admits $\uCat_{\cT}$-shaped limits.
  
    If $\cC$ and $\cD$ are $\cT$-$\infty$-categories which strongly admit all $\cK$-shaped limits and $F\colon \cC \rightarrow \cD$ is a $\cT$, functor, we say $F$ \emph{strongly preserves $\cK$-shaped limits} if for all $V \in \cT$ and all $K \in \cK_V$, postcomposition with the $V$-functor $F_{\uV}\colon \cC_{\uV} \rightarrow \cD_{\uV}$ sends $K$-shaped limits diagrams to limits diagrams.

    If $\cC \subset \cD$ is a full $\cT$-subcategory whose inclusion strongly preserves $\cK$-shaped limits, we say that $\cC$ is \emph{strongly closed under $\cK$-shaped limits}.
\end{definition}

An important class of examples is \emph{indexed (co)products}.
\begin{definition}\label{Indexed coproducts notation}
  Consider $S \in \FF_V$, considered as a $V$-category under the inclusion $\Set_V \hookrightarrow \Cat_V$ extending the \emph{representable $V$-category} functor $\cT_{/V} \rightarrow \Cat_V$ via coproducts.
  Then, we refer to $S$-shaped $V$-limits as \emph{$S$-indexed products} and $S$-shaped $V$-colimits as \emph{$S$-indexed coproducts}.

  If $\cC \subset \uFF_{\cT}$ is a full $\cT$-subcategory, we refer to $\cT$-colimits of the corresponding class as \emph{$\cC$-indexed coproducts};
  similarly, following \cite{Windex}, if $I \subset \Set_{\cT}$ is a pullback-stable and core-full subcategory, we define the full $\cT$-subcategory $\uSet_I \subset \uSet_{\cT}$ of \emph{$I$-admissible $\cT$-sets} by 
  \[
    \prn{\uSet_{I}}_V \deq \Set_{I,V} \deq \cbr{S \; \middle| \; \Ind_V^{\cT} S \rightarrow V \in I} \subset \Set_{V}.
  \]
  We refer to the class of $\uSet_{I}$-indexed coproducts as \emph{$I$-indexed coproducts}, and use the dual language for $I$-indexed products.
  If $\cD$ strongly admits $\uSet_I$-shaped limits, we simply say $\cD$ \emph{admits $I$-indexed coproducts};
  we use the following language.
  \begin{itemize}
    \item $\Set_{\cT}$-indexed coproducts are \emph{small indexed coproducts};
    \item $\FF_{\cT}$-indexed coproducts are \emph{finite indexed coproducts};
    \item $\cbr{\nabla\colon n \cdot S \rightarrow S}$-indexed coproducts are \emph{trivially indexed coproducts} (or \emph{ordinary coproducts}).\qedhere
  \end{itemize}
\end{definition}
\begin{notation}
  Given $\cC$ a $\cT$-category and $S \in \Set_{\cT}$, we write 
  \begin{align*}
    \cC_S &\deq \prod_{U \in \Orb(S)} \cC_U\\
        &\simeq \Fun_{\cT}(S,\cC);
  \end{align*}
  more generally, given $S \in \Set_V$, we write $\cC_S$ for $\cC_{\Ind_V^{\cT} S}$.
  where $\Orb(S)$ is the set of \emph{orbits} expressing $S$ as a disjoint union of elements of $\cT$.
  Given $S \in \Set_{I,V}$, and $(X_U) \in \cC_S$, we write the $S$-indexed products and coproducts as
  \[
    \begin{tikzcd}[ampersand replacement=\&, row sep=tiny, column sep=large]
      {\cC_S} \& {\cC_V} \&\& {\cC_S} \& {\cC_V} \\
      {(X_U)_{U \in \Orb(S)}} \& {\prod\limits_U^S X_U} \&\& {(X_U)_{U \in \Orb(S)}} \& {\coprod\limits_U^S X_U}
      \arrow["{\prod^S}", from=1-1, to=1-2]
      \arrow["{\coprod^S}", from=1-4, to=1-5]
      \arrow["\in"{marking, allow upside down}, draw=none, from=2-1, to=1-1]
      \arrow[maps to, from=2-1, to=2-2]
      \arrow["\in"{marking, allow upside down}, draw=none, from=2-2, to=1-2]
      \arrow["\in"{marking, allow upside down}, draw=none, from=2-4, to=1-4]
      \arrow[maps to, from=2-4, to=2-5]
      \arrow["\in"{marking, allow upside down}, draw=none, from=2-5, to=1-5]
    \end{tikzcd}
  \]
  In particular, in the case that $S$ has one orbit $U$, we write $\Ind_U^V(-)$ and $\CoInd_U^V(-)$ for $S$-indexed coproducts and products, respectively. 
\end{notation}

Given $\cK \subset \uCat_{\cT}$ a restriction-stable collection of $V$-categories and $W \in \cT$, we let $\cK_{\uW} \subset \uCat_W$ be the corresponding restriction-stable collection $V$-categories, where $V$ ranges over $\cT_{/W}$.
We will use the following notation for strongly (co)limit-presereving functors.
\begin{notation}
    Let $I \subset \FF_{\cT}$ be a pullback-stable subcategory.
    Following and slightly extending \cite[Notn~1.15]{Shah}, we use the following notation for the described distinguished full $\cT$-subcategories of $\uFun_{\cT}(\cC,\cD)$:
    \begin{enumerate}
        \item $\uFun_{\cT}^{\cK-L}(\cC,\cD)$: the $V$-functors which strongly preserve $\cK_{\uV}$-indexed colimits;
        \item $\uFun_{\cT}^{\cK-R}(\cC,\cD)$: the $V$-functors which strongly preserve $\ucK_{\uV}$-indexed limits;
         \item $\uFun_{\cT}^L(\cC,\cD)$: the $V$-functors which strongly preserve small $V$-colimits;
        \item $\uFun_{\cT}^{R}(\cC,\cD)$: the $V$-functors which strongly preserve small $V$-limits;
        \item $\uFun_{\cT}^{I-\sqcup}(\cC,\cD)$: the $V$-functors which (strongly) preserve $I$-indexed coproducts;
        \item $\uFun_{\cT}^{I-\times}(\cC,\cD)$: the $V$-functors which (strongly) preserve $I$-indexed products.
        \item $\uFun_{\cT}^{\sqcup}(\cC,\cD)$: the $V$-functors which (strongly) preserve finite ordinary coproducts;
        \item $\uFun_{\cT}^{\times}(\cC,\cD)$: the $V$-functors which (strongly) preserve finite ordinary products.\qedhere
    \end{enumerate}
\end{notation}

\subsubsection{Parameterized Kan extensions}
Fix a (not necessarily commuting) triangle of $\cT$-functors
\[
  \begin{tikzcd}[ampersand replacement=\&]
	\cC \& \cE \\
	\cD
	\arrow["F", from=1-1, to=1-2]
	\arrow["\varphi"', from=1-1, to=2-1]
	\arrow["G"', dashed, from=2-1, to=1-2]
\end{tikzcd}
\]
and $x \in \cD_V$ a $V$-object.
Assume $\cD$ has a final $V$-object.
We define the composite $V$-functor
\[
  G^x\colon \prn{\cC_{\uV}^{/x}}^{\utr} \xrightarrow{\;\;\; \varphi \;\;\;} \prn{\cD_{\uV}^{/x}}^{\utr} \xrightarrow{\prn{H',\pi}} \cD_{\uV}^{/x} \times \Delta^1 \xrightarrow{\;\;\; H \;\;\;} \cD_{\uV} \xrightarrow{\;\;\; G \;\;\;}  \cE_{\uV}
\]
where $H'$ takes the cone point to a $\uV$-final object and $H$ is adjunct to the evident map $\cD_{\uV}^{/x} \rightarrow \mathrm{Ar}(\cD_{\uV})$.
\begin{theoremdefinition}[{\cite[Thm~2.13]{Shah}}]
  The pullback $\cT$-functor
  \[
    \varphi^*\colon \uFun_{\cT}(\cD,\cE) \rightarrow \uFun_{\cT}(\cC,\cE)
  \]
  has a partially defined left adjoint $\varphi_!$ whose values are uniquely characterized by the property that $\prn{\varphi_! F}^x$ is a $V$-colimit diagram for all $V \in \cT$ and $x \in \cD_V$. We call this the \emph{$\cT$-left Kan extension of $F$ along $\varphi$.}
\end{theoremdefinition}
For instance, $\cT$-left Kan extensions along $\cC \rightarrow *_{\cT}$ are precisely $\cT$-colimits.
More generally, we will view the above colimit formula via the shorthand
\[
  \phi_! F(y) = \ucolim_{\varphi(x) \rightarrow y} F(x).
\]
Of course, $\cT$-right Kan extensions are defined dually, and denoted $\varphi_*$.

\subsubsection{Parameterized adjunctions}%
Related to indexed colimits, there is a theory of \emph{parameterized adjunctions} 

\begin{definition}
    A $\cT$-functor $L\cln \cC \rightarrow \cD$ is \emph{left adjoint} to $R\cln \cD \rightarrow \cC$ if the associated functors $L_V\cln \cC_V \rightarrow \cD_V$ are left adjoint to $R_V\cln \cD_V \rightarrow \cC_V$ for all $V \in \cT$.
\end{definition}
These are the same as \emph{relative adjunctions} over $\cT^{\op}$ by \cite[Prop~7.3.2.1]{HA};
$\cT$-left adjoints strongly preserve small $\cT$-colimits and $\cT$-right adjoints strongly preserve small $\cT$-limits \cite[Thm~3.1.10]{Hilman}, and they satisfy a parameterized version of the adjoint functor theorem \cite[Thm~6.2.1]{Hilman}. 
\begin{remark}
  By \cite[Rmk~5.4]{Shah}, $\cT$-limits form a (partially defined) right $\cT$-adjoint $\ulim\colon \uFun_{\cT}(K,\cC) \rightarrow \cC$ to the ``diagonal'' $\cT$-functor $\Delta^K\colon \cC \rightarrow \uFun_{\cT}(K,\cC)$, which itself may be computed as precomposition along the canonical $\cT$-functor $K \rightarrow *_{\cT}$.
\end{remark}
As observed in \cite{Windex}, diagonals are functorial, so composing right adjoints to the diagonal of the ``orbit set'' factorization $\Ind_V^{\cT} S \rightarrow \coprod_{U \in \Orb(S)} V \rightarrow V$ 
thus yields natural equivalences
\begin{equation}\label{Coproducts and ind observation}
  \coprod\limits_U^S X_U \simeq \coprod\limits_{U \in \Orb(S)} \Ind_U^V X_U;
  \hspace{50pt}
  \prod\limits_U^S X_U \simeq \prod\limits_{U \in \Orb(S)} \CoInd_U^V X_U.
\end{equation}

We may construct many more $\cT$-adjunctions using $\uCoFr^{\cT}$:
\begin{lemma}\label{Cofree adjunction}
    Suppose $L\cln \cC \rightleftarrows \cD\cln R$ is an adjunction of $\infty$-categories.
    Then,
    \[
        \uCoFr^{\cT}L\cln \uCoFr^{\cT} \cC \rightleftarrows \uCoFr^{\cT} \cD\cln \uCoFr^{\cT} R
    \]
    is an adjunction of $\cT$-$\infty$-categories.
\end{lemma}
\begin{proof}
  This follows from the fiberwise description of $\uCoFr^{\cT} (-)$;
    indeed, the $V$-values
    \[
      L_*:\Fun\prn{(\cT_{/V})^{\op},\cC} \rightleftarrows \Fun\prn{(\cT_{/V})^{\op},\cD}:R_*
    \]
    are adjoint.
\end{proof}
\begin{example}
    We may use \cref{Cofree adjunction} to realize the full $\cT$-subcategory of $\cT$-spaces whose fixed points are $d$-connected or $d$-truncated as (co)localizing $\cT$-subcategories
    \[
      \begin{tikzcd}[ampersand replacement=\&]
        {\ucS_{\cT,\geq d}} \& {\ucS_{\cT}} \& {\ucS_{\cT, \leq d}}
        \arrow[""{name=0, anchor=center, inner sep=0}, curve={height=-16pt}, hook', from=1-1, to=1-2]
        \arrow[""{name=1, anchor=center, inner sep=0}, curve={height=-16pt}, from=1-2, to=1-1]
        \arrow[""{name=2, anchor=center, inner sep=0}, curve={height=-16pt}, from=1-2, to=1-3]
        \arrow[""{name=3, anchor=center, inner sep=0}, curve={height=-16pt}, hook', from=1-3, to=1-2]
        \arrow["\dashv"{anchor=center, rotate=-90}, draw=none, from=0, to=1]
        \arrow["\dashv"{anchor=center, rotate=-90}, draw=none, from=2, to=3]
      \end{tikzcd}
    \]
    We will use this line of thought to understand \emph{truncatedness and connectedness of $\cT$-operads and $\cT$-symmetric monoidal categories}.
\end{example}

\begin{example}\label{Core example}
    By \cref{Cofree adjunction}, the \emph{classifying space and core} double adjunction $(-)_{\simeq} \dashv \iota \dashv (-)^{\simeq}$ yields a double $\cT$-adjunction 
    \[
      \begin{tikzcd}[ampersand replacement=\&]
        {\uCat_{\cT}} \&\& {\ucS_{\cT}} \& {}
        \arrow[""{name=0, anchor=center, inner sep=0}, "{{(-)_{\simeq}}}", curve={height=-18pt}, from=1-1, to=1-3]
        \arrow[""{name=1, anchor=center, inner sep=0}, "{{(-)^{\simeq}}}"{description}, curve={height=18pt}, from=1-1, to=1-3]
        \arrow[""{name=2, anchor=center, inner sep=0}, hook', from=1-3, to=1-1]
        \arrow["\dashv"{anchor=center, rotate=-90}, draw=none, from=0, to=2]
        \arrow["\dashv"{anchor=center, rotate=-90}, draw=none, from=2, to=1]
      \end{tikzcd}
    \]
\end{example}

Additionally, we can make genuine adjunction \emph{non-genuine} using \cite[Prop~7.3.2.1]{HA}.
\begin{proposition}\label{Gamma adjunction prop}
    If $L\cln \cC \rightleftarrows \cD\cln R$ are adjoint $\cT$-functors, then $\tot L\cln \tot \cC \rightleftarrows \tot \cD\cln \tot R$ and $\Gamma L\cln \Gamma \cC \rightleftarrows \Gamma \cD\cln \Gamma R$ are adjoint pairs.
\end{proposition}
\begin{proof}
  The adjunction on $\tot$ is \cite[Prop~7.3.2.1]{HA}, and it induces an adjunction
  \[
    \tot L_*\cln \Fun_{/\cT}(\cT,\tot \cC) \longrightleftarrows \Fun_{/\cT}(\cT,\tot \cD)\cln \tot R_*,
  \]
  which restricts to the full subcategories of cocartesian sections, and hence yields an adjunction
  \[
    \Gamma^{\cT} L\cln \Gamma^{\cT} \cC \longrightleftarrows \Gamma^{\cT} \cD\cln \Gamma^{\cT}.\qedhere.
  \]
\end{proof}

We will need the following lemma and proposition later.
\begin{lemma}\label{Gamma of conservative}
  Suppose a $\cT$-functor $F\cln \cC \rightarrow \cD$ has $F_V:\cC_V \rightarrow \cD_V$ conservative for all $V \in \cT$;
  then, $\Gamma^{\cT} F$ is conservative.
\end{lemma}
\begin{proof}
  Suppose $f_\bullet:X_\bullet \rightarrow Y_\bullet$ is a map of $\cT$-objects in $\cC$, i.e. a natural transformation of cocartesian sections of $\tot \cC \rightarrow \cT^{\op}$.
  Then, $f_\bullet$ is an equivalence if and only if $f_V$ is an equivalence for each $V$;
  by conservativity of $F_V$, this is true if and only if $F_v f_v$ is an equivalence for each $V$, i.e. if and only if $F f_\bullet$ is an equivalence, so $\Gamma^{\cT} F$ is conservative.
\end{proof}

\begin{proposition}\label{Gamma of monadic}
  Suppose $L\cln \cC \rightleftarrows \cD\cln R$ is a $\cT$-adjunction such that $R_V$ is monadic for all $V \in \cT$;
  Then, $\Gamma^{\cT} R\cln \Gamma^{\cT} \cD \rightarrow \Gamma^{\cT} \cC$ is monadic.
\end{proposition}
\begin{proof}
  We verify that $\Gamma^{\cT} R$ satisfies the conditions of the $\infty$-categorical Barr-Beck theorem \cite[Thm~4.7.3.5(c)]{HA}.
  First, by \cref{Gamma adjunction prop,Gamma of conservative}, $\Gamma^{\cT} R$ is a conservative right adjoint.
  Second, note that a simplicial object $Z_{\bullet}(-)$ in $\Gamma^{\cT} \cD$ corresponds to a family of simplicial objects $Z_V(-)$ in $\cD_V$, and a $\Gamma^{\cT} R$-splitting of $Z_{\bullet}(-)$ corresponds with a restriction-stable family of $R_V$-splittings of $Z_V(-)$.
  Thus $R_V$ creates a colimit of $Z_V$ for all $V$, and the resulting cocartesian section creates a colimit for $Z_\bullet$, i.e. $\Gamma^{\cT} R$ createss $\Gamma^{\cT}R$-split simplicial colimits, so $\Gamma^{\cT}R$ is monadic by \cite[Thm~4.7.3.5(c)]{HA}.
\end{proof}

\subsubsection{Language in the case $\cT = \cO_G$}
When $G$ is a finite group, the category $\cO_G$ has objects the homogeneous $G$-sets $[G/H]$ and morphisms the $G$-equivariant maps $[G/K] \rightarrow [G/H$];
tracking the image of the identity, the hom set from $[G/K]$ to $[G/H]$ may alternatively be presented as
\[
  \Hom([G/K],[G/H]) \simeq \frac{\cbr{a \in G \mid aKa^{-1} \subset H}}{a \sim b \;\; \text{ when } \;\; ab^{-1} \in K}
\]
(see e.g. \cite[Prop~1.3.1]{Dieck} for details).
In particular, the endomorphism monoid of $[G/K]$ is the Weyl group $W_G H = N_G(H)/H$.
Using this, one may see that when $G$ is a finite group, the map $\Ind_H^G\cln \cO_H \rightarrow \cO_{G,/(G/H)}$ is an equivalence of categories.
Thus we may set the following notation without creating clashes.
\begin{notation}    
  In the setting that $\cT = \cO_G$, we use the following notation:
  \begin{enumerate}
    \item we refer to $\underline{\brk{G/H}}$ as $\underline{H}$;
    \item we refer to $\cO_G$-$\infty$-categories as $G$-$\infty$-categories and $\uCat_{\cO_G}$ as $\uCat_G$;
    we refer to $\cO_G$-spaces as $G$-spaces and $\ucS_{\cO_G}$ as $\ucS_G$;
    \item we refer to $\cC_{[G/H]}$ as $\cC_H$ and $\Res_{[G/K]}^{[G/H]}$ as $\Res_K^H$;
      the superscripts and subscripts of $\Ind$, $\CoInd$, $\Gamma$, $\uCoFr$, $\star$, $(-)^{(-,\cT)/}$, and $*$ are determined similarly.
    \item we refer to $\coprod_{[H/K]}^S X_K$ as $\coprod_K^S X_K$, and similar for $\prod^S$.\qedhere
  \end{enumerate}
\end{notation}

\subsection{\tI-commutative monoids}\label{I-commutative monoids subsection}
 Following \cite{Barwick1}, we say that an \emph{adequite triple} is the data of two core-preserving wide subcategories $\cX_b \subset \cX \supset \cX_f$ of an $\infty$-category such that cospans $X \xrightarrow{\varphi_f} Y \xleftarrow{\varphi_b}Z$  satisfying $\varphi_f \in \cX_f$ and $\varphi_b \in \cX_b$ lift to pullback diagrams
    \[
      \begin{tikzcd}[sep=small, row sep = tiny]
            & X \times_Y Z \arrow[dl,"\psi_b" above] \arrow[dr,"\psi_f"] \arrow[dd,phantom,"\lrcorner" twlabl]\\
            X \arrow[dr,"\varphi_f" below] 
            && Z \arrow[dl,"\varphi_b"]\\
            & Y
        \end{tikzcd}
    \]
    satisfying $\psi_b \in \cX_b$ and $\psi_f \in \cX_f$.
    Given an adequate triple $\cX_b \subset \cX \supset \cX_f$, we define the \emph{span category} to be
    \[
        \Span_{b,f}(\cX) \deq A^{eff}(\cX,\cX_b,\cX_f),
    \]
    the latter denoting the effective Burnside category of \cite{Barwick1}.
    In particular, the objects of $\Span_{b,f}(\cX)$ are precisely those of $\cX$, and the morphisms from $X$ to $Z$ are the spans $X \xleftarrow{\varphi_b} Y \xrightarrow{\varphi_f} Z$ with $\varphi_b \in \cX_b$ and $\varphi_f \in \cX_f$, with composition defined by taking pullbacks.
    \footnote{Those readers more familiar with \cite{Elmanto} may note that this specializes to the notion of a \emph{span pair}, when backwards maps are $\cX_b = \cX$, in which case $\Span_f(\cX)$ recovers that of \cite{Elmanto}, and hence lifts to an $(\infty,2)$-category with a universal property that we will not use.}
\begin{example}
  For $\cT$ an orbital $\infty$-category and $I \subset \FF_{\cT}$ a pullback-stable wide subcategory with $I^{\simeq} \simeq \FF_{\cT}^{\simeq}$, $\FF_{\cT} = \FF_{\cT} \hookleftarrow I$ is an adequate triple;
write
\[
    \Span_I(\FF_{\cT}) := \Span_{all,I}(\FF_{\cT}).
\]
More generally, if there exists some full subcategory $\cC \subset \FF_{\cT}$ such that $I \subset \cC$ is a pullback-stable wide subcategory with $I^{\simeq} \subset \cC^{\simeq}$, we write
\[
    \Span_I(\FF_{\cT}) := \Span_{all,I}(\cC).\qedhere
\]
\end{example}
\begin{warning}
  Even when $\FF_{\cT}$ is a 1-category (i.e. $\cT$ is a 1-category), $\Span_I(\FF_{\cT})$ will seldom be a 1-category;
  indeed, in this case, $\Span_I(\FF_{\cT})$ is a $2$-category whose 2-cells are the isomorphisms of spans
  {\[\noafterskip
    \begin{tikzcd}[sep=small, row sep=tiny]
            & Y' \arrow[dl] \arrow[dr] \arrow[dd,"\sim"description]\\
            X \arrow[dr] 
            && Z \arrow[dl]\\
            & Y
    \end{tikzcd}
  \]}
\end{warning}
In this subsection, we review the cartesian algebraic theory $\Span_I(\FF_{\cT})$ corepresents, called \emph{$I$-commutative monoids}. 
We will find that, in the same way that $\CMon$ is easily characterized via \emph{semiadditivity} (c.f. \cite{Gepner}), $\CMon_I$ is easily characterized via \emph{$I$-semiadditivity}.
Little of this subsection is original; instead, the results concerning $I$-commutative monoids form a slight generalization of \cite{Nardin-Stable} and a massive specialization of \cite{Cnossen_semiadditive}, and the results concerning weak indexing systems are largely review of \cite{Windex}.

\subsubsection{Weak indexing systems}
We begin by briefly reviewing the setting of \emph{weak indexing systems} introduced in \cite{Windex}, which we view as the combinatorial context for the intersection of category theoretic and algebraic notions of $I$-commutative monoids.

\begin{definition}\label{Windex definition}
  A \emph{$\cT$-weak indexing category} is a subcategory $I \subset \FF_{\cT}$ satisfying the following conditions:
  \begin{enumerate}[label={(IC-\alph*)}]
    \item \label[condition]{Restriction stable condition} (restrictions) $I$ is stable under arbitrary pullbacks in $\uFF_{\cT}$;
    \item \label[condition]{Windex segal condition} (segal condition) $T \rightarrow S$ and $T' \rightarrow S$ are both in $I$ if and only if $T \sqcup T' \rightarrow S \sqcup S'$ is in $I$; and
    \item \label[condition]{Automorphism condition} ($\Sigma_{\cT}$-action) if $S \in I$, then all automorphisms of $S$ are in $I$.
  \end{enumerate}
  A \emph{$\cT$-weak indexing system} is a full $\cT$-subcategory $\uFF_{I} \subset \uFF_{\cT}$ satisfying the following conditions:
  \begin{enumerate}[label={(IS-\alph*)}]
    \item whenever the $V$-value $\FF_{I,V} \deq \prn{\uFF_I}_V$ is nonempty, we have $*_V \in \uFF_{I,V}$; and
    \item $\uFF_{I} \subset \uFF_{\cT}$ is closed under $\uFF_I$-indexed coproducts.\qedhere
  \end{enumerate}
\end{definition}

\begin{observation}\label{Reduction to maps to orbits observation}
  By a basic inductive argument, condition (IC-b) is equivalent to the condition that $S \rightarrow T$ is in $I$ if and only if $T_U = T \times_S U \rightarrow U$ is in $I$ for all $U \in \Orb(S)$;
  in particular, $I$ is determined by its slice categories over \emph{orbits}. 
\end{observation}
We denote the \emph{$I$-admissible sets} by $\uFF_I \deq \uSet_I \subset \FF_{\cT}$ as in \cref{Indexed coproducts notation}.
This is a full $\cT$-subcategory.
\begin{remark}\label{Pullback along orbits remark}\label{Profinite remark}
   By \cref{Reduction to maps to orbits observation}, in the presence of \cref{Windex segal condition}, \cref{Restriction stable condition} is equivalent to the condition that for all Cartesian diagrams in $\FF_{\cT}$
    \begin{equation}\label{Restriction along transfer}
        \begin{tikzcd}
            T \times_V U \arrow[r] \arrow[d,"\alpha'"] \arrow[rd,phantom,"\lrcorner" very near start]
            & T \arrow[d,"\alpha"]\\
            U \arrow[r]
            & V
        \end{tikzcd}
    \end{equation}
    with $U,V \in \cT$ and $\alpha \in I$, we have $\alpha' \in I$.
\end{remark}

Inspired by \cref{Reduction to maps to orbits observation,Profinite remark}, in \cite[Thm~A]{Windex} we prove the following.
\begin{proposition}
  The assignment $I \mapsto \uFF_I$ implements an equivalence between the posets of $\cT$-weak indexing categories and $\cT$-weak indexing systems.
\end{proposition}
We additionally recall the following conditions, which may equivalently be restated for weak indexing categories by \cite[Thm~A]{Windex}.
In view of \cite[\S~2.4]{Windex}, we encourage the reader to think primarily of \emph{unitality}.
\begin{definition}
  We say that $\uFF_I$
  \begin{enumerate}[label={(\roman*)}]
    \item has one color if for all $V \in \cT$, we have $\FF_{I,V} \neq \emptyset$,
    \item is almost essentially unital if for all non-contractible $V$-sets $S \sqcup S' \in \FF_{I,V}$, we have $S,S' \in \FF_{I,V}$,
    \item is unital if it has one color and for all $V$-sets $S \sqcup S' \in \FF_{I,V}$, we have $S,S' \in \FF_{I,V}$, and
    \item is an \emph{indexing system} if the subcategory $\uFF_{I,V} \subset \FF_V$ is closed under finite coproducts for all $V \in \cT$.
  \end{enumerate}
  These lie in a diagram of embedded sub-posets
  \[
    \IndSys_{\cT} \subset \wIndSys_{\cT}^{\uni} \subset \wIndSys_{\cT}^{aE\uni} \subset \wIndSys_{\cT}.\qedhere
  \]
\end{definition}

If a weak indexing category $I$ corresponds with a weak indexing system satisfying property $P$, we say that $I$ satisfies property $P$;
if $\uFF_I$ is an indexing system, we say $I$ is an indexing category.
When $\cT = \cO_G$, $\cT$-indexing systems and indexing categories recover the notions given the same name in \cite{Blumberg-op,Blumberg_incomplete} (see \cite{Windex}).
Some useful invariants of these include
\begin{equation}\label{The families equation}
  \begin{split}
    c(I) &\deq \cbr{V \in \cT \mid *_V \in \FF_{I,V}};\\ 
    \upsilon(I) &\deq \cbr{V \in \cT \mid \emptyset_V \in \FF_{I,V}};\\ 
    \nabla(I) &\deq \cbr{V \in \cT \mid 2 \cdot *_V \in \FF_{I,V}}.
  \end{split}
\end{equation}
These are each families \cite[\S~1.2]{Windex}, which we call the families of \emph{colors}, \emph{units}, and \emph{fold maps} in $I$.

We will import these into the setting of $\cT$-operads in \cref{Definition of NIinfty proposition} through \emph{weak $\cN_\infty$-operads}, which play an important structural role in $\Op_{\cT}$.
Narrowly, this role comes down to the fact that indexed coproducts appear as the arities of \emph{compositions} of indexed operations, so weak indexing systems occur as the possible ``arity supports'' that $\cT$-equivariant algebraic theories can have, so long as they possess identity operations and they allows for the formation of composite operations.

Moreover, in the setting of indexing systems, it is typical to make frequent reductions of $S$-ary operations to $[H/K]$-ary operations and binary operations.
In the setting of weak indexing systems, we say that a\emph{sparse $V$-set} is a $V$-set of the form 
\[
  \epsilon \cdot *_V \sqcup W_1 \sqcup \cdots \sqcup W_n
\]
where $\epsilon \in \cbr{0,1}$ and there exist no maps $W_i \rightarrow W_j$ over $V$ for $i \neq j$.
The relevant generation statement is the following.
\begin{proposition}[{\cite[\S~3.1]{Windex}}]\label{Sparse generation prop}
   Suppose $\uFF_I,\uFF_J$ are weak indexing systems.
   \begin{enumerate}
     \item If $\uFF_I$ is almost essentially unital then $\uFF_I \subset \uFF_J$ if and only if $\uFF_J$ contains all sparse $I$-admissible $V$-sets.
     \item If $\uFF_I$ is an indexing system then $\uFF_I \subset \uFF_J$ if and only if for all $V \in \cT$, $\FF_{J,V}$ contains $\emptyset_V, 2 \cdot *_V$, and all transitive $I$-admissible $V$-sets.
   \end{enumerate}
\end{proposition}
\subsubsection{Indexed semiadditivity}
One central source of weak indexing categories is \emph{indexed semiadditivity}, which only makes sense to evaluate in the pointed setting.
\begin{definition}
  Given $\cF \subset \cT$ a $\cT$-family, we say that $\cD$ is \emph{$\cF$-pointed} if $\cD_V$ is pointed for all $V \in \cF$.
\end{definition}

\def\Nm{\mathrm{Nm}}

Given $S \in \FF_V$ a finite $V$-set with a distinguished orbit $W \subset S$, $\cD$ a $\cT_{\leq V}$-pointed $\cT$-$\infty$-category admitting $S$-indexed products and coproducts, and $(X_U) \in \cD_S$ a $S$-tuple in $\cD$, \cite[Cons~5.2]{Nardin-Stable} constructs a map
\[
  \chi_W\cln \Res_W^V \coprod_U^S X_U \rightarrow X_W
\]
by distinguishing a ``diagonal'' $X_W$-summand on the left hand side and dictating the map to be the indentity on this summand and 0 elsewhere;
then, the \emph{norm map} 
\[
  \mathrm{Nm}_S\cln \coprod_U^S X_U \rightarrow \prod_U^S X_W
\]
has projected map $\coprod_U^S X_U \rightarrow \CoInd_W^V X_W$ adjunct to $\chi_W$.

\begin{definition}
  Given $\cD$ a $\cT$-$\infty$-category and $S \in \FF_V$ a finite $V$-set, we say that $S$ is \emph{$\cD$-ambidextrous} if $\cD$ admits $S$-indexed products and coproducts, is $\cT_{\leq V}$-pointed, and for all $(X_U) \in \cD_S$, the norm map is an equivalence
  \[
    \coprod_U^S X_U \xrightarrow\sim \prod_U^S X_U.
  \]
  Given $I$ a $\cT$-weak indexing category, we say that $\cD$ is \emph{$I$-semiadditive} if $S$ is $\cD$-ambidextrous for all $S \in \uFF_I$.
\end{definition}

\begin{remark}\label{Ambidextrous composition remark}
  We've given an elementary presentation of this notion;
  this has been generalized to encapsulate Hopkins-Lurie's \emph{higher semiadditivity} in \cite{Cnossen_semiadditive} (see Example 3.37 there).
  In particular, we find that $T \rightarrow S$ is $\cD$-ambidextrous in the sense of \cite{Cnossen_semiadditive} if and only if the $U$-set $T \times_S U$ is $\cD$-ambidextrous for all orbits $U \subset S$, so we adpot their language for \emph{ambidextrous maps}.
  In particular, by \cite[Prop~3.13,Prop~3.16]{Cnossen_twisted}, ambidextrous maps are closed under composition and base change.  
\end{remark}

Given $\cD$ a $\cT$-$\infty$-category, we define the \emph{semiadditive locus} 
\[
  s(\cD) = \cbr{f\cln T \rightarrow S \, \mid \, f \text{ is } \cD\text{-ambidextrous}} \subset \FF_{\cT}.
\]
This is closed under composition by \cref{Ambidextrous composition remark};
furthermore, it's clear that an equivalence $T \simeq S$ is $\cD$-ambidextrous if and only if $\cD$ is $\cT_{\leq V}$-pointed, so $s(\cD) \subset \FF_{\cT}$ is a subcategory satisfying \cref{Automorphism condition}. 
In fact, we may say more.
\begin{proposition}
  $s(\cD)$ is a weak indexing category and $\cD$ is $I$-semiadditive if and only if $I \leq s(\cD)$.
\end{proposition}
\begin{proof}
  By \cref{Reduction to maps to orbits observation,Ambidextrous composition remark}, $s(\cD)$ satisfies \cref{Windex segal condition}.
  In fact, by \cref{Ambidextrous composition remark}, ambidextrous maps are closed under base change, i.e. $s(\cD)$ satisfies \cref{Restriction stable condition}.
  We're left with verifying that $\cD$ is $I$-semiadditive if and only if $I \leq s(\cD)$, but this follows immediately by unwinding definitions.
\end{proof}

By \cite{Windex}, the poset $\wIndCat_{\cT}$ has joins, which we write as $- \vee -$.
The following is immediate.
\begin{corollary}\label{Join of semiadditivity}
  $\cD$ is $I \vee J$-semiadditive if and only if it is $I$-semiadditive and $J$-semiadditive.
\end{corollary}
From \cref{Sparse generation prop}, we acquire a proof of a familiar corollary:
in the setting of indexing categories, $I$-semiadditivity is a combination of fiberwise semiadditivity and $I$-admissible Wirthm\"uller isomorphisms.
\begin{corollary}
  Let $I$ be an almost essentially unital weak indexing category and $\cD$ a $c(I)$-pointed $\cT$-$\infty$-category.
  Then,
  \begin{enumerate}
    \item $\cD$ is $I$-semiadditive if and only if all sparse $V$-sets are $\cD$-ambidextrous.
    \item If $I$ is an indexing category, then $\cD$ is $I$-semiadditive if and only if $\cD_V$ is semiadditive for all $V \in \cT$ and for all maps of orbits $U \rightarrow V$ in $I$ and objects $X \in \cD_U$, the norm map
      \[
        \Ind_U^V X \rightarrow \CoInd_U^V X
      \]
      is an equivalence.
  \end{enumerate}
\end{corollary}

\subsubsection{$I$-commutative monoids as the $I$-semiadditive completion} 
Let $\mathrm{Trip}^{\adeq} \subset \Fun(\bullet \rightarrow \bullet \leftarrow \bullet, \Cat)$ be the full subcategory spanned by adequate triples.
By definition \cite[Def~3.6]{Barwick1}, $\Span_{-,-}(-)$ forms a functor $\mathrm{Trip}^{\adeq} \rightarrow \Cat$.
Fix $I$ a one-color weak indexing category.
Write $\uFF_{V} \deq \uFF_{\cT, /V} \simeq \uFF_{\cT_{/V}}$ and let $\uFF_{\cT}^I \subset \uFF_{\cT}$ be the wide subcategory whose $V$-value is $\prn{\uFF_{\cT}^I}_V \deq I_V \subset \FF_V \simeq \FF_{\cT, /V}$ is the wide subcategory of maps whose underlying map in $\FF_{\cT}$ lies in $I$.

The wide $\cT$-subcategory inclusion $\uFF_{\cT}^I \subset \uFF_{\cT}$ is fiberwise given by a (one object) weak indexing category \cite[\S~2.1]{Windex}, so in particular, this yields a functor $\cT^{\op} \rightarrow \mathrm{Trip}^{\adeq}$ (c.f. \cite[\S~4.1]{Cnossen_semiadditive}).
We use this to define the composite $\cT$-functor
{\zerodisplayskips
  \[
  \uSpan_I(\uFF_{\cT})\colon \cT^{\op} \xrightarrow{(\uFF_{\cT},\uFF_{\cT}, \uFF_{\cT}^I)} \mathrm{Trip}^{\adeq} \xrightarrow{\Span} \Cat.
\]}
\begin{definition}
  If $\cC$ is a $\cT$-$\infty$-category admitting $I$-indexed products, then the \emph{$\cT$-$\infty$-category of $I$-commutative monoids in $\cC$} is
  \[
    \uCMon_I(\cC) \deq \uFun_{\cT}^{I-\times}\prn{\uSpan_I(\uFF_{\cT}), \cC}.
  \]
  The \emph{$\infty$-category of $I$-commutative monoids} is $\CMon_I(\cC) \deq \Gamma^{\cT} \uCMon(\cC) \simeq \Fun_{\cT}^{I-\times}\prn{\uSpan_I(\uFF_{\cT}), \cC}$.
\end{definition}

\begin{definition}
  We say that a $\cT$-functor $F\colon \cD \rightarrow \cC$ is the \emph{$I$-semiadditive completion of $\cC$} if $\cD$ is $I$-semiadditive and for all $I$-semiadditive $\cT$-categories $\cE$, postcomposition along $F$ yields an equivalence
  \[
    \uFun^{I-\times}(\cE,\cD) \xrightarrow\sim \uFun^{I-\times}(\cE,\cC).
  \]
\end{definition}
Write $\Cat^{I-\times}_{\cT} \subset \Cat_{\cT}$ for the non-full subcategory of $\cT$-$\infty$-categories with $I$-indexed products and $I$-product preserving functors;
write $\iota\colon \Cat^{I-\oplus}_{\cT} \subset \Cat^{I-\times}_{\cT}$ for the full subcategory spanned by $I$-semiadditive $\cT$-$\infty$-categories.
The \emph{$I$-semiadditive completion} is, if it exists, the unit of a partially defined adjunction whose right adjoint is $\iota$.
In fact, it does exist, by the following fundamental theorem.\footnote{To see that the $\cT$-$\infty$-category $\uCMon_I(\cC)$ of \cite{Cnossen_semiadditive} agrees with ours, apply \cite[Lem~4.7]{Cnossen_semiadditive}.}
\begin{theorem}[{\cite[Thm~B]{Cnossen_semiadditive}}]\label{Semiadditivization theorem}
  $U\cln \uCMon_I(\cC) \rightarrow \cC$ is the $I$-semiadditive completion.
\end{theorem}

\subsubsection{Commutative monoids in \tcT-objects}\label{Commutative monoids subsection}
Let $I^{\infty}_{\cT} \subset \FF_{\cT}$ be the minimal indexing category \cite{Windex}.
\begin{observation}\label{I infty obs}
  $I^{\infty}_{\cT}$-indexed products are precisely \emph{trivially} indexed products;
  by \cref{Trivially indexed limits} the $I_{\cT}^{\infty}$-indexed product preserving functors are precisely the fiberwise product-preserving $\cT$-functors.
  Furthermore, a $\cT$-category is $I_{\cT}^{\infty}$-semiadditive if and only if, for each $V \in \cT$, the $\infty$-category $\cC_V$ is semiadditive.
  Thus we have equivalences 
  \begin{align*}
    \Cat_{\cT}^{I^{\infty}_{\cT}-\times} &\simeq \CoFr^{\cT}(\Cat^{\times}),\\
    \Cat_{\cT}^{I^{\infty}_{\cT}-\oplus} &\simeq \CoFr^{\cT}(\Cat^{\oplus}),\\
  \end{align*}
  compatible with the inclusions.
\end{observation}
\cref{Cofree adjunction,I infty obs} directly imply that the $I^\infty$-semiadditive closure satisfies
\[
  \uCMon_{I^{\infty}_{\cT}}(\cC) \simeq \prn{\cT^{\op} \xrightarrow{\cC} \Cat^{\times} \xrightarrow{\CMon} \Cat^{\oplus}};
\]
Cnossen-Lenz-Linsken's semiadditive closure theorem (i.e. \cref{Semiadditivization theorem}) then yields the following.
\begin{corollary}\label{Commutative monoids corollary}
  There is a canonical equivalence $\CMon_{I^\infty_{\cT}}(\cC) \simeq \CMon\prn{\Gamma^{\cT} \cC}$.
\end{corollary}

\subsubsection{$I$-commutative monoids in $\infty$-categories}
We recall a special case of Cnossen-Lenz-Linsken's Mackey functor theorem.
\begin{theorem}[{\cite[Thm~C]{Cnossen_semiadditive}}]\label{CMon in coeff}
  For every presentable $\infty$-category $\cC$, there are canonical equivalences
  \begin{align*}
    \CMon_I(\uCoFr^{\cT}(\cC)) 
    &\simeq \Fun^{\times}\prn{\Span_I(\FF_{\cT}), \cC};\\ 
    \uCMon_I(\uCoFr^{\cT}(\cC))_V 
    &\simeq \Fun^{\times}\prn{\Span_{I_V}(\FF_{V}), \cC}.
  \end{align*}
  Furthermore, given a map $f\colon V \rightarrow W$, the associated restriction functor 
  \[
    \Res_V^W\cln \Fun(\Span_{I_W}(\FF_{W}), \cC) \rightarrow \Fun(\Span_{I_V}(\FF_V), \cC)
  \]
  is given by precomposition along $\Span(\Ind_V^W(-))$.
\end{theorem}
This motivates us to make the following definition.
\begin{definition}
  If $\cC$ is an $\infty$-category with finite products, then the \emph{$\cT$-$\infty$-category of $I$-commutative monoids in $\cC$} is
  \[
    \uCMon_I(\cC) \deq \uCMon_I(\uCoFr^{\cT}(\cC)).\qedhere
  \]
\end{definition}

Similar to the case of $\uCoFr^{\cT}$, this construction is compatible with adjunctions.
\begin{lemma}\label{CMon adjunction lemma}
    Let $I \subset \cT$ be a pullback-stable wide subcategory of an orbital $\infty$-category.
    \begin{enumerate}
        \item If $f:\cC \rightarrow \cD$ is a product-preserving functor, then postcomposition yields a $\cT$-functor
        \[
            f_*\cln \uCMon_I \cC \rightarrow \uCMon_I \cD.
        \]
        \item If $L:\cC \rightleftarrows:R$ is an adjunction whose right adjoint $R$ is product preserving, then
        \[
            L_*\cln \uCMon_I \cC \longrightleftarrows \uCMon_I \cD\colon R_*
        \]
        is a $\cT$-adjunction.
    \end{enumerate}
\end{lemma}
\begin{proof}
    (1) follows by noting that $f_*$ exists since $f$ is product preserving, and it is compatible with restriction because postcomposition and precomposition commute.
    (2) follows by noting that the associated functors
    \[
    L_*\cln \prn{\CMon_I \cC}_V \simeq \Fun^\times\prn{\Span_{I_V}(\FF_V),\cC} \longrightleftarrows \Fun^\times\prn{\Span_{I_V}(\FF_V),\cD} = \prn{\CMon_I \cD)_V}:R_*
    \]
    are adjoint.
\end{proof}

We may unpack the structure of $I$-commutative monoids more using the following.
\begin{construction}\label{V-value construction}
    Let $\cC$ be an $\infty$-category, $X \in \CMon_{I} \cC$ be an $I$-commutative monoid, $V \in \cT$ be an orbit, and $\iota_V\colon \FF \rightarrow \FF_{\cT}$ the finite coproduct-preserving functor sending $* \mapsto V$.
    Then, the \emph{$V$-value} is the pullback
    \[
        \begin{tikzcd}
            \CMon_{I} \cC \arrow[r,"(-)_V"] \arrow[d,phantom,"\simeq"labl]
            & \CMon_{I \times_{\FF_{\cT}, \iota-V} \FF}(\cC) \arrow[d,phantom,"\simeq"labl]\\
            \Fun^{\times}(\Span_I(\FF_{\cT}),\cC) \arrow[r,"\iota_V^*"]
            & \Fun^{\times}(\Span_{I \times_{\FF_{\cT},\iota_V} \FF}(\FF),\cC)
        \end{tikzcd}
    \]
    In particular, $I$ is an indexing category and $X$ is an $I$-commutative monoid, $X_V$ is a commutative monoid in $\cC$.
\end{construction}
\begin{construction}
  Fix $X \in \CMon_{I} (\cC)$ and $f\colon V \rightarrow W$ a map in $I$.
    There exists a natural transformation $\alpha_f\colon \iota_V \rightarrow \iota_W$ whose value on $n$ is the copower map $n \cdot V \rightarrow n \cdot W$;
    this induces a natural transformation $N_V^W\colon (-)_V \implies (-)_W$, which we refer to as the \emph{norm map}.
\end{construction}

\subsubsection{$I$-symmetric monoidal $\infty$-categories}
We refer to 
\[
  \uCat_{I}^{\otimes} \deq \uCMon_I \Cat
\]
as the \emph{$\cT$-$\infty$-category of $I$-symmetric monoidal $\infty$-categories}, and write $\Cat_I^{\otimes} \deq \CMon_I \Cat$.
In the case $I = \FF_{\cT}$, we refer to these simply as \emph{$\cT$-symmetric monoidal $\infty$-categories} and write $\uCat_{\cT}^{\otimes} \deq \uCat_{\FF_{\cT}}^{\otimes}$ and $\Cat_{\cT}^{\otimes} \deq \Cat_{\FF_{\cT}}^{\otimes}$. 

\begin{notation}
  Suppose $S \in \FF_{I,V}$.
  Associated with the structure map $\Ind_V^{\cT} S \rightarrow V$ we have functors
  \[
    \bigotimes\limits_U^S:\cC_S \rightarrow \cC_V, \hspace{40pt} \Delta^S:\cC_V \rightarrow \cC_S
  \]
  called the \emph{$S$-indexed tensor product and $S$-indexed diagonal}.
  We refer to the composite $(-)^{\otimes S}:\cC_V \xrightarrow{\Delta^S} \cC_S \xrightarrow{\otimes_U^S} \cC_V$ as the \emph{$S$-indexed tensor power}.
  In the case $\Ind_V^{\cT} S = W$ is an orbit (i.e. $S$ is a \emph{transitive $V$-set}), we write
  \[
    N_W^V \deq \bigotimes\limits_U^W:\cC_W \rightarrow \cC_V.
  \]
  In general, we will use the inset notation $- \otimes -$ for $\otimes_U^{2 \cdot *_V}$, and when $\emptyset_V \in \uFF_I$, we will refer to the $\emptyset_V$-ary operation $* \rightarrow \cC_V$ as the \emph{$V$-unit} and denote its essential image as $1_V$.
\end{notation}

\begin{observation}
  Suppose $S$, $\abs{\Orb(S)} \cdot *_V$, and all orbits of $S$ are $I$-admissible $V$-sets.
  Then, the following path lies in $I$
  \[
    \Ind_V^{\cT} S \xrightarrow{\Ind_H^G\coprod\limits_{U \in \Orb(S)} \prn{U \xrightarrow{!} V}} \abs{\Orb(S)} \cdot V \xrightarrow{\hspace{30pt} \nabla \hspace{30pt}} V,
  \]
  In algebra, this yields the commutative diagram
  \[\begin{tikzcd}
	{\cC_S} && {\cC_V} \\
	& {\cC_{V}^{\times \Orb(S)}}
	\arrow["{\bigotimes\limits_U^S}"{description}, from=1-1, to=1-3]
	\arrow["{\prn{N_U^V-}}" left, from=1-1, to=2-2]
	\arrow["\otimes" below, from=2-2, to=1-3]
\end{tikzcd}\]
  i.e. $\bigotimes\limits_U^S X_U \simeq \bigotimes_{U \in \Orb(S)} N_U^V X_U$.
  Thus, when $I$ is an indexing category, the indexed tensor products in an $I$-symmetric monoidal $\infty$-category is are determined by their binary tensor products and norms.  
\end{observation}
In \cite[\S~1.2]{Windex}, we saw that $I$-symmetric monoidal $\infty$-categories satisfy a version of the \emph{double coset formula}
\[
  \Res_W^V N_U^V Z \simeq \bigotimes_X^{U \times_V W} \Res^U_X Z
\]
for all cospans $U \rightarrow V \leftarrow W$ in $\cT$ such that $U \rightarrow W$ is in $I$.
Moreover, $\Res_V^W$ and $N_V^W$ preserve applicable trivially indexed tensor products;
when $I$ is an indexing category, this and the double coset formula characterize \emph{all} interactions between restrictions and indexed tensor products.

\begin{construction}\label{Symmetric monoidal evaluation construction}
    Right Kan extensions preserve product preserving functors;
    applying this to the \emph{orbits} functor $F_{\cT}\colon \FF_\cT \rightarrow \FF$ yields a functor
    \[
        \Gamma \deq \Span(F_{\cT})_*\colon \Fun^\times(\Span(\FF_{\cT}),\cC) \rightarrow \Fun^\times(\Span(\FF),\cC).
    \]
    In particular, $\Gamma$ is right adjoint to $\Infl_e^{\cT} \deq \Span(F_{\cT})^*$.
    When $\cC = \Cat$, the counit of this adjunction is a natural $\cT$-symmetric monoidal functor. 
    \[
      \Infl_e^{\cT} \Gamma \cC^{\otimes} \rightarrow \cC^{\otimes}
    \]
    We refer to the (symmetric monoidal) $V$-value of this as the \emph{symmetric monoidal $V$-evaluation}
    \[
        \ev_V\colon \Gamma \cC^{\otimes} \rightarrow \cC_V^{\otimes}.\qedhere
    \]
\end{construction}

\subsubsection{Symmetric monoidal $\cT$-$\infty$-categories}
The $\infty$-category of \emph{symmetric monoidal $\cT$-$\infty$-categories} is
\[
  \Cat_{I^\infty,\cT}^{\otimes} \simeq \CoFr^{\cT} \Cat_\infty^{\otimes} \simeq \CMon \Cat_{\cT}.
\]

\begin{definition}
Suppose $L\cC \subset \cC$ is a localizing $\cT$-subcategory of a symmetric monoidal $\cT$-$\infty$-category.
We say that $L$ is \emph{compatible with the symmetric monoidal structure} if for each $V \in \cT$, the localization $L_V$ is compatible with the symmetric monoidal structure on $\cC_V$ in the sense of \cite[Def~2.2.1.6]{HA}.
\end{definition}
We will crucially use the following proposition in \cref{Canonical SMC section}.
\begin{proposition}\label{Compatible symmetric monoidal prop}
  If $L$ is compatible with the symmetric monoidal structure, there exists a commutative diagram of $\cT$-$\infty$-categories
  \[
    \begin{tikzcd}
      \cC^{\otimes} \arrow[rr,"{L^{\otimes}}"] \arrow[rd,"p"]
      && L\cC^{\otimes} \arrow[dl]\\
      & \prn{\FF_{*}}_{\triv}
    \end{tikzcd}
  \]
  satisfying the following conditions:
  \begin{enumerate}[label={(\alph*)}]
    \item $L\cC^{\otimes}$ is a symmetric monoidal $\cT$-$\infty$-category and $L^{\otimes}$ is a symmetric monoidal $\cT$-functor,
    \item the underlying $\cT$-functor of $L^{\otimes}$ is $L\cln \cC \rightarrow L\cC$, and
    \item $L^{\otimes}$ possesses a fully faithful and lax symmetric monoidal right $\cT$-adjoint extending the inclusion $L\cC \subset \cC$.
  \end{enumerate}
\end{proposition}
\begin{proof}
  This is the specialization of \cite[Thm~2.9.2]{Nardin} to $\cO^{\otimes} \deq \EE_\infty^{\otimes}$.
\end{proof}

\subsection{The canonical symmetric monoidal structure on \tI-commutative monoids}\label{Canonical SMC section}
We now explore the observation that the parameterized presentability results of \cite{Hilman} are sufficiently strong to power non-indexed lifts of \cite{Gepner} in the $I$-semiadditive setting.
\def\PrL{\mathrm{Pr}^L}
\def\uPrL{\underline{\mathrm{Pr}}^L}

\begin{definition}[{c.f. \cite[Thm~3.1.9(2),Thm~6.1.2]{Hilman}}]\label{T-presentable definition}
  A (large) $\cT$-$\infty$-category $\cC$ is \emph{$\cT$-presentable} if it admits finite $\cT$-coproducts and its straightening factors as
    \[
        \cC:\cT^{\op} \rightarrow \mathrm{Pr}^{L,\kappa} \rightarrow \widehat \Cat
    \]
    for some regular cardinal $\kappa$.
    The (nonfull) subcategory 
    \[
        \PrL_{\cT} \subset \widehat\Cat_{\cT}
    \]
    has objects given by $\cT$-presentable $\infty$-categories and morphisms given by $\cT$-left adjoints.
\end{definition}

\begin{observation}
    The conditions of factoring through $\Pr^{L,\kappa}$, of strongly admitting finite $\cT$-coproducts, and of being $\cT$-left adjoints are preserved by restriction;
    hence $\Pr_{\cT}^{L}$ canonically lifts to a (nonfull) $\cT$-subcategory
    \[
        \uPrL_{\cT} \subset \widehat \uCat_{\cT}\qedhere
    \]
\end{observation}

These satisfy an adjoint functor theorem \cite[Thm~6.2.1]{Hilman} and have analogous characterizations to the non-equivariant case;
in particular, $\PrL_{\cT} \subset \widehat \Cat_{\cT}$ is closed under functor $\cT$-$\infty$-categories from small $\cT$-$\infty$-categories \cite[Lem~6.7.1]{Hilman} and by \cref{T-presentable definition}, $\PrL_{\cT}$ is closed under fiberwise $\kappa$-accessible $\cT$-localizations.
Hence $\uCMon_I(\cC)$ is $\cT$-presentable when $\cC$ is $\cT$-presentable.

Additionally, in \cite{Nardin_thesis}, a $\cT$-symmetric monoidal structure was constructed on $\uPrL_{\cT}$.
In order to characterize this structure, we use the following definition 
(c.f. \cite[\S~5.1]{Quigley}).
\begin{definition}[{\cite[Def~5.14]{Quigley}}]
  Fix $S$ a finite $V$-set, $(\cC_U)$ an $S$-$\infty$-category,  $\cD$ a $V$-$\infty$-category, and $F\cln \prod_U^S \cC_U \rightarrow \cD$ a $V$-functor.
  Denote by $(-)_*$ the indexed products in $\Cat_{\cT}$ and $(-)^*$ the restriction.
  We say that $F$ is \emph{$S$-distributive} if, for every pullback diagram
  \[
    \begin{tikzcd}
      T \times_V S \arrow[r,"f'"] \arrow[d,"g'"] \arrow[dr, "\lrcorner" very near start, phantom]
      & T \arrow[d, "g"]\\
      S \arrow[r, "f"]
      & V
    \end{tikzcd}
  \]
  and $S$-colimit diagram $\overline{p}\cln K^{\utr} \rightarrow g^{\prime *} \cC$ for $p\cln K \rightarrow g^{\prime *} \cC$, the composite $T$-functor
  \[
    \prn{f'_* K}^{\utr} \xrightarrow{\mathrm{can}} f'_* \prn{K^{\utr}} \xrightarrow{f'_* \overline{p}} f'_* g^{\prime *} \cC \simeq g^* f_* \cC \xrightarrow{g^* F} g^* \cD
  \]
  is a $T$-colimit diagram for the associated composite $f'_* K \rightarrow g^* \cD$.
  We denote by 
  \[
    \Fun_{\cT}^{\delta}\prn{f_* \cC, \cD} \subset \Fun_{\cT} \prn{f_* \cC, \cD}
  \]
  the full subcategory spanned by $S$-distributive functors.
\end{definition}

By the proof of \cite[Prop~3.25]{Nardin_thesis}, Nardin's $\cT$-symmetric monoidal structure on $\uPrL_{\cT}$ has $V$ unit $\ucS_V$ and indexed tensor products characterized by the universal property
{\[
  \Fun^L_{\cT}\prn{\bigotimes_U^S \cC, \cE} \simeq \Fun_{\cT}^{\delta}\prn{\prod_U^S  \cC, \cD}.
\]\noafterskip}
\begin{definition}
  The $\infty$-category of \emph{presentably $\cT$-symmetric monoidal $\infty$-categories} is the (non-full) subcategory $\CAlg_{\cT}\prn{\underline{\Pr}^{L,\otimes}_{\cT}} \subset \widehat \Cat_{\cT}^{\otimes}$;
  the $\infty$-category of \emph{presentably symmetric monoidal $\cT$-$\infty$-categories} is the (non-full) subcategory $\CAlg\prn{\Pr^L_{\cT}} \subset \CMon\prn{\widehat \Cat_{\cT}}$.
\end{definition}

\begin{observation}\label{Distributivity observation}
  By definition, a $\cT$-symmetric monoidal $\infty$-category whose underlying $\cT$-$\infty$-category is presentable factors through the inclusion $\uPr^L_{\cT} \subset \uCat_{\cT}$ if and only if its structure maps $\cC^{\times S}_V \rightarrow \cC_V$ are in $\Fun^{\delta}_V(\cC^{\times S}_V,\cC_V)$;
  in the language of \cite{Nardin}, a presentably $\cT$-symmetric monoidal $\infty$-category is precisely a \emph{distributive} $\cT$-symmetric monoidal $\infty$-category whose underlying $\cT$-$\infty$-category is presentable. 
\end{observation}

\begin{example}[{\cite[Ex~3.17]{Nardin_thesis}}]
  In the case $\cT = \cO_{C_2}$, given a $[C_2/e]$-distributive $C_2$-functor $F\colon \CoInd_e^{C_2} \cC \rightarrow \cD$, the $e$-value of $F$ is a functor $\cC \times \cC \rightarrow \cD_e$ and the $C_2$-value of $F$ turns coproducts into $\CoInd_e^{C_2} [2] \simeq [2] \sqcup [C_2/e]$-indexed coproducts:
  \[
    F_{C_2}(X \sqcup Y) \simeq F_{C_2}(X) \sqcup F_{C_2}(Y) \sqcup \Ind_e^{C_2} F_e(X,Y).
  \]
  In particular, the norms in a presentably $\cT$-symmetric monoidal $\infty$-category are often not compatible with coproducts.
\end{example}

\begin{example}
  By \cite[Prop~3.2.5]{Nardin}, if $\cC$ is a cocomplete $\infty$-category with finite products such that finite products preserve colimits separately in each variable, then the cartesian symmetric monoidal structures on $\CoFr^V \cC$ lift to a distributive $\cT$A-symmetric monoidal $\infty$-category $\uCoFr^{\cT} \cC^\times$, which we refer to as the \emph{Cartesian structure}.
  It follows from Hilman's characterization of parameterized presentability \cite[Thm~6.1.2]{Hilman} that $\uCoFr^{\cT} \cC$ is presentable, so \cref{Distributivity observation} implies that $\uCoFr^{\cT} \cC^\times$ is presentably symmetric monoidal.
\end{example}

Hilman used the universal property of $\otimes$ in \cite[Prop~6.7.5]{Hilman} to prove the formula 
\[
  \cC \otimes \cD \simeq \uFun_{\cT}^{R}\prn{\cC^{\op},\cD}.
\]
Using this, for any $\cT$-presentable $\cT$-$\infty$-category $\cC$, we have
\begin{align*}
    \uCMon_I(\cC)
    &\simeq \uFun^{I-\times}_{\cT}(\uSpan_I(\uFF_{\cT}),\cC)\\
    &\simeq \uFun^{I-\times}_{\cT}(\uSpan_I(\uFF_{\cT}),\uFun^{R}_{\cT}(\cC^{\op},\ucS_{\cT}))\\
    &\simeq \uFun^{R}_{\cT}(\cC^{\op},\uFun^{I-\times}_{\cT}(\uSpan_I(\uFF_{\cT}),\ucS_{\cT}))\\
    &\simeq \cC \otimes \uCMon_I(\ucS_{\cT}).
\end{align*}
i.e. the functor $\cC \mapsto \uCMon_I(\cC)$ is \emph{smashing}.
In fact, we can say more.

\begin{notation}
  We say that a presentable $\cT$-$\infty$category is \emph{$I$-semiadditive} if its underlying $\cT$-$\infty$-category is $I$-semiadditive, and we let $\PrLS_{\cT} \subset \Pr^{L}_{\cT}$ be the full subcategory spanned by $I$-semiadditive presentable $\cT$-categories.
\end{notation}
It follows from \cref{Semiadditivization theorem} that a $\cT$-presentable $\cT$-$\infty$-category is fixed by $\uCMon_I(-)$ if and only if it's $I$-semiadditive, i.e. the smashing localization corresponding with $\uCMon_I(-)$ is left adjoint to the inclusion $\Pr_{\cT}^{L} \subset \Pr_{\cT}^{L, I-\oplus}$.
By \cite[Lemma~3.6]{Gepner}, this implies that given $\cC \in \CAlg(\PrL_{\cT})$, there is a unique compatible commutative algebra structure on its localization $\uCMon_I(\cC)$.
In other words, we've shown the following.
\begin{theorem}\label{Mode SMC}
    The localizing subcategory
    \[
        \uCMon_I\cln \PrL_{\cT} \rightleftarrows \PrLS_{\cT}\cln \iota
    \]
    is smashing;
    in particular, if $\cD^{\otimes}$ is a presentably symmetric monoidal $\cT$-category, then there is an essentially unique presentably symmetric monoidal $\cT$-$\infty$-category $\uCMon_I^{\otimes-\Mode}(\cD)$ possessing a (necessarily unique) symmetric monoidal lift
    \[
        \Fr^{\otimes}\cln \cD^{\otimes} \rightarrow \uCMon_I^{\otimes-\Mode}(\cD)
    \]
    of $\Fr\cln \cD \rightarrow \uCMon_I(\cD)$.
\end{theorem}
\begin{warning}
  \cref{Mode SMC} is not as \emph{genuinely equivariant} as the user may want, as it constructs \emph{symmetric monoidal structures}, but never norm maps.
  The author is content with this for the purposes of this paper, as the algebraic interpretation of indexed tensor products of $\cT$-operads is unclear.
  She hopes to address the indexed case in forthcoming work.
\end{warning}
\begin{remark}
    Under the equivalence of \cref{CMon in coeff}, writing $\cD = \uCoFr^{\cT}(\cC)$, \cref{Mode SMC} constructs an essentially unique presentably symmetric monoidal structure on $\uCMon_I(\cC)$ subject to the condition that the free functor $\uCoFr^{\cT} \cC \rightarrow \uCMon_I(\cC)$ bears a symmetric monoidal structure under the Cartesian structure.
\end{remark}
\begin{observation}\label{Unit observation}
  The $\cT$-$\infty$-category $\ucS_{\cT}$ is freely generated under $\cT$-colimits by one $\cT$-point, in the sense that evaluation at the $V$-units $(*_{V})$ yields an equivalence \cite[Thm~11.5]{Shah2}
    \[
        \Fun^L_{\cT}(\ucS_{\cT},\cC) \simeq \Gamma \cC.
    \]
    In particular, every symmetric monoidal $\cT$-$\infty$-category receives at most one symmetric monoidal $\cT$-left adjoint from $\ucS_{\cT}$;
    in the case $\cC = \ucS_{\cT}^{\times}$ the condition of \cref{Mode SMC} then may be read as saying that there is a unique presentably symmetric monoidal structure on $\uCMon_I(\ucS_{\cT})$ with $V$-unit $1^{\Mode}_V = \Fr(*_{V})$ for all $V \in \cT$.
    
    Furthermore, by Yoneda's lemma, these $V$-units are characterized by the property that
    \[
        \Map_V(1_V^{\Mode}, X_V) \simeq \Map(*_V,X_V(*_V)) \simeq X_V(*_V).\qedhere
    \]
\end{observation}
We'd like to identify this symmetric monoidal structure via a familiar formula.
We have a candidate:
\begin{proposition}[{\cite[Prop~4.24]{Moshe}, via \cite[Prop~3.3.4]{Cnossen_bispans}}]\label{Localized Day convolution}
    If $\cC$ is a presentably symmetric monoidal $\infty$-category, then the Day convolution structure on $\Fun(\Span_I(\FF_{V}),\cC)$ with respect to the smash product on $\Span_I(\FF_{V})$ is compatible with the localization
    \[
        L_{\Seg}:\Fun(\Span_I(\FF_{V}), \cC) \rightarrow \uCMon_{I}(\cC)_V
    \]
\end{proposition}
\begin{proof}
  By the general criterion \cite[Prop~3.3.4]{Cnossen_bispans}, it suffices to verify that $A_+ \wedge -:\Span(\FF_{V}) \rightarrow \Span(\FF_{V})$ is product-preserving, which follows by the fact that it is colimit preserving and $\Span(\FF_{V})$ is semiadditive.
\end{proof}

By \cref{Compatible symmetric monoidal prop}, \cref{Localized Day convolution} constructs a symmetric monoidal structure $\uCMon_I(\cC)^{\circledast}$ on $\uCMon_I(\cC)$. 
We will show that this agrees with the mode symmetric monoidal structure.
\begin{theorem}\label{Mode is Day theorem}
  Let $\cC^{\otimes}$ be a presentably symmetric monoidal $\infty$-category.
  Then, there is a unique equivalence between the Day convolution and mode symmetric monoidal structures on $\uCMon_I(\cC)$ lifting the identity.
\end{theorem}

The proof of \cite[Lemma~4.21]{Moshe} and \cite[Lemma~5.2.1]{Carmeli} apply identically to the following.
\begin{lemma}\label{Tensor compatible localization lemma}
  Fix $\cA_0,\cA_1,\cB \in \CAlg(\Pr^L_{\cT})$ and $L:\cA_0 \rightarrow \cA_1$ a $\cT$-localization functor which is compatible with the symmetric monoidal structure on $\cA_0$.
  Then, $L \otimes \id_{\cB}:\cA_0 \otimes \cB \rightarrow \cA_1 \otimes \cB$ is a $\cT$-localization functor which is compatible with the symmetric monoidal structure on $\cA_0 \otimes \cB$.
\end{lemma}

\begin{proof}[Proof of \cref{Mode is Day theorem}]
  Set the temporary notation $\uPCMon_I(-) \deq \uFun_{\cT}\prn{\uSpan_I(\uFF_{\cT}), -}$.
  Our argument follows along the lines of \cite[Thm~4.26]{Moshe}.
  Repeating the argument of \cref{Mode SMC}, for all presentably symmetric monoidal $\cT$-$\infty$-categories $\cD$, we acquire a diagram
  \[
    \begin{tikzcd}
      {\uPCMon_I(\cD)} & {\uPCMon_I(\ucS_{\cT}) \otimes \cD} \\
      {\uCMon_I(\cD)} & {\uCMon_I(\ucS_{\cT}) \otimes \cD}
    \arrow["\simeq"{description}, draw=none, from=1-1, to=1-2]
    \arrow[from=2-1, to=1-1, hook]
    \arrow[from=2-2, to=1-2, hook]
    \arrow["\simeq"{description}, draw=none, from=2-1, to=2-2]
  \end{tikzcd}
  \]
  Moreover, under the identification of tensor products and coproducts in $\CAlg(\Pr^L_{\cT})$, the top equivalence corresponds with the postcomposition symmetric monoidal $\cT$-functor $\uPCMon_I(\cS_{\cT}) \rightarrow \uPCMon_I(\cD)$ along the canonical symmetric monoidal left $\cT$-adjoint $\cS_{\cT} \rightarrow \cD$ and the symmetric monoidal free $\cT$-functor $\cD \rightarrow \uPCMon_I(\cD)$ pushforward functor (see \cite[Prop~3.3,3.6]{Moshe}).
  Thus the top arrow can be lifted to a \emph{symmetric monoidal} equivalence.
  We may take adjoint functors to find the diagram
  \[
  \begin{tikzcd}
    {\uPCMon_I(\cD)} & {\uPCMon_I(\ucS_{\cT}) \otimes \cD} \\
    {\uCMon_I(\cD)} & {\uCMon_I(\ucS_{\cT}) \otimes \cD}
	\arrow["\simeq"{description}, draw=none, from=1-1, to=1-2]
	\arrow["{{L_{\Seg}}}", from=1-1, to=2-1]
	\arrow["{{L_{\Seg}}}", from=1-2, to=2-2]
	\arrow["\simeq"{description}, draw=none, from=2-1, to=2-2]
\end{tikzcd}
  \]
  of \cite[Prop~3.3.4]{Cnossen_bispans}.
  The bottom functor is a symmetric monoidal localization of the top. 
  In particular, by \cref{Tensor compatible localization lemma}, it suffices to prove this in the case $\cD = \ucS_{\cT}$.

    The $\cT$-Yoneda embedding is $\cT$-symmetric monoidal for the $\cT$-Day convolution by \cite[Thm~6.0.12]{Nardin}, so  $1^{\Day}_V \simeq y(*_V)$.
    Hence Yoneda's lemma yields that 
    \[
      \Map_V(1_V^{\circledast},X_V) \simeq \Map(y(*_V),X_V) \simeq X_V(*_V),
    \]
    which implies that $1^{\circledast}_V \simeq 1^{\Mode}_V$, naturally in $V$. 
    The theorem then follows by \cref{Unit observation}.
\end{proof}

\begin{remark}
  It is not likely that it is necessary for $\cT$ to be atomic orbital in the above argument;
  indeed, for $\uCMon_I(\cC) \deq \uFun_{\cT}^{\times}(\Span_I(\uFF_{\cT}),\cC)$ to implement \emph{$I$-semiadditivization}, it suffices to assume that $I \subset \FF_{\cT}$ is a weakly extensive subcategory whose slice categories $I_{/V}$ are $n_V$-categories for some finite $n_V \in \NN$ in the sense of \cite{Cnossen_tambara}.
  
  For instance, if $\cP \subset \cT$ is an atomic orbital subcategory of an $\infty$-category, then weakly extensive subcategories $I \subset \FF_{\cT}^{\cP}$ are pre-inductible (and hence satisfy the semiadditive closure theorem)and represent a global version of one-color weak indexing categories.
  Unfortunately, the author is not aware of a symmetric monodial structure on partially presentable $\cT$-categories, and developing such a thing would lead us far afield from our current operadic goals, 
\end{remark}

\subsection{The homotopy \tI-symmetric monoidal \td-category}\label{Truncations of SMCs subsubsection}
Recall that a space is \emph{$(-2)$-truncated} if it is empty, \emph{$(-1)$-truncated} if it is empty or contractible, and for $d \geq 0$, a space $X$ is \emph{$d$-truncated} if it is a disjoint union of connected spaces $(X_\alpha)_{\alpha \in A}$ such that $\pi_m \prn{X_\alpha} = 0$ for all $m > d$ and $\alpha \in A$.

Recall that a \emph{$(d+1)$-category} is an $\infty$-category $\cC$ such that the space $\Map(X,Y)$ is $d$-truncated for all $X,Y \in \cC$.
We say that an $\infty$-category is a \emph{$(-1)$-category} if it is either $*$ or empty.
In general, we write $\Cat_{d} \subset \Cat$ for the full subcategory spanned by the $\infty$-categories with the property that they are $d$-categories.

\begin{definition}
    The $\cT$-$\infty$-category of small $\cT$-$d$-categories is
    \[
        \uCat_{\cT,d} := \uCoFr^{\cT} \Cat_d.
    \]
    A \emph{$\cT$-poset} is a $\cT$-0-category. 
    If $I \subset \FF_{\cT}$ is pullback-stable, the \emph{$\cT$-$\infty$-category of small $I$-symmetric monoidal $d$-categories} is
    \[
        \uCat_{I,d}^{\otimes} := \uCMon_{I} \Cat_d.
    \]
    We write $\Cat_{\cT, d} \deq \Gamma^{\cT} \uCat_{\cT,d}$ and $\Cat_{I,d}^{\otimes} \deq \Gamma^{\cT} \uCat_{I,d}^{\otimes}$.
\end{definition}
By the following lemma, $\uCat_{\cT,d}$ is a $\cT$-$(d+1)$-category and $\Cat_{\cT,d}$ is a $(d+1)$-category.
\begin{lemma}[{\cite[Cor~2.3.4.8,Prop~2.3.4.12,Cor~2.3.4.19]{HTT}}]\label{Catd lemma}
    $\Cat_d$ is a $(d+1)$-category and the inclusion
    \[
        \Cat_d \hookrightarrow \Cat
    \]
    has a right adjoint $h_d\cln \Cat \rightarrow \Cat_d$.
\end{lemma}

\begin{construction}
  By \cref{Cofree adjunction,Catd lemma}, the functor $\uCat_{\cT,d} \hookrightarrow \uCat_{\cT}$ is an inclusion of a localizing $\cT$-subcategory;
let $h_{d}\cln \uCat_{\cT} \rightarrow \uCat_{\cT,d}$ be the associated $\cT$-left adjoint.

The mapping spaces in a product of categories are the product of the mapping spaces;
in particular, the inclusion $\Cat_d \hookrightarrow \Cat$ is product-preserving.
Hence \cref{CMon adjunction lemma,Catd lemma} construct a $\cT$-adjunction
\[
  \begin{tikzcd}[ampersand replacement=\&, column sep = large]
	{\uCat_{I}^{\otimes}} \& {\uCat_{I,d}^{\otimes}}
	\arrow[""{name=0, anchor=center, inner sep=0}, "{h_d}", curve={height=-18pt}, from=1-1, to=1-2]
	\arrow[""{name=1, anchor=center, inner sep=0}, "\iota", curve={height=-18pt}, hook', from=1-2, to=1-1]
	\arrow["\dashv"{anchor=center, rotate=-90}, draw=none, from=0, to=1]
  \end{tikzcd}
\]
whose right adjoint is fully faithful.
We refer to $h_{d}$ as the \emph{homotopy $I$-symmetric monoidal $d$-category}.
\end{construction}

The remainder of this subsection will be dedicated to recognition results for $\cT$-symmetric monoidal $d$-categories, which will be useful throughout the remainder of the paper.
We first reduce this consideration to that of plain $\cT$-$\infty$-categories;
the following proposition follows by unwinding definitions and noting that $\Cat_{d} \hookrightarrow \Cat$ is closed under products.
\begin{proposition}
  If $I \subset \FF_{\cT}$ is a one-object weak indexing system, then $\cC^{\otimes} \in \Cat_{I}^{\otimes}$ is a $I$-symmetric monoidal $d$-category if and only if its underlying $\cT$-$\infty$-category $\cC$ is a $\cT$-$d$-category.
\end{proposition}

Often in equivariant higher algebra, we will find that our objects come with natural $\cT$-functors to $\cT$-$1$-categories, and we'd like to develop a recognition theorem in this case in terms of mapping spaces.
\def\uMor{\underline{\Mor}}
\begin{proposition}\label{Mor truncatedness criterion}
    A $\cT$-$\infty$-category $\cC$ is a $\cT$-$d$-category if and only if
    \[
        \Mor_V(\cC) \deq \Fun(\Delta^1,\cC_V)^{\simeq}
    \]  
    is $(d-2)$-truncated for all $V \in \cT$.
\end{proposition}
\begin{proof}
    By definition, it suffices to prove this in the case $\cT = *$.
    Fix $f,g \in \Mor_V(\cC)$.
    Then, we may present $\Map(f,g)$ as a disjoint union over $a,b$ of homotopies
    \[
        \begin{tikzcd}
        	W & X \\
        	Y & Z
        	\arrow["f", from=1-1, to=1-2]
        	\arrow["a"', from=1-1, to=2-1]
        	\arrow["b", from=1-2, to=2-2]
        	\arrow[shorten <=8pt, shorten >=8pt, Rightarrow, from=2-1, to=1-2]
        	\arrow["g"', from=2-1, to=2-2]
        \end{tikzcd}
    \]
    For fixed $a,b$, this is either empty or equivalent to the component of the space $\Map(S^1,\Map(W,Z))$ whose underlying map is homotopic to $bf$.
    If $\cC$ is a $d$-category, then this is $(d-2)$-truncated;
    conversely, choosing $a,b = \mathrm{id}$ and $f = g$, if this is $(d-2)$-truncated for all $f$, then the mapping spaces of $\cC_V$ are $(d-1)$-truncated for all $V$, i.e. $\cC$ is a $\cT$-$d$-category.
\end{proof}
Given a $\cT$-functor $F\cln \cC \rightarrow \cD$ and a map $\psi\cln \Delta^1 \rightarrow \cC_V$, define the pullback space
\[
    \begin{tikzcd}
        \Mor_F^{\psi}(\cC) \arrow[r] \arrow[d] \arrow[rd,"\lrcorner" very near start, phantom]
        & \Mor_V(\cC) \arrow[d]\\
        B\Aut_\psi \arrow[r,hook]
        & \Mor_V(\cD)
    \end{tikzcd}
\]
so that $\Mor_F^{\psi}(\cC)$ is the disjoint union of the connected components of $\Mor_V(\cC)$ whose image in $\Mor_V(\cD)$ is equivalent to $\psi$.
We say that $F$ has \emph{$(d-1)$-truncated mapping fibers}  if $\Mor_F^{\psi}(\cC)$ is $(d-2)$-truncated for all $V \in \cT$ and $\psi \in \Mor_V(\cC)$.
\begin{corollary}\label{T-fiber truncatedness corollary}
    Suppose $F:\cC \rightarrow \cD$ is a $\cT$-functor and $\cD$ is a $\cT$-1-category.
    Then, the following are equivalent for $d \geq 1$:
    \begin{enumerate}
        \item $F$ has $(d-1)$-truncated mapping fibers.
        \item $\cC$ is a $\cT$-$d$-category.
    \end{enumerate}
    Additionally, the following are equivalent.
    \begin{enumerate}[label={(\arabic*')}]
      \item $F^{\simeq}\cln \cC^{\simeq} \rightarrow \cD^{\simeq}$ is fully faithful and $F$ has $(-1)$-truncated mapping fibers.
      \item $F$ includes $\cC$ as a (replete) $\cT$-subcategory of $\cD$.
    \end{enumerate}
\end{corollary}
\begin{proof}
    After \cref{Mor truncatedness criterion}, the only remaining part is the equivalence between (1') and (2').
    Note that $B\Aut_{\psi}$ is $(-1)$-truncated by \cref{Mor truncatedness criterion}, so (1') is equivalent to the conditions that $\cC$ is a $\cT$-$1$-category and $F_V\colon \cC_V \rightarrow \cD_V$ is a faithful functor which is fully faithful on cores, i.e. it is a (replete) subcategory inclusion.
  \end{proof}

\section{Equivariant operads and symmetric sequences}
In \cref{Algebraic patterns subsection}, we begin by recalling rudiments of the theory of \emph{algebraic patterns and Segal objects} of \cite{Chu} and the theory of \emph{fibrous patterns and the Segal envelope} of \cite{Barkan};
in the case of $\fO = \Span(\FF_{\cT})$, we show in \cref{Operads are fibrous subsection} that this recovers the theory of $\cT$-symmetric monoidal $\infty$-categories, $\cT$-$\infty$-operads (henceforth $\cT$-operads), and the $\cT$-symmetric monoidal envelope of \cite{Nardin}.
We go on in \cref{I operads subsection} to specialize several results of \cite{Chu,Barkan} to this setting and construct the family of \emph{weak $\cN_\infty$-operads}.

After this, we go on to study the \emph{underlying $\cT$-symmetric monoidal sequence} functor in \cref{Symmetric sequence subsection}, showing in \cref{Monadic sseq corollary} that it forms a fiberwise-monadic $\cT$-functor
\[
  \usseq_{\cT}:\uOp^{\oc}_{\cT} \rightarrow \uFun_{\cT}(\uSigma_{\cT},\ucS_{\cT});
\]
in particular, this implies that it is a conservative right $\cT$-adjoint and confirms an atomic orbital lift of \cref{Sseq cool theorem}.
In \cref{Discrete genuine nerve subsection}, we use this to confirm \cref{Nerve cool theorem}.

In \cref{Monad subsection} we go on to compute the monad $T_{\cO}$ for $\cO$-algebras in arbitrary $\cT$-symmetric monoidal $\infty$-categories;
in particular, when $\cC \simeq \ucS_{\cT}$ for a structure whose indexed tensor products are indexed products, we naturally split off a $\cO(S)$-summand from $T_{\cO}(S)$;
using our atomic orbital lift of \cref{Sseq cool theorem}, we conclude that $\Alg_{(-)}(\ucS_{\cT})\cln \Op^{\oc}_{\cT} \rightarrow \Cat$ is conservative.

Last, in preparation for forthcoming work, we initiate in \cref{d-operads subsection} the study of the localizing subcategory of $\cT$-operads whose underlying $\cT$-symmetric sequence is $(d-1)$-truncated, called \emph{$\cT$-$d$-operads};
we show in particular that $\Alg_{(-)}(\ucS_{\cT, \leq n+1})$ detects $n$-equivalences.
Moreover, in \cref{Arity support subsection}, we confirm that the full subcategory of $\cT$-0-operads agrees with the poset of subterminal objects, which themselves agree with the weak $\cN_\infty$-operads.

We finish in \cref{Discrete genuine nerve subsection} by verifying that Bonventre's nerve restricts to an equivalence between categories of $G$-$1$-operads and explicitly describing algebras over $\cT$-1-operads.
We assure the reader exclusively interested in \emph{using} $\cT$-operads that the relevant interpretations of the results of \cref{Algebraic patterns subsection} will be restated throughout the following subsections, so these sections may be black-boxed at the cost of completeness of proofs.

\subsection{Recollections on algebraic patterns} \label{Algebraic patterns subsection}
An algebraic pattern is a collection of data encoding \emph{Segal conditions} for the purpose of homotopy-coherent algebra.
Given an algebraic pattern $\fO$ and a complete $\infty$-category $\cC$, there is an $\infty$-category of \emph{Segal $\fO$-objects in $\cC$}, which we view as $\fO$-monoids in $\cC$;
these are presented as functors $\fO \rightarrow \cC$ satisfying a Segal condition.

We may view Segal $\fO$-objects in $\Cat$ (aka Segal $\fO$-$\infty$-categories) as $\fO$-monoidal $\infty$-categories;
these straighten to cocartesian fibrations over $\fO$ satisfying conditions.
As in \cite[\S~2]{HA}, the condition of \emph{being a cocartesian fibration} may be relaxed to construct a form of operads parameterized by $\fO$, called \emph{fibrous $\fO$-patterns.}

In contrast to the categorical patterns of \cite[\S~B]{HA}, these are manifestly $\infty$-categorical, and it is relatively easy to construct push-pull adjunctions between categories of fibrous patterns over different algebraic patterns;
we found our theory of $I$-operads in this syntax for this reason, as the Boardman-Vogt tensor product is most easily defined in terms of pushforward along maps of algebraic patterns.

The author would like to emphasize that the program surrounding algebraic patterns has achieved many results not mentioned here, as fibrous patterns only play a foundational role.
For a significantly more thorough and elegant treatment, we recommend \cite{Chu,Barkan,Chu_enriched}.

\subsubsection{Algebraic patterns and Segal objects}
\begin{definition}
  An \emph{algebraic pattern} is a triple $(\base,(\base^{\In},\base^{\act}), \base^{\el})$, where $(\base^{\In},\base^{\act})$ is a factorization system on $\base$ and $\base^{\el} \subset \base^{\In}$ is a full subcategory.\footnote{Throughout this paper, we adopt the definition of \emph{factorization system} used in \cite[Rmk~2.2]{Chu}, which does not assert any lifting properties;
  that is, a factorization system on $\cC$ is a pair of wide subcategories $\cC^L,\cC^R \subset \cC$ satisfying the condition that, for all maps $X \xrightarrow{f} X'$, the space of factorizations $X \xrightarrow{l} Y \xrightarrow{r} X'$ with $l \in \cC^L$ and $r \in \cC^R$ is contractible.}
  The $\infty$-category $\AlgPatt \subset \Fun(\mathbf{Q},\Cat)$ is the full subcategory spanned by algebraic patterns, where
  \begin{equation}\label{Q definition}
    \mathbf{Q} := \bullet \rightarrow \bullet \rightarrow \bullet \leftarrow \bullet.\qedhere
  \end{equation}
\end{definition}
We refer to the morphisms in $\base^{\In}$ as ``inert morphisms,'' morphisms in $\base^{\act}$ as ``active morphisms,'' and objects in $\base^{\el}$ as ``elementary objects.''
When it is clear from context, we will abusively refer to the quadruple $\prn{\base,(\base^{\In},\base^{\act}),\base^{\el}}$ simply as $\base$.
The following is our primary source of examples.
\begin{construction}[{\cite[Def~3.2.6]{Barkan}}]\label{Span construction} 
  An \emph{adequate quadruple} is the data of an adequate triple $\cX_b,\cX_f \subset \cX$ in the sense of \cref{I-commutative monoids subsection} together with a full subcategory $\cX_0 \subset \cX_b$;
    the \emph{$\infty$-category of adequate quadruples} is the full subcategory
    \[
        \Quad^{\adeq} \subset \Fun(\mathbf{Q},\Cat)
    \]
    spanned by adequate quadruples, where $\mathbf{Q}$ is defined by \cref{Q definition}.

    Given an adequate quadruple $\cX_0 \subset \cX_b \subset \cX \supset \cX_f$, let $\cX_b^{\op} \subset \Span_{b,f}(\cX)$ be the wide subcategory spanned by the spans $X \xleftarrow{\psi_b} R \xrightarrow{\psi_f} Y$ with $\psi_f$ an equivalence, and similarly $\cX_f \subset \Span_{b,f}(\cX)$ the wide subcategory of spans with $\psi_b$ an equivalence.
    This yields a factorization system \cite[Prop~4.9]{Haugseng_two} 
    \[
        \cX_b^{\op} \hookrightarrow \Span_{b,f}(\cX) \hookleftarrow \cX_f.
    \]
    We define the span pattern $\Span_{b,f}\prn{\cX;\cX^{\op}_0}$ via the data
    \begin{itemize}
      \item underlying $\infty$-category $\Span_{b,f}(\cX)$,
      \item inert morphisms $\cX_b^{\op} \subset \Span(\cX)$,
      \item active morphisms $\cX_f \subset \Span(\cX)$, and
      \item elementary objects $\cX_0^{\op} \subset \cX_b^{\op}$.
    \end{itemize}  
    Given a map of adequate quadruples $\prn{\cX,(\cX_b,\cX_f),\cX_0} \rightarrow \prn{\cY,(\cY_b,\cY_f),\cY_0}$ the associated functor $\Span_{b,f}(\cX) \rightarrow \Span_{b,f}(\cY)$
    preserves inert morphisms, active morphisms, and elementary objects by definition;
    hence the functor $\Span_{-,-}(-;-)\cln \Quad^{\adeq} \rightarrow \Fun(\mathbf{Q}, \Cat)$ descends to a functor
    \[
        \Span_{-,-}(-;-)\cln \Quad^{\adeq} \rightarrow \AlgPatt.\qedhere
    \]
\end{construction}

The central example for equivariant higher algebra is the following.
\begin{example} 
  When $\cT$ is an orbital $\infty$-category, $I \subset \FF_{\cT}$ a $\cT$-weak indexing system (e.g. $I = \FF_{\cT}$), and $c(I)$ its \emph{color family} in the sense of \cref{The families equation}, we define the \emph{effective $I$-Burnside pattern}
    \[
      \Span_I(\FF_{\cT}) := \Span_{\mathrm{all},I}\prn{\FF_{c(I)};c(I)}\qedhere
    \]
\end{example}

\begin{example}
    Given $\cT$ an orbital $\infty$-category, we may define the \emph{algebraic pattern of finite pointed $\cT$-sets} as
    \[
      \Tot \uFF_{\cT,*} \deq \Span_{\mathrm{si},\mathrm{tdeg}} \prn{ \Tot \uFF_{\cT}; \cT^{\op}},
    \]
    where morphisms are in $\Tot \uFF_{\cT}^{\mathrm{tdeg}}$ if their projection to $\cT^{\op}$ is homotopic to an identity, morphisms are in $\Tot \uFF_{\cT}^{\mathrm{si}}$ if they're a composition of cocartesian arrows and target-degenerate summand inclusions, and the inclusion $\cT^{\op} \rightarrow \Tot \uFF_{\cT,*}$ corresponds with the $\cT$-object $*_{\cT} \in \Gamma^{\cT} \uFF_{\cT}$.
    Note that the \emph{target functor} $\Tot \uFF_{\cT} \rightarrow \cT^{\op}$ determines a functor $\Tot \uFF_{\cT,*} \rightarrow \cT^{\op}$;
    this corresponds with the structure functor of the free pointed $\cT$-$\infty$-category $\uFF_{\cT,*}$ on $\uFF_{\cT}$ as defined in \cite{Nardin, Cnossen_semiadditive}.
    Moreover, there is a composite map of algebraic patterns
    \begin{equation}\label{Span comparison map}
      \varphi\colon \Tot \uFF_{\cT,*} \hookrightarrow \Span_{\mathrm{all}, \mathrm{tdeg}}(\Tot \uFF_{\cT};\cT^{\op}) \xrightarrow{U} \Span(\FF_{\cT}).\qedhere
    \end{equation}
\end{example}

Algebraic patterns provide a general framework for algebraic structures satisfying the associated \emph{Segal conditions}, which are encoded in the notion of \emph{Segal objects}.
\begin{definition}
  Let $\cC$ be a complete $\infty$-category and $\fO$ an algebraic pattern.
  Then, the $\infty$-category of \emph{Segal $\fO$-objects in $\cC$} is the full subcategory $\Seg_{\fO}(\cC) \subset \Fun(\fO,\cC)$ consisting of functors $F\colon \fO \rightarrow \cC$ such that, for every object $O \in \fO$, the natural map
  \[
    F(O) \rightarrow \lim_{E \in \fO^{\el}_{O/}} F(E)
  \]
  is an equivalence, where $\fO^{\el}_{O/} \deq \fO^{\el} \times_{\fO^{\In}, \ev_1} \fO^{\In}_{O/}$ is the $\infty$-category whose objects consist of inert morphisms from $O$ to an elementary object.
\end{definition}

\begin{remark}
    By \cite[Lem~2.9]{Chu}, a functor $F\colon \fO \rightarrow \cC$ is a Segal $\fO$-object if and only if the associated functor $F|_{\fO^{\Int}}$ is right Kan extended from $F|_{\fO^{\el}}$ along the inclusion $\fO^{\el} \rightarrow \fO^{\Int}$.
\end{remark}

\begin{example}\label{CMon pattern example}
    We show in \cref{Span segal condition lemma} that $\Span_I(\FF_{\cT})^{\el}_{S/} \simeq \prn{\FF_{\cT, /S}}^{\op}$ contains the set of orbits $\Orb(S)$ as an initial subcategory.
    Hence there is an equivalence of full subcategories 
    \[
        \Seg_{\Span_I(\uFF_{\cT})}(\cC) \simeq \CMon_I(\cC) \hspace{10pt} \subset  \hspace{10pt} \Fun(\Span_I(\FF_{\cT}), \cC).\qedhere
    \]
\end{example}

One benefit of the framework of Segal objects is the following monadicity result.
\begin{proposition}[{\cite[Cor~8.2]{Chu}}]\label{Segal objects are monadic}
    if $\fO$ is an algebraic pattern and $\cC$ a presentable $\infty$-category, then the forgetful functor
    \[
        U\colon \Seg_{\fO}(\cC) \rightarrow \Fun(\fO^{\el}, \cC)
    \]
    is monadic; 
    in particular, it is conservative.
\end{proposition}
\begin{corollary}\label{Underlying conservative proposition}
    A morphism of $I$-commutative monoids is an equivalence if and only if its underlying morphism of $c(I)$-objects is an equivalence;
    in particular, an $I$-symmetric monoidal functor $F\colon \cC^{\otimes} \rightarrow \cD^{\otimes}$ is an equivalence if and only if the underlying $c(I)$-functor is an equivalence.
\end{corollary}

Another benefit of Segal objects is a rich framework for functoriality. 
\begin{definition}
  Suppose $\fP,\fO$ are algebraic patterns.
  A functor $f\colon \fP \rightarrow \fO$ is \emph{compatible with Segal objects} if it preserves the inert-active factorization system and $f^*\colon \Fun(\fO,\cC) \rightarrow \Fun(\fP,\cC)$ preserves Segal objects in any complete $\infty$-category $\cC$.
  Moreover, a morphism of algebraic patterns $f\colon \fP \rightarrow \fO$ is a called a:
  \begin{itemize}
    \item \emph{Segal morphism} if it is compatible with Segal objects, and a
    \item \emph{strong Segal morphism} if the associated functor $f^{\el}_{X/}\colon \fP^{\el}_{X/} \rightarrow \fO^{\el}_{f(X)/}$ is initial for all $X \in \fP$.\qedhere
  \end{itemize}
\end{definition}

\begin{observation}\label{AlgPatt Seg observation}
  The conditions for Segal morphisms and strong Segal morphisms are each compatible with compositions and equivalences;
  that is, there are \emph{core-preserving} wide subcategories 
  \[
    \AlgPatt^{\Seg}, \AlgPatt^{\mathrm{Strong-}\Seg} \subset \AlgPatt
  \]
  whose morphisms are the Segal morphisms and strong Segal morphisms, respectively.
\end{observation}

There is a universal example of coefficients, which we can use to verify that a functor is a Segal morphism.
\begin{remark}
  \cite[Lem~4.5]{Chu} concludes that $f$ is a Segal morphism if $f^*$ preserves Segal objects in \emph{spaces}.
\end{remark}

\begin{example}\label{Segal morphisms between span patterns example}
    We show in \cref{Segal morphisms between span patterns prop} that, given any functor $\cT \rightarrow \cT'$ of atomic orbital $\infty$-categories, the associated functor
    \[
        \Span(\FF_{\cT}) \rightarrow \Span(\FF_{\cT'})
    \]
    is a Segal morphism.
    Additionally, in \cref{Span fibrous corollary}, we show that the map $\varphi$ of \cref{Span comparison map} is a segal morphism, constructing a pullback map
    \[
        \CMon_{\cT}(\cC) \simeq \Seg_{\Span(\FF_{\cT})}(\cC) \rightarrow \Seg_{\tot \uFF_{\cT,*}}(\cC).
    \]
    In \cite[Cor~2.64]{Barkan-arity}, conditions for a strong Segal morphism were developed concerning when their pullback maps are equivalences, and these conditions were checked in \cite[Prop~5.2.14]{Barkan} in the case $\cT = \cO_G$;
    we review their argument and extend it to arbitrary atomic orbital $\infty$-categories in \cref{Operads are fibrous subsection}.
    The existence of such an equivalence (not necessarily induced by a map of patterns) is not new, and to the author's knowledge, first appeared as \cite[Thm~6.5]{Nardin-Stable}.
\end{example}

\emph{Limits of patterns} construct a large number of examples according to the following lemma.
\begin{lemma}[{\cite[Cor~5.5]{Chu}}]
    $\AlgPatt \subset \Fun(\mathbf{Q},\Cat)$ is a localizing subcategory;
    in particular, $\AlgPatt$ has small limits.
\end{lemma}

\def\coeff{\op,\el}
\begin{example}\label{Product pattern example}
    In particular, $\AlgPatt$ has products.
    By \cite[Ex~5.7]{Chu}, there is an equivalence
    \[
        \Seg_{\base \times \base'}(\cC) \simeq \Seg_{\base} \Seg_{\base'}(\cC).
    \]
    In particular, this combined with \cref{CMon pattern example} gives a complete segal space model for $I$-symmetric monoidal categories;
    indeed, the pattern $\Delta^{\op,\natural}$ of \cite[Ex~4.9]{Chu} has Segal $\Delta^{\op,\natural}$-objects in $\cC$ given by \emph{complete Segal objects in $\cC$}, specializing to the fact that $\Seg_{\Delta^{\op,\natural}}(\cS) \simeq \Cat$, and hence
    \[
      \Seg_{\Delta^{\op,\natural}}(\cS_{\cT}) \simeq \Seg_{\cT^{\coeff} \times \Delta^{\op,\natural}}(\cS) \simeq \Seg_{\cT^{\coeff}}(\Cat) \simeq \Cat_{\cT},
    \]
    where $\cT^{\coeff}$ is the algberaic pattern with $\prn{\cT^{\coeff}}^{\el} = \prn{\cT^{\coeff}}^{\Int} = \cT^{\op} = \prn{\cT^{\coeff}}^{\act}$.
    Additionally,
    \[
      \Seg_{\Delta^{\op,\natural}}(\CMon_{\cT}(\cS)) \simeq \Seg_{\Delta^{\op,\natural} \times \Span(\FF_{\cT})}(\cS) \simeq \Seg_{\Span(\FF_{\cT})}(\Cat) \simeq \CMon_{\cT}(\Cat).\qedhere
    \]
\end{example}

Cartesian products of patterns play nicely with well-structured maps of patterns.
\begin{lemma}\label{Products of strong segal morphisms}
    Suppose $f\colon \fO \rightarrow \fP$ and $f'\colon \fO' \rightarrow \fP'$ are (resp. strong) Segal morphisms.
    Then,
    \[
        f \times f'\colon \fO \times \fO' \rightarrow \fP \times \fP'
    \]
    is a (strong) Segal morphism.
\end{lemma}
\begin{proof}
    The case of Segal morphisms follows immediately from \cref{Product pattern example}, so we assume that $f,f'$ are strong Segal.
    Then, the induced map
    \[
        f^{\el}_{X/} \times f^{'\el}_{X'/} = \prn{f \times f'}^{\el}_{(X,X')/}:\prn{\fO \times \fO'}^{\el}_{(X,X')/} \rightarrow \prn{\fP \times \fP'}^{\el}_{(fx,fx')/}
    \]
    is a product of initial maps; it follows that it is initial, since limits in product categories are computed pointwise.
\end{proof}

\subsubsection{An interlude on soundness and extendability}%
We will move on to describe the theory of operads corresponding with an algebraic pattern, but to do so, we make some technical assumptions.
Let $\fO$ be an algebraic pattern and $\omega\colon X \rightarrow Y$ an active map.
Define the pullback square 
\[
  \begin{tikzcd}
    \fO^{\el}(\omega) \arrow[r] \arrow[d] \arrow[rd,"\lrcorner" very near start, phantom]
    & \Ar(\fO^{\Int}_{X/}) \arrow[d,"{(s,t)}"]\\
    \fO^{\el}_{Y/} \times \fO^{\el}_{X/} \arrow[r,"\prn{\omega_{(-)},\id}"]
    & \fO^{\Int}_{X/} \times \fO^{\Int}_{X/}
  \end{tikzcd}
\]
where $\omega_{(-)}\cln \fO^{\el}_{Y/} \rightarrow \fO^{\Int}_{X/}$ sends $\alpha:Y \rightarrow E$ to the inert map $\omega_a$ of the inert-active factorization of $X \xrightarrow \omega Y \xrightarrow{a} E$.

\begin{definition}
  $\fO$ is \emph{sound} if, for all $\omega:X \rightarrow Y$ active, the associated map $\fO^{\el}(\omega) \rightarrow \fO^{\el}_{X/}$ is initial.
  A sound pattern $\fO$ is \emph{soundly extendable} if $\sA_{\fO} \deq \Ar^{\act}(\fO) \xrightarrow{t} \fO$ is a Segal $\fO$-$\infty$-category, where $\Ar^{\act}(\fO) \subset \Ar(\fO) = \Fun(\Delta^1, \fO)$ is the full subcategory spanned by active arrows.
\end{definition}

Soundness as a condition allows one to simplify Segal conditions;
sound extendibility reduces many instances of \emph{relative Segal objects} in the sense \cite[Def~3.1.8]{Barkan} to a morphism with Segal domain by \cite[Obs~3.1.9]{Barkan}.
A condition of \emph{extendability} was originally introduced in \cite[Def~8.5]{Chu} for the sake of explicit formulas for the free Segal $\fO$-object monad, and is equivalent to sound extendability in the presence of soundness \cite[Rmk~3.3.17]{Barkan};
we will not consider the reasoning for this notion further, but instead remark that it is true of our main examples.

\begin{example}/
  We verify in \cref{Soundly extendable lemma} that $\Span(\FF_{\cT})$ is soundly extendable;
  moreover, we verify in \cref{Sound lemma} that $\Tot \uFF_{\cT,*}$ is sound, and one may verify that it is soundly extendable.
\end{example}

\subsubsection{Fibrous patterns}%
The unstraightening functor of \cite{HTT} realizes $\Seg_{\fO}(\Cat)$ as a non-full subcategory of $\Cat_{/\fO}$ consisting of cocartesian fibrations satisfying Segal conditions;
we relax this for the following definition, which is equivalent to the original definition stated in \cite[Def~4.1.2]{Barkan} by \cite[Prop~4.1.6]{Barkan}.

\begin{definition}\label{Fibrous patterns definition}
   Let $\base$ be a sound algebraic pattern.
  A \emph{fibrous $\base$-pattern} is a functor $\pi:\fO \rightarrow \base$ such that
  \begin{enumerate}
    \item (inert morphisms) $\fO$ has $\pi$-cocartesian lifts for inert morphisms of $\base$,
    \item (Segal condition for colors) For every active morphism $\omega\cln V_0 \rightarrow V_1$ in $\base$, the functor
      \[
        \fO_{V_0} \rightarrow \lim_{\alpha \in \base^{\el}_{V_1/}} \fO_{\omega_{\alpha,!} V_1}
      \]
      induced by cocartesian transport along $\omega_\alpha$ is an equivalence, where $\omega_{(-)}\cln \base^{\el}_{Y/} \rightarrow \base^{\Int}_{X/}$ is the inert morphism appearing in the inert-active factorization of $\alpha \circ \omega$, and
    \item (Segal condition for multimorphisms) for every pair of objects $V_1,V_2 \in \base$ and colors $X_i \in \fO_{V_i}$, the commutative square
      \[
        \begin{tikzcd}
          \Map_{\fO}(X_0,X_1) \arrow[r] \arrow[d]
          & \lim\limits_{\alpha\colon V_1 \rightarrow E \in \base^{\el}_{V_1/}} \Map_{\fO}(X_0,\alpha_! X_1) \arrow[d]\\
          \Map_{\base}(V_0,V_1) \arrow[r]
          & \lim\limits_{\alpha\colon V_1 \rightarrow E \in \base^{\el}_{V_1/}} \Map_{\base}(V_0, E)
        \end{tikzcd}
      \]
      is cartesian.
  \end{enumerate}
  We denote by $\Fbrs(\base) \subset \Cat^{\Int-\cocart}_{/\base}$ the full subcategory spanned by the fibrous $\base$-patterns, where the latter category has objects the functors to $\base$ possessing cocartesian lifts over inert morphisms and morphisms the functors preserving such cocartesian lifts.
\end{definition} 

\begin{remark}\label{Enough to assert essentially surjective remark}
  As noted in \cite[Rmk~4.1.8]{Barkan}, in the presence of condition (3) above, condition (2) may be weakened to assert that the functor $\fO_{V_0} \rightarrow \lim_{\alpha \in \base^{\el}_{V_1/}} \fO_{\omega_{\alpha,!} V_1}$ is a $\pi_0$-surjection without changing the resulting notion.
  To match \cite[Prop~4.1.6]{Barkan}, we may even take the intermediate assumption that this functor induces an equivalence on cores.
\end{remark}

\begin{example}
  Fibrous $\FF_*$-patterns are equivalent to $\infty$-operads (c.f. \cite{HA}), and in \cref{Operads are fibrous subsection} we will extend a proof due to \cite{Barkan} (in the case $\cT = \cO_G$) that fibrous $\Tot \uFF_{\cT,*}$-patterns are equivalent to the $\cT$-$\infty$-operads of \cite{Nardin}. 
\end{example}

The fully faithful functor $U\colon \Fbrs(\base) \rightarrow \Cat^{\Int-\cocart}_{/\base}$ is a reflective subcategory inclusion.
\begin{proposition}[{\cite[Cor~4.2.3]{Barkan}}]
  $U$ participates in an adjunction 
    \[
      \begin{tikzcd}
        {\Cat_{/\base}^{\Int-\cocart}}
        & {\Fbrs(\base)} 
	\arrow[""{name=0, anchor=center, inner sep=0}, "U", hook', curve={height=-15pt}, from=1-2, to=1-1]
	\arrow[""{name=1, anchor=center, inner sep=0}, "{L_{\Fbrs}}", curve={height=-15pt}, from=1-1, to=1-2]
	\arrow["\dashv"{anchor=center, rotate=-90}, draw=none, from=0, to=1]
\end{tikzcd}
    \]
\end{proposition}

In terms of functoriality, we prove the following in \cref{Soundly extendable pull-push prop}, extending \cite[Lem~4.1.19]{Barkan}.
\begin{proposition}\label{Fibrous pullback prop}
  Suppose $f\cln \fP \rightarrow \fO$ is a Segal morphism and either $\fO$ is soundly extendable or $f$ is strong Segal.
  Then, the pullback functor $f^*\cln \Cat_{/\fP} \rightarrow \Cat_{/\fO}$ preserves fibrous patterns;
  furthermore, the functor 
  \[
    f^*\cln \Fbrs(\fO) \rightarrow \Fbrs(\fP)
  \]
  has a left adjoint given by $L_{\Fbrs} f_!$.
\end{proposition}

\cref{Segal morphisms between span patterns example,Fibrous pullback prop} together yield a functor
  \[
    \Fbrs(\Span(\FF_{\cT})) \rightarrow \Fbrs(\Tot \uFF_{\cT,*});
  \]
  we review a proof that this is an equivalence (originally due to \cite{Barkan} when $\cT = \cO_G$) in \cref{Span fibrous corollary}.

A fibrous pattern $\pi\colon \fO \rightarrow \base$ inherits a structure of an algebraic pattern whose inert morphisms consist of $\pi$-cocartesian lifts of inert morphisms in $\base$, whose active morphisms are arbitrary lifts of active morphisms in $\base$, and whose elementary objects are spanned by lifts of elementary objects.
This is canonical:
\begin{proposition}[{\cite[Cor~4.1.7]{Barkan}}]\label{Overcategory of fibrous patterns prop}
    Fibrous patterns are closed under composition for the above pattern structure, inducing an equivalence
    \[        
        \Fbrs(\fO) \simeq \Fbrs(\base)_{/\fO}.
    \]
\end{proposition}

We construct many Segal morphisms in \cref{Segal morphisms subsection}.
Many more are constructed in the following.
\begin{proposition}[{\cite[Obs~4.1.14]{Barkan}}]\label{Fibrous patterns are strong Segal morphisms}
    Fibrous patterns are strong Segal morphisms.
\end{proposition}

\subsubsection{The Segal envelope}
In \cite[Lem~4.2.4]{Barkan} it was verified that a cocartesian fibration to $\fO$ is a fibrous $\fO$-pattern if and only if it's the straightening of a Segal $\fO$-category (assuming $\fO$ is sound);
this lifts the fact that an operad $\cC^{\otimes}$ is a symmetric monoidal $\infty$-category if and only if the corresponding functor $\cC^{\otimes} \rightarrow \FF_*$ is a cocartesian fibration.
We would like to describe adjunctions relating fibrous patterns to Segal objects, but to do so, we need a few constructions.
\begin{definition}
  Given $\fO \rightarrow \base$ a map of algebraic patterns, the \emph{Segal envelope of $\fO$ over $\base$} is the horizontal composite
    \[
        \begin{tikzcd}
        	{\Env_{\base} \fO} & {\Ar^{\act}(\base)} & \base \\
        	\fO & \base
        	\arrow[from=1-1, to=1-2]
        	\arrow[from=1-1, to=2-1]
        	\arrow["\lrcorner"{anchor=center, pos=0.125}, draw=none, from=1-1, to=2-2]
        	\arrow["t", from=1-2, to=1-3]
        	\arrow["s", from=1-2, to=2-2]
        	\arrow[from=2-1, to=2-2]
        \end{tikzcd}
    \]
    where $\Ar^{\act}(\base) \subset \Ar(\base) = \Fun(\Delta^1,\base)$ is the full subcategory spanned by active arrows and $s,t$ are the \emph{source and target} functors.
    We denote the envelope of the terminal $\base$-pattern as
    \[
        \sA_\base:= \Ar^{\act}(\base) \xrightarrow{t} \base.\qedhere
    \]
\end{definition}

Given $f\cln \fP \rightarrow \fO$ a Segal morphism between algebraic patterns, we then define the composite functor 
\[
  f^{\circledast}\cln\Seg_{\fO}^{/\sA_{\fO}} \xrightarrow{f^{*}} \Seg_{\fO}^{/f^* \sA_{\fO}} \xrightarrow{q^*} \Seg_{\fO}^{/\sA_{\fP}} 
\]
where $q$ is the map fitting into the following diagram:
\[\begin{tikzcd}[row sep=small]
	{\sA_{\fP}} \\
	& {f^* \sA_{\fO}} & {\sA_{\fO}} \\
	& \fP & \fO
  \arrow["\sA_{f}", from=1-1, to=2-3]
	\arrow["p"', from=1-1, to=3-2]
	\arrow["p", from=2-3, to=3-3]
	\arrow["f"', from=3-2, to=3-3]
  \arrow[from=2-2, to=2-3]
	\arrow[from=2-2, to=3-2]
	\arrow["q"{description}, from=1-1, to=2-2]
	\arrow["\lrcorner"{anchor=center, pos=0.125}, draw=none, from=2-2, to=3-3]
\end{tikzcd}\]
This participates in the following theorem, which was proved under a \emph{strong Segal} assumption which is rendered unnecessary by \cref{Fibrous pullback prop}.
\begin{theorem}[{\cite[Prop~4.2.1,Prop~4.2.5,Thm~4.2.6,Rem~4.2.8]{Barkan}}]\label{Envelope theorem}
  Let $\fO$ be a soundly extendable pattern.
  Then, $\Env_{\fO}$ participates in an adjunction
  \[
    \begin{tikzcd}
	{\Fbrs(\fO)} & {\Seg_{\fO}(\Cat).}
	\arrow[""{name=0, anchor=center, inner sep=0}, "{\Env_{\fO}}", curve={height=-12pt}, from=1-1, to=1-2]
	\arrow[""{name=1, anchor=center, inner sep=0}, "\Un", curve={height=-12pt}, from=1-2, to=1-1]
	\arrow["\dashv"{anchor=center, rotate=-90}, draw=none, from=0, to=1]
\end{tikzcd}
  \]
  By taking slice categories, this induces an adjunction
  \[
    \begin{tikzcd}
	{\Fbrs(\fO)} & {\Seg_{\fO}(\Cat)}
  \arrow[""{name=0, anchor=center, inner sep=0}, hook', "{\Env^{/\sA_{\fO}}_{\fO}}", curve={height=-12pt}, from=1-1, to=1-2]
	\arrow[""{name=1, anchor=center, inner sep=0}, curve={height=-12pt}, from=1-2, to=1-1]
	\arrow["\dashv"{anchor=center, rotate=-90}, draw=none, from=0, to=1]
\end{tikzcd}
  \]
  whose left adjoint is fully faithful.
  Furthermore, if $f\colon \fO \rightarrow \fP$ is a Segal morphism between soundly extendable patterns, the following diagram commutes:
  \[\begin{tikzcd}[row sep = large, column sep = huge]
	{\Seg_{\fO}(\Cat_\infty)} & {\Fbrs(\fO)} & {\Seg_{\fO}(\Cat_\infty)_{/\sA_{\fO}}} & {\Fbrs(\fO)} \\
	{\Seg_{\fP}(\Cat_\infty)} & {\Fbrs(\fP)} & {\Seg_{\fP}(\Cat_\infty)_{/\sA_{\fP}}} & {\Fbrs(\fP)}
	\arrow["{f^*}"', from=1-1, to=2-1]
	\arrow["{f^*}"', from=1-2, to=2-2]
	\arrow["\Un", from=1-1, to=1-2]
	\arrow["\Un"', from=2-1, to=2-2]
  \arrow[hook', "{\Env_{\fO}^{/\sA_{\fO}}}", from=1-2, to=1-3]
  \arrow[hook, "{\Env_{\fP}^{/\sA_{\fP}}}"', from=2-2, to=2-3]
	\arrow["{f^{\circledast}}"', from=1-3, to=2-3]
	\arrow["\Un", from=1-3, to=1-4]
	\arrow["\Un"', from=2-3, to=2-4]
	\arrow["{f^*}", from=1-4, to=2-4]
\end{tikzcd}\]
\end{theorem}

We will make frequent use of product patterns, so we observe their interaction with Segal envelopes.

\begin{observation}\label{Product of envelopes observation}
  If $\fO,\fP$ are fibrous $\base$-patterns, then their Segal envelopes satisfy
  \begin{align*}
    \Env_{\base \times \base}(\fO \times \fP) 
    &\simeq \prn{\fO \times \fP} \times_{\base \times \base} \Ar^{\act}(\base \times \base)\\
    &\simeq \prn{\fO \times_{\base} \Ar^{\act}(\base)} \times \prn{\fP \times_{\base} \Ar^{\act}(\base)}\\
    &\simeq \Env_{\base}(\fO) \times \Env_{\base}(\fP)\qedhere
  \end{align*}
\end{observation}

We finish with a right handed construction which will be useful in \cref{Operads section}. 
\begin{observation}\label{Push-pull along projection corollary}
  Suppose $\base$, $\base'$ are soundly extendable algebraic patterns, modelled within quasicategories.
  $\base$ has an associated categorical pattern 
  \[
  \mathrm{CatPatt}(\base) \deq \prn{\base, \mathrm{inert}, \mathrm{all}, \cbr{\base^{\el}_{O/}}_{O \in \base}}
  \]
  Unwinding definitions, fibrous $\base$ patterns are presented by the model structure on $\Set^+_{\Delta, /\mathrm{CattPatt}(\base)}$ constructed in \cite[Thm~B.0.20]{HA}.
  In particular, we may apply \cite[Rmk~B.2.5]{HA} to conclude that cartesian products furnish a distributive bifunctor 
  \[
    \Fbrs(\base) \times \Fbrs(\base') \rightarrow \Fbrs(\base \times \base');
  \]
  the restriction $\Fbrs(\base) \simeq \Fbrs(\base) \times \cbr{\base'} \rightarrow \Fbrs(\base \times \base')$ is then seen to be both equivalent to pullback along the projection $p\colon \base \times \base' \rightarrow \base$ and a left adjoint.
  We will write $p_*$ for its right adjoint.
  The same result applies for Segal $\base$-$\infty$-categories using the pattern
  \[
    \mathrm{CatPatt}^{\otimes}(\base) \deq \prn{\base, \mathrm{all}, \mathrm{all}, \cbr{\base^{\el}_{O/}}_{O \in \base}}.\qedhere
  \]
\end{observation}

\subsection{\tcT-operads and \tI-operads}\label{I operads subsection} 
We're finally ready to specialize to equivariant operads.
Fix $\cT$ an atomic orbital $\infty$-category.
\begin{definition}\label{Opt definition}
    The $\infty$-category of $\cT$-operads is
    \[
        \Op_{\cT} := \Fbrs(\Span(\FF_{\cT})).
    \]
    More generally, when $I \subset \FF_{\cT}$ is a weak indexing category, the \emph{$\infty$-category of $I$-operads} is 
    \[
        \Op_{I}:= \Fbrs(\Span_I(\FF_{\cT})).
    \]
    The associated localization functors are $L_{\Op_{\cT}}\colon \Cat_{/\Span(\FF_{\cT})}^{\Int-\cocart} \rightarrow \Op_{\cT}$ and $L_{\Op_I}\colon \Cat_{/\Span_I(\FF_{\cT})}^{\Int-\cocart}  \rightarrow \Op_I$.
\end{definition}

By \cref{Overcategory of fibrous patterns prop}, if $\cO^{\otimes}$ is an $I$-operad, then it has a natural pattern structure such that $\cO^{\otimes} \rightarrow \Span_I(\FF_{\cT})$ is a morphism of patterns;
the inert morphisms are cocartesian lifts of backwards maps, and the active maps are \emph{arbitrary} lifts of forwards maps.

\begin{definition}
    If $\cO^{\otimes},\cP^{\otimes}$ are $I$-operads, then an \emph{$\cO$-algebra in $\cP$} is a map of $I$-operads $\cO^{\otimes} \rightarrow \cP^{\otimes}$;
    the $\infty$-category of $\cO$-algebras in $\cP$ is written
    \[
      \Alg_{\cO}(\cP) \deq \Fun^{\Int-\cocart}_{/\Span_I(\FF_{\cT})}(\cO^{\otimes},\cP^{\otimes}).\qedhere
    \]
\end{definition}
\begin{remark}
  It follows by unwinding definitions that $\Map_{\Op_I}(\cO^{\otimes},\cP^{\otimes}) \simeq \Alg_{\cO}(\cP)^{\simeq}$.
\end{remark}

The following proposition verifies that the pushforward functor $\Op_I \rightarrow \Op_{\cT}$ is simply given by postcomposition along the canonical functor $\iota_I^{\cT}\cln \Span_I(\FF_{\cT}) \rightarrow \Span(\FF_{\cT})$ (c.f. \cite[Ex~2.4.7]{Nardin}).
\begin{proposition}\label{Definition of NIinfty proposition}
  Let $I \subset \FF_{\cT}$ be a pullback-stable replete subcategory.
    Then, the functor
    \[
        \cN_{I \infty}^{\otimes} \deq \prn{\Span_I(\FF_{\cT}) \xrightarrow{\pi_I} \Span(\FF_{\cT})}
    \]
  presents a $\cT$-operad if and only if $I$ is a weak indexing category.
\end{proposition}
We will delay the proof of this until \cpageref{Proof of ninfty proposition}.
If $\cO^{\otimes} \simeq \cN_{I\infty}^{\otimes}$ arises from \cref{Definition of NIinfty proposition}, we say that $\cO^{\otimes}$ is a \emph{weak $\cN_\infty$ $\cT$-operad} (or simply a weak $\cN_\infty$-operad), and if $I$ is an indexing category, then we say that $\cN_{I \infty}^{\otimes}$ is an \emph{$\cN_\infty$-operad}; in either case, we write
\[
  \CAlg_I(\cC) \deq \Alg_{\cN_{I \infty}}(\cC)
\]
for the $\infty$-category of \emph{$I$-commutative algebras in $\cC$}.
This fits nicely into the theory of $\cT$-operads:
\begin{corollary}
  Pushforward along $\iota_I^{\cT}$ yields an equivalence of $\infty$-categories $\Op_I \simeq \Op_{\cT, /\cN_{I \infty}^{\otimes}}$.
\end{corollary}
\begin{proof}
  Unwinding definitions, this is \cref{Overcategory of fibrous patterns prop} for $\fO \deq \cN^{\otimes}_{I \infty}$ and $\base \deq \Span(\FF_{\cT})$.
\end{proof}
In \cref{0-operads subterminal,Unslicing ff corollary} we will show that the morphism $\cN_{I \infty}^{\otimes} \rightarrow \Comm_{\cT}^{\otimes}$ is monic, so pushforward $\Op_{I} \rightarrow \Op_{\cT}$ is fully faithful.
Until then, we will largely consider $\Op_I$ and $\Op_{\cT}$ separately.

\begin{example}\label{Commf example}
    The terminal $\cT$-operad is presented by $\Comm^{\otimes}_{\cT} = \prn{\Span(\FF_{\cT}) \xrightarrow{\id} \Span(\FF_{\cT})}$, and hence it is an $\cN_\infty$-operad;
    we write $\CAlg_{\cT}(\cC) \deq \CAlg_{\FF_{\cT}}(\cC)$, and call these \emph{$\cT$-commutative algebras}.
    For any $\cT$-operad $\cO^{\otimes}$, pullback along the unique map $\cO^{\otimes} \rightarrow \Comm_{\cT}^{\otimes}$ determines a unique natural transformation
    \[
      \CAlg_{\cT}(\cC) \rightarrow \Alg_{\cO}(\cC),
    \]
    so we view $\cT$-commutative algebras as the universal $\cT$-equivariant algebraic structure.
\end{example}

\begin{definition}\label{CatI definition}
  If  $\cO^{\otimes}$ is an $I$-operad, then the \emph{$\infty$-category of small $\cO$-monoidal $\infty$-categories is}
  \[
    \Cat^{\otimes}_{\cO} \deq \Seg_{\cO^{\otimes}}(\Cat).
  \]
  If $\cC^{\otimes},\cD^{\otimes}$ are $\cO$-monoidal $\infty$-categories, then the \emph{$\infty$-category of $\cO$-monoidal functors from $\cC^{\otimes}$ to $\cD^{\otimes}$} is
  \[
    \Fun^{\otimes}_{\cO}(\cC,\cD) \deq \Fun^{\cocart}_{/\cO^{\otimes}}\prn{\cC^{\otimes}, \cD^{\otimes}}.
  \]
  A \emph{lax $\cO$-symmetric monoidal functor} is a functor of $I$-operads $\cC^{\otimes} \rightarrow \cD^{\otimes}$ over $\cO^{\otimes}$.
\end{definition}
In particular, we write $\Cat_I^{\otimes} \deq \Cat_{\cN_{I \infty}}^{\otimes}$ and $\Cat_{\cT}^{\otimes} \deq \Cat_{\FF_{\cT}}^{\otimes}$;
\cref{CMon I cat corollary} constructs an equivalence
\[
  \Cat_I^{\otimes} \simeq \CMon_I(\Cat).
\]
Note that a lax $I$-symmetric monoidal functor is an $I$-symmetric monoidal functor if and only if it is a morphism in $\Cat_I^{\otimes}$, i.e. if and only if it preserves cocartesian lifts for arbitrary maps in $\Span_I(\FF_{\cT})$.

\begin{remark}
  \cref{Opt definition,CatI definition} appear to depend on $I$, but we omit $I$ from our notation;
  we will show in \cref{0-operads subterminal} that $\cN_{I \infty}^{\otimes} \rightarrow \Comm_{\cT}^{\otimes}$ is monic, obviating this dependence.
\end{remark}

The rest of this subsection proceeds through a series of vignettes.
In \cref{Structure subsubsection}, we explicitly describe the structure of $I$-operads through the lens of their underlying $\cT$-categories and their \emph{structure spaces}.
Following this, in \cref{T-category subsubsection}, we summarize the comparison between $\cT$-operads and \cite{Nardin}'s $\cT$-operads, and we derive $\cT$-$\infty$-categorical lifts of $\Op_{\cT}$ and $\Alg_{\cO}(\cC)$.
Then, in \cref{Envelopes subsubsection}, we summarize the specialization of the Segal envelope to $I$-operads.
Finally, in \cref{Trivial subsubsection} we describe the family of \emph{trivial $\cT$-operads}, which form the left adjoint to the underlying $\cT$-$\infty$-category.

\subsubsection{The structure of $I$-operads}\label{Structure subsubsection}
The Segal conditions for fibrous $\Span(\FF_{\cT})$-patterns were characterized in \cite{Barkan} in the case $\cT = \cO_G$;
we generalize this to weak indexing systems over general atomic orbital $\infty$-categories in \cref{Span segal condition lemma}, and summarize the results here.
\begin{construction}\label{Underlying construction}
    Given $\pi_{\cO}:\cO^{\otimes} \rightarrow \Span_I(\FF_{\cT})$ an $I$-operad and $S \in \FF_{\cT}$ a finite $\cT$-set, we define
    \[
        \cO_S := \pi^{-1}_{\cO}(S).
    \]
    Inert arrows endow on $(\cO_V)_{V \in \cT}$ the structure of a $\cT$-$\infty$-category $U(\cO^{\otimes})$, formally given by the pullback
    \[
        \begin{tikzcd}
            \Tot U(\cO^{\otimes}) \arrow[r] \arrow[d] \arrow[rd,"\lrcorner" very near start, phantom]
            & \cO^{\otimes} \arrow[d]\\
            \cT^{\op} \arrow[r,hook]
            & \Span(\FF_{\cT})
        \end{tikzcd}
    \]
    We call this the \emph{underlying $\cT$-$\infty$-category of $\cO^{\otimes}$}, and refer to it as $\cO$ when this won't cause confusion.
\end{construction}

\begin{proposition}\label{GOperads conditions}  
    A functor $\pi:\cO^{\otimes} \rightarrow \Span_I(\FF_{\cT})$ is an $I$-operad if and only if the following are satisfied:
  \begin{enumerate}[label={(\alph*)}]
    \item \label[condition]{Inert cocartesian lifts} $\cO^{\otimes}$ has $\pi$-cocartesian lifts for backwards maps in $\Span_I(\FF_{\cT})$;
    \item \label[condition]{Color Segal condition} (Segal condition for colors) for every $S \in \FF_\cT$, cocartesian transport along the $\pi$-cocartesian lifts lying over the inclusions $\prn{S \leftarrow U = U \mid U \in \mathrm{Orb}(S)}$ together induce an equivalence
      \[
        \cO_S \simeq \prod_{U \in \mathrm{Orb}(S)} \cO_U; 
      \]
    \item \label[condition]{Multimorphism Segal condition} (Segal condition for multimorphisms) for every map of orbits $T \rightarrow S$ in $I$ and pair of objects $(\bC,\bD) \in \cO_T \times \cO_U$, postcomposition with the $\pi$-cocartesian lifts $\bD \rightarrow D_U$ lying over the inclusions 
    $\left(S \leftarrow U = U \mid U \in \mathrm{Orb}(S)\right)$ induces an equivalence
      \[
        \Map^{T \rightarrow S}_{\cO^{\otimes}}(\bC,\bD) \simeq \prod_{U \in \mathrm{Orb}(S)} \Map^{T \leftarrow T_U \rightarrow  U}_{\cO^{\otimes}}(\bC,D_U).
      \]
      where $T_U \deq T \times_S U$. 
  \end{enumerate}
  Furthermore, a cocartesian fibration $\pi:\cO^{\otimes} \rightarrow \Span_I(\FF_{\cT})$ is an $I$-operad if and only if its unstraightening $\Span_I(\FF_{\cT}) \rightarrow \Cat$ is an $I$-symmetric monoidal category.
\end{proposition}
\begin{proof}
  Each of our conditions nearly matches with that of \cref{Fibrous patterns definition}, with the exception being that we evaluate the limits on the sub-diagram $\Orb(S) \subset \Span_I(\FF_{\cT})^{\el}_{S/}$;
  we show in \cref{Segal condition lemma} that this is an initial subcategory, proving the proposition.
\end{proof}

\begin{remark}
  Cocartesian lifts over backwards maps furnish an equivalence 
  \[
    \Map^{T \leftarrow T_U \rightarrow U}_{\cO^{\otimes}}(\bC,D_U) \simeq \Map^{T_U \rightarrow U}_{\cO^{\otimes}}(\bC_{T_U},D_U),
  \]
  where $\bC_{T_U} \in \cO_{T_U}$ is the $T_U$-tuple of colors underlying $\bC$.
  Hence in the presence of \cref{Inert cocartesian lifts,Color Segal condition}, \cref{Multimorphism Segal condition} may equivalently stipulate that the map 
      \[
        \Map^{T \rightarrow S}_{\cO^{\otimes}}(\bC,\bD) \rightarrow \prod_{U \in \mathrm{Orb}(S)} \Map^{T_U \rightarrow  U}_{\cO^{\otimes}}(\bC_{T_U},D_U)
      \]
  is an equivalence.
  We will generally prefer this version, as the data of a $\cT$-operad is most naturally viewed as living over the \emph{active} (i.e. forward) maps.
\end{remark}

\begin{remark}
  Practitioners of \cite[Def~2.1.10]{HA} should note that, by \cref{Enough to assert essentially surjective remark}, we may weaken \cref{Color Segal condition} to assert only that cocartesian transport induces a $\pi_0$-surjection $\cO_S \rightarrow \prod\limits_{U \in \Orb(S)} \cO_U$.
\end{remark}

We're finally ready to prove \cref{Definition of NIinfty proposition}.
\begin{proof}[Proof of \cref{Definition of NIinfty proposition}]\label{Proof of ninfty proposition}
  Note that \cref{Restriction stable condition,Automorphism condition} of \cref{Windex definition} are true by assumption (they were forced on us in order to make $\Span_I(\FF_{\cT})$ definable).
  We verify the conditions of \cref{GOperads conditions} for $\cT$-operads.
  
  Note that $\Span_I(\FF_{\cT})$ has \emph{unique} lifts for backwards maps, so condition (a) follows always.
  Furthermore, $\Span_I(\FF_{\cT})$ always satisfies condition (b) by construction.
  Lastly, by unwinding definitions and noting that there exists a map of spaces $X \rightarrow Y \times \emptyset = \emptyset$ if and only if $X$ is empty, \cref{Reduction to maps to orbits observation} implies that (c) is equivalent to \cref{Windex segal condition}.
\end{proof}

Using \cref{GOperads conditions}, we gain access to the \emph{structure spaces} of $\cT$-operads.
\begin{notation}
  Following \cite{Bonventre-color} We will refer to tuples $(V \in \cT,S \in \FF_V,(\bC;D) \in \cO_S \times \cO_V)$ as \emph{$\cO$-profiles}, and we will often abusively refer to the profile $(V,S, (\bC;D))$ as $(\bC;D)$.
  Additionally, if $\psi\colon T \rightarrow S$ is a map of finite $V$-sets, we will write $T_U \deq T \times_U S$;
we refer to a \emph{composable datum over $\psi$} as the data of a $\cO$-profile $(V,S,(\bC;D))$ together with an $S$-tuple of profiles $\prn{(U,T_U,(\bB_U,C_U))}_{U \in \Orb(S)}$;
  in this case, $(\bB_U)_{U \in \Orb(S)}$ assemble into a $T$-tuple $\bB$, and we refer to $(V,T,(\bB;D))$ as the \emph{composite $\cO$-profile}.
\end{notation}

\begin{construction}
    Let $\cO^{\otimes}$ be a $\cT$-operad.
    Given an $\cO$-profile $(V, S, (\bC;D))$, we write
    \[
      \cO(\bC;D) \deq \Map_{\cO}^{\Ind_V^{\cT} S \rightarrow V}(\bC,D).
    \]
    Similarly, given $S \in \FF_V$, we write
    \[
      \cO(S) := \coprod_{(\bC,D) \in \cO_{S} \times \cO_V} \cO(\bC;D);
    \]
    we refer to this is the \emph{space of $S$-ary operations in $\cO$}.\qedhere
\end{construction}

This simplifies in a particular setting:
\begin{definition}\label{Many definitions definition}
    A $\cT$-operad $\cO^{\otimes}$ is:
    \begin{itemize}
        \item \emph{at most one-colored} if $\cO_V \in \cbr{\emptyset,*}$ for all $V \in \cT$, i.e. $\cO(*_V) \in \cbr{\emptyset,*}$ for all $V \in \cT$,
        \item \emph{at least one-colored} if $\cO_V \neq \emptyset$ for all $V \in \cT$, i.e. $\cO(*_V) \neq \emptyset$ for all $V \in \cT$, and
         \item \emph{one-colored} if $\cO^{\otimes}$ is at least one-colored and at-most one colored.
    \end{itemize}    
  We denote the associated full subcategories by $\Op_{\cT}^{\oc} \subset \Op_{\cT}^{\geq \oc}, \Op_{\cT}^{\leq \oc} \subset \Op_{\cT}$.
\end{definition}
We acquire a simpler description for one-color $\cT$-operads:
they are functors $\pi\colon \cO^{\otimes} \rightarrow \Span(\FF_{\cT})$ with:
\begin{enumerate}[label={(\alph*')}]
  \item cocartesian lifts of backwards maps,
  \item contractible fibers, and
  \item for whom cocartesian transport induces an equivalence
    \[
      \Map_{\cO}^{T \rightarrow S}(iT, iS) \simeq \prod_{U \in \Orb(S)}\cO(T_U), 
    \]
    where $\pi^{-1}(T) = \cbr{iT}$ and $T_U = T \times_S U \rightarrow U$, considered as a $U$-set.
\end{enumerate}
For most applications, algebras over one-color $\cT$-operads suffice.
For now, we describe the general setting.
\begin{construction}\label{T-operad structure maps}
  Given $\cO^{\otimes} \in \Op_{\cT}$ and $\prn{S,V,(\bC;D)}$ an $\cO$-profile,  for any $T \leftarrow \Ind_V^{\cT} S$, we have an equivalence
  \[
      \cO(\bC;D) \simeq \Map^{T \leftarrow \Ind_V^{\cT} S \rightarrow V}_{\pi_{\cO}}(\bC; D)
  \]
  due to the existence of cocartesian lifts for inert morphisms. 
  Given a map $U \rightarrow V$ in $\cT$ and a finite $V$-set $S \in \FF_V$, postcomposition with the cocartesian lift of the backwards map $V \leftarrow U = U$ yields a restriction map
  \begin{equation}\label{Restriction map}
      \begin{tikzcd}
          \cO(\bC;D) \arrow[d,"\simeq" labl, phantom] \arrow[r,"\Res_U^V"]
          & \cO\prn{\Res_U^V \bC; \; \Res_U^V D} \arrow[d,"\simeq" labl, phantom]\\
          \Map_{\pi_{\cO}}^{\Ind_V^{\cT} S \rightarrow V}(\bC,D) \arrow[r]
          &  \Map_{\pi_{\cO}}^{\Ind_V^{\cT} S \leftarrow \Ind^{\cT}_V S \times_V U \rightarrow U}\prn{\Res_U^V \bC, \; \Res_U^V D}
      \end{tikzcd}
  \end{equation}
  where the right hand side corresponds with the profile $\prn{U, \Res_U^V S, \prn{\Res_U^V \bC;\; \Res_U^V D}}$ induced by restriction.

  Moreover, given a composable datum $\prn{(\bC;D),(\bB_U;C_U)_{U \in \Orb(S)}}$ lying over a map of $V$-sets $\varphi_{TS}\colon T \rightarrow S$, writing $\varphi_{TV}$ for the structure map of $T$, composition in $\cO^{\otimes}$ restricts to a map 
  \begin{equation}\label{Composition map}
      \begin{tikzcd}
          \cO(\bC;D) \times \dprod_{U \in \Orb(S)} \cO(\bB_U;C_U) \arrow[r,"\gamma"] \arrow[d, "\simeq" labl, phantom]
          & \cO(\bB;D) \arrow[d,"\simeq" labl, phantom]\\
          \Map_{\cO^{\otimes}}^{\varphi_{SV}}(\bC, D) \times \Map_{\pi_{\cO}}^{\varphi_{TS}}(\bB,\bC)  \arrow[r, "\circ"]
          & \Map_{\cO^{\otimes}}^{\varphi_{TV}}(\bB, D)
      \end{tikzcd}
  \end{equation}
  Lastly, define the \emph{color preserving automorphism group} to be the subgroup $\Aut_{V}(\bC) \subset \Aut_V(S)$ consisting of automorphisms $\sigma$ such that $C_U \simeq C_{\sigma U}$ for all $U \in \Orb(S)$.
  Note that $\pi_{\cO}$-cocartesian lifts of $\Aut_V(\bC)$ preserve $\bC$;
  cocartesian transport then yields an action
  \begin{equation}\label{Sigma action map}
      \rho_S:\Aut_V(\bC) \times \cO(\bC;D) \longrightarrow \cO(\bC;D).
  \end{equation}
  We refer to $\Res_U^V$ as \emph{restriction}, $\gamma$ as the \emph{composition}, and $\rho_S$ as \emph{$\Sigma$-action}.\qedhere
\end{construction}

\begin{example}\label{Ninfty sseq example}
  Let $I$ be a weak indexing category.
  Recall the example $\cN_{I \infty}^{\otimes} = \prn{\Span_I(\FF_{\cT}) \rightarrow \Span(\FF_{\cT})}$ of \cref{Definition of NIinfty proposition}.
  Then, it follows by definition that $U \cN_{I \infty}^{\otimes} \simeq *_{c(I)}$;
  that is, $\cN_{I \infty}$ always has at most one color, and it has one color if and only if $I$ has one color in the sense of \cite{Windex}.

  Moreover, we have
  \[
    \cN_{I \infty}(S) \simeq \begin{cases}
      * & S \in \FF_{I,V};\\
      \emptyset & S \not \in \FF_{I,V}.
    \end{cases}
  \]
  Each of the maps $\Res_U^V, \gamma, \rho_S$ are uniquely determined by their domain and codomain.
\end{example}

For the following observation, we restrict to the one-color case simply for notational clarity;
the general case follows identically.
\begin{observation}\label{Discretization remark}
  The structures of \cref{Restriction map,Composition map,Sigma action map} are compatible in the following ways:
  \begin{enumerate}
    \item The restriction maps are Borel $\Aut_V(S)$-equivariant, i.e. the following commutes:
      \[\begin{tikzcd}[column sep={6.5em, between origins}]
          {\cbr{\text{cocart lifts of } \Aut_V(S)} \times \Map^{\varphi_{SV}}_{\cO^{\otimes}}(iS,iV)} &&&& {\Map^{\varphi_{SV}}_{\cO^{\otimes}}(iS,iV)} \\
	& {\Aut_V(S) \times \cO(S)} && {\cO(S)} \\
	& {\Aut_W(\Res_W^V S) \times \cO(\Res_W^V S)} && {\cO(\Res_W^V S)} \\
  {\cbr{\text{cocart lifts of } \Aut_W(\Res_W^V S)} \times \Map^{\varphi_{SV}}_{\cO^{\otimes}}(i\Res_W^V S,iW)} &&&& {\Map^{\varphi_{SV}}_{\cO^{\otimes}}(i\Res_W^V S,iW)}
	\arrow["\circ"{description}, from=1-1, to=1-5]
	\arrow[Rightarrow, no head, from=1-1, to=2-2]
  \arrow["{\Res_W^V}"{description}, from=1-1, to=4-1, curve={height=30pt}]
  \arrow["{\Res_W^V}"{description}, from=1-5, to=4-5, curve={height=-30pt}]
  \arrow[from=2-2, to=2-4, "\rho"]
  \arrow[from=2-2, to=3-2, "\Res_W^V"]
	\arrow[Rightarrow, no head, from=2-4, to=1-5]
  \arrow[from=2-4, to=3-4, "\Res_W^V"]
  \arrow[from=3-2, to=3-4, "\rho"]
	\arrow[Rightarrow, no head, from=3-4, to=4-5]
	\arrow[Rightarrow, no head, from=4-1, to=3-2]
	\arrow["\circ"{description}, from=4-1, to=4-5]
\end{tikzcd}\]
    Here we write $\cO_S = \cbr{iS}$.
    \item The composition maps are Borel $\Aut_V(S) \times \prod\limits_{U \in \Orb{S}} \Aut_U(T_U)$-equivariant in an analogous way.
    \item The identity map on $*_V$ yields an element $1_V \in *_V$ which is taken to $1_V$ by $\Res_U^V$.
    \item The composition maps are unital, i.e. the following commutes.
      \[
        \begin{tikzcd}[column sep={5em, between origins}]
          {\Map^{\varphi_{SV}}_{\cO^{\otimes}}(iS,iV)} &&&&& {\Map^{\varphi_{SV}}_{\cO^{\otimes}}(iS,iV) \times \Map^{\id}_{\cO^{\otimes}}(iS,iS)} \\
          & {\cO(S)} &&& {\cO(S) \times \mspc \prod\limits_{U \in \Orb(S)} \mspc \cO(*_U)} \\
          & {\cO(*_V) \times \cO(S)} &&& {\cO(S)} \\
          {\Map^{\id}_{\cO^{\otimes}}(iV,iV) \times \Map^{\varphi_{SV}}_{\cO^{\otimes}}(iS,iV) } &&&&& {\Map^{\varphi_{SV}}_{\cO^{\otimes}}(iS,iV)}
          \arrow["{(\id, \cbr{\id})}"{description}, from=1-1, to=1-6]
          \arrow[Rightarrow, no head, from=1-1, to=2-2]
          \arrow["{(\cbr{\id},\id)}"{description}, curve={height=30pt}, from=1-1, to=4-1]
          \arrow[Rightarrow, no head, from=1-6, to=2-5]
          \arrow["\circ"{description}, curve={height=-30pt}, from=1-6, to=4-6]
          \arrow["{(\id, (\cbr{1_U}))}", from=2-2, to=2-5]
          \arrow["{(\cbr{1_V}, \id)}"', from=2-2, to=3-2]
          \arrow[Rightarrow, no head, from=2-2, to=3-5]
          \arrow["\gamma"{description}, from=2-5, to=3-5]
          \arrow["\gamma"{description}, from=3-2, to=3-5]
          \arrow[Rightarrow, no head, from=3-5, to=4-6]
          \arrow[Rightarrow, no head, from=4-1, to=3-2]
          \arrow["\circ"{description}, from=4-1, to=4-6]
        \end{tikzcd}
      \]
    \item The map $\gamma$ is compatible with restriction, i.e. given a composable pair of morphisms
      \[\begin{tikzcd}[row sep=tiny, column sep = large]
        {\Ind_V^{\cT}T} & {\Ind_V^{\cT} S} & V
	\arrow["{\varphi_{SV}}"{description}, from=1-2, to=1-3]
	\arrow["{\varphi_{TS}}"{description}, from=1-1, to=1-2]
  \arrow["{\varphi_{TV}}"{description}, from=1-1, to=1-3, curve = {height=-20pt}]
\end{tikzcd}\]
and $W \rightarrow V$ a map in $\cT$, the following diagram commutes, where $T \deq \coprod_{U}^S T_U$: 
    \[
      \begin{tikzcd}[column sep={4.8em, between origins}]
        {\Map_{\cO^{\otimes}}^{\varphi_{SV}}(iS, iV) \times \Map_{\cO^{\otimes}}^{\varphi_{TS}}(iT, iS)} &&&&& {\Map_{\cO^{\otimes}}^{\varphi_{TV}}(iT, iV)} \\
        & {\cO(S) \times \mspc \prod\limits_{U \in \Orb(S)} \mspc \cO(T_U)} &&& {\cO(T)} \\
        & {\cO\prn{\Res_W^V S} \times \mspc \prod\limits_{U' \in \Orb(\Res^V_WS)} \mspc \cO(T_{U'})} &&& {\cO\prn{\Res_W^V T}} \\
        {\Map_{\cO^{\otimes}}^{\Res_W^V\varphi_{SV}}\prn{i\Res_W^VS, iW} \times \Map_{\cO^{\otimes}}^{\Res_W^V\varphi_{TS}}(i\Res_W^V T, i\Res_W^VS)} &&&&& {\Map_{\cO^{\otimes}}^{\Res_W^V\varphi_{TV}}(i\Res_W^VT, iW)}
        \arrow["\circ"{description}, from=1-1, to=1-6]
        \arrow[Rightarrow, no head, from=1-1, to=2-2]
        \arrow["{\Res_W^V}"{description}, shift right=10, curve={height=30pt}, from=1-1, to=4-1]
        \arrow[Rightarrow, no head, from=1-6, to=2-5]
        \arrow["{\Res_W^V}"{description}, shift left=5, curve={height=-30pt}, from=1-6, to=4-6]
        \arrow[from=2-2, to=2-5, "\gamma"]
        \arrow[from=2-2, to=3-2, "\Res_V^W"]
        \arrow[from=2-5, to=3-5, "\Res_V^W"]
        \arrow[from=3-2, to=3-5, "\gamma"]
        \arrow[Rightarrow, no head, from=3-5, to=4-6]
        \arrow[Rightarrow, no head, from=4-1, to=3-2]
        \arrow["\circ"{description}, from=4-1, to=4-6]
      \end{tikzcd}
    \]
  \item The composition maps are associative, i.e. given a collection of maps and composites
      \[\begin{tikzcd}[column sep=huge, row sep=tiny]
          {\Ind_V^{\cT}R} & {\Ind_V^{\cT} T} & {\Ind_V^{\cT}S} & V,
	\arrow["{\varphi_{RT}}"{description}, from=1-1, to=1-2]
	\arrow["{\varphi_{RS}}"{description}, curve={height=20pt}, from=1-1, to=1-3]
	\arrow["{\varphi_{RV}}"{description}, curve={height=-32pt}, from=1-1, to=1-4]
	\arrow["{\varphi_{TS}}"{description}, from=1-2, to=1-3]
	\arrow["{\varphi_{TV}}"{description}, curve={height=-20pt}, from=1-2, to=1-4]
	\arrow["{\varphi_{SV}}"{description}, from=1-3, to=1-4]
\end{tikzcd}\]
the following commutes, where $R \deq \coprod_U^S \coprod_W^{T_U} R_W$:
  \[
    \begin{tikzcd}[column sep={4.5em, between origins}]
        {\Map_{\cO^{\otimes}}^{\varphi_{SV}}(iS,iV) \times \Map^{\varphi_{TS}}_{\cO^{\otimes}}(iT,iS) \times \Map_{\cO^{\otimes}}^{\varphi_{RT}}(iR,iT)} 
        &&&&&& {\Map_{\cO^{\otimes}}^{\varphi_{TV}}(iT,iV) \times \Map_{\cO^{\otimes}}^{\varphi_{RT}}(iR,iT)} 
                \\
        & {\prn{\cO(S) \times \mspc \prod\limits_{U \in \Orb(S)} \mspc \cO(T_U)} \times \mspc \prod\limits_{\stackrel{U \in \Orb(S)}{W \in \Orb(T_U)}} \mspc  \cO(R_W)} 
        &&&& {\cO(T) \times \mspc \prod\limits_{W \in \Orb(T)} \mspc \cO(R_W)} 
                \\
                & {\cO(S) \times \mspc \prod\limits_{U \in \Orb(S)} \prn{\cO(T_U) \times \mspc \prod\limits_{W \in \Orb(T_U)} \mspc \cO(R_W)}} 
        &&&& {\cO\prn{\coprod\limits_W^T R_W}} 
                \\
        & {\cO(S) \times \mspc \prod\limits_{U \in \Orb(S)} \mspc \cO\prn{\coprod\limits_W^{T_U} R_W}} 
        &&&& {\cO\prn{R}} 
                \\
        {\Map_{\cO^{\otimes}}^{\varphi_{SV}}(iS,iV) \times \Map_{\cO^{\otimes}}^{\varphi_{RS}}(iR,iS)} 
        &&&&&& {\Map_{\cO^{\otimes}}^{\varphi_{RV}}(iR,iV)}
        \arrow["\circ"{description}, from=1-1, to=1-7]
        \arrow[Rightarrow, no head, from=1-1, to=2-2]
        \arrow["\circ"{description}, shift right=20, curve={height=45pt}, from=1-1, to=5-1]
        \arrow[Rightarrow, no head, from=1-7, to=2-6]
        \arrow["\circ"{description}, shift left=5, curve={height=-45pt}, from=1-7, to=5-7]
        \arrow["\gamma", from=2-2, to=2-6]
        \arrow["\gamma", from=2-6, to=3-6]
        \arrow[Rightarrow, no head, from=3-2, to=2-2]
        \arrow["\gamma", from=3-2, to=4-2]
        \arrow["\gamma", from=4-2, to=4-6]
        \arrow[Rightarrow, no head, from=4-2, to=5-1]
        \arrow[Rightarrow, no head, from=4-6, to=3-6]
        \arrow[Rightarrow, no head, from=4-6, to=5-7]
        \arrow["\circ"{description}, from=5-1, to=5-7]
      \end{tikzcd}
    \]
  \end{enumerate}
  Thus, passing to the homotopy category, the data of a $\cT$-operad supplies a discrete genuine $\cT$-operad in $\ho \cS$ in the sense of \cref{Discrete genuine definition}.
\end{observation}

\subsubsection{The $\cT$-$\infty$-category of $\cT$-operads}\label{T-category subsubsection}
Recall the map of algebraic patterns $\varphi\cln \Tot \uFF_{\cT,*} \rightarrow \Span(\FF_{\cT})$ of \cref{Span comparison map}.  
By assumption, if $\cO^{\otimes}$ is a fibrous $\Tot \uFF_{\cT,*}$-pattern, it possesses cocartesian lifts over \emph{all} morphisms in the composite $\cO^{\otimes} \rightarrow \Tot \uFF_{\cT,*} \rightarrow \cT^{\op}$.
Thus, fibrous $\Tot \uFF_{\cT,*}$-patterns possess total \emph{$\cT$-$\infty$-categories};
we refer to the associated functor as 
\[
  \Tot_{\cT}\cln \Op_{\cT} \rightarrow \Fbrs(\Tot \uFF_{\cT,*}) \rightarrow \Cat_{\cT}.
\]
In \cref{Pointed prop,Span fibrous corollary,Two models for segal objects}, we prove the following generalization of the contents of \cite[\S5.2]{Barkan}, which identifies our $\cT$-operads with those of \cite{Nardin}.

\begin{proposition}\label{Operads are fibrous theorem}\label{Familiar structures corollary}
  Suppose $\cT$ is an atomic orbital $\infty$-category.
  Then, pullback along $\varphi\cln\Tot \uFF_{\cT,*} \rightarrow \Span(\FF_{\cT})$ implements equivalences of categories 
   \begin{align*} 
     \Cat_{\cO}^{\otimes} &\simeq \Seg_{\Tot \Tot_{\cT} \cO^{\otimes}}\prn{\Cat};\\
     \Op_{\cT} &\simeq \Fbrs\prn{\Tot \underline{\FF}_{\cT,*}}.
  \end{align*}
  Moreover, $\Fbrs\prn{\Tot \uFF_{\cT,}}$ is equivalent to the $\infty$-category of $\cT$-$\infty$-operads of \cite{Nardin} and $\Seg_{\Tot \Tot_{\cT} \cO^{\otimes}}(\Cat)$ is equivalent to the $\infty$-category of small $\cO$-symmetric monoidal $\cT$-$\infty$-categories of \cite{Nardin}.
\end{proposition}
This enables us to make the following definition.
\begin{definition}
  Let $\cO^{\otimes}, \cP^{\otimes}$ be $\cT$-operads,
  Then, the  \emph{$\cT$-$\infty$-category of $\cO$-algebras in $\cP$} is the full subcategory
  \begin{align*}
    \uAlg_{\cO}(\cP) 
    &\deq \uFun^{\Int-\cocart}_{\cT,/\uFF_{\cT,*}}\prn{\Tot_{\cT} \cO^{\otimes} , \Tot_{\cT} \cP^{\otimes}}\\
    &\subset \uFun_{\cT,/\uFF_{\cT,*}}\prn{\Tot_{\cT} \cO^{\otimes},\Tot_{\cT} \cP^{\otimes}}  
  \end{align*}
  with $V$-values spanned by the $V$-functors $\Res_V^{\cT} \Tot_{\cT} \cO^{\otimes} \rightarrow \Res_V^{\cT} \Tot_{\cT} \cP^{\otimes}$ preserving cocartesian lifts over inert arrows in $\uFF_{V,*}$.
\end{definition}

We lift $\Op_{\cT}$ to a $\cT$-$\infty$-category by the following.
\begin{definition}
  We show in \cref{Ind is Segal} that $\Span\prn{\Ind_U^V}\cln \Span(\FF_{U}) \rightarrow \Span(\FF_{V})$ is a Segal morphism for all maps $U \rightarrow V$ in $\cT$.
  We refer to the resulting $\cT$-$\infty$-category
  \[
    \uOp_{\cT}\cln \cT^{\op} \xrightarrow{\Span(\FF_{(-)})}  \AlgPatt^{\mathrm{SE},\Seg, \op} \xrightarrow{\Fbrs} \Cat.
  \]
  as the \emph{$\cT$-$\infty$-category of $\cT$-operads,} where $\AlgPatt^{\mathrm{SE},\Seg} \subset \AlgPatt$ is the non-full subcategory of soundly extendable patterns and Segal morphisms.
\end{definition}
\begin{observation}\label{Restrictions fibrous remark}
  The $V$-value of $\uOp_\cT$ is $\Op_{V} \deq \Op_{\cT_{/V}}$;
  the restriction functor $\Res_U^V\cln \Op_{V} \rightarrow \Op_{U}$ is implemented by the pullback
\[
    \begin{tikzcd}
      \Res_U^V \cO^{\otimes} \arrow[r] \arrow[d] \arrow[rd, "\lrcorner" very near start, phantom]
      & \cO^{\otimes} \arrow[d]\\
      \Span(\FF_U) \arrow[r]
      & \Span(\FF_V).
    \end{tikzcd}
  \]
  with bottom functor is $\Span(\Ind_U^V)$.
  Moreover, $\varphi\colon \Tot \uFF_{\cT,*} \rightarrow \Span(\FF_{\cT})$ is natural in $\cT$, i.e. it yields a commutative diagram
  \[
    \begin{tikzcd}
      {\Tot \uFF_{U,*}} & {\Span(\FF_{U})} \\
      {\Tot \uFF_{V,*}} & {\Span(\FF_{V})}
      \arrow[from=1-1, to=1-2]
      \arrow[from=1-1, to=2-1]
      \arrow[from=1-2, to=2-2]
      \arrow[from=2-1, to=2-2]
    \end{tikzcd}  
  \]
  Thus $\varphi^*$ identifies $\Res_U^V\colon \Op_{V} \rightarrow \Op_W$ with pullback along $\Tot \uFF_{U,*} \rightarrow \Tot \uFF_{V,*}$.
\end{observation}

\begin{observation}\label{Alg values}
  Via \cref{Operads are fibrous theorem}, we find that $\Gamma^{\cT} \uAlg_{\cO}(\cP) \simeq \Alg_{\cO}(\cP)$.
  Furthermore, we find that
  \[
    \uAlg_{\cO}(\cP)_V \simeq \Fun^{\Int-\cocart}_{/\Span(\FF_V)}(\Res_V^{\cT} \cO^{\otimes}, \Res_V^{\cT} \cP^{\otimes}) \simeq \Alg_{\Res_V^{\cT} \cO}(\Res_V^{\cT} \cO)
  \]
  with restriction functors induced by functoriality of $\Res_U^V$. 
  Moreover, applying \cref{Operads are fibrous theorem} and unwinding definitions yields an equivalence
  \[
    \Gamma^{\cT} \uOp_{\cT} \simeq \Op_{\cT}.\qedhere
  \]
\end{observation}

\subsubsection{Envelopes}\label{Envelopes subsubsection}
In \cite{Nardin}, a left adjoint to the inclusion $\CMon_{\cT} \Cat \rightarrow \Op_{\cT}$ was constructed, called the \emph{$\cT$-symmetric monoidal envelope}.
This was greatly generalized by \cref{Envelope theorem} in view of \cref{Familiar structures corollary}.
For convenience, we spell this out here.
\begin{corollary}\label{Env is maps of operads corollary}
    If $\cP^{\otimes} \rightarrow \cO^{\otimes}$ is a map of $\cT$-operads, then the following diagram consists of maps of $\cT$-operads
    \[\begin{tikzcd}
	{\Env_{\cO} \cP^{\otimes}} & {\Ar^{\act}\prn{\cO^{\otimes}}} & {\cO^{\otimes}} \\
	{\cP^{\otimes}} & {\cO^{\otimes}}
	\arrow["s", from=1-2, to=2-2]
	\arrow[from=2-1, to=2-2]
	\arrow[from=1-1, to=2-1]
	\arrow[from=1-1, to=1-2]
	\arrow["\lrcorner"{anchor=center, pos=0.125}, draw=none, from=1-1, to=2-2]
	\arrow["t", from=1-2, to=1-3]
\end{tikzcd}\]
  and the top horizontal composition is an $\cO$-monoidal $\infty$-category.
  The corresponding functor 
  \[
    \Env_{\cO}\cln \Op_{\cT, /\cO^{\otimes}} \rightarrow \Cat^{\otimes}_{\cO}
  \]
  is left adjoint to the inclusion of $\cO$-monoidal $\infty$-categories into $\cT$-operads over $\cO^{\otimes}$, and the induced functor
  \[
    \Env_{\cO}^{/\sA_{\cO}} \cln \Op_{\cT, /\cO^{\otimes}} \rightarrow \Cat_{\cO, /\sA_{\cO}}^{\otimes} 
  \]
  is fully faithful, with image spanned by equifibrations in the sense of \cite[Thm~C]{Barkan}.
\end{corollary}

We will simply write $\Env_{I}(-) \deq \Env_{\cN_{I \infty}}(-)$ and $\Env(-) \deq \Env_{\Comm_{\cT}}(-)$.
\begin{example}\label{Windex windcat example}
    Let $I$ be a weak indexing category.
    Then, unwinding definitions, we find that 
    \[
      \Env_{I} \cN_{I \infty}^{\otimes} \simeq \uFF_{I}^{I-\sqcup},
    \]
    where $\uFF_{I}^{I-\sqcup} \subset \uFF_{\cT}^{\cT-\sqcup}$ is the $I$-symmetric monoidal full $\cT$-subcategory defined in \cref{I-commutative monoids subsection}, i.e. it is the $I$-symmetric monoidal subcategory generated by $\cbr{*_V \mid V \in c(I)}$.
\end{example}

We record a convenient property of $\Env_I(-)$ here, which follows by unwinding definitions.
\begin{lemma}[{\cite[Rmk~2.4.4.3]{HA}}]\label{Envelope omnibus corollary}
  If $\cO^{\otimes} \in \Op_{I}$ and $\psi\cln T \rightarrow S$ is a map of $V$-sets, then there is an equivalence
  \begin{align*}
    \Mor^{\psi}_{\Env_I(\cO)_V \rightarrow \FF_{I,V}}(\Env_I(\cO)_V) 
    &\simeq 
    \coprod_{(\bC,\bD) \in \cO_T \times \cO_S} \Map^{\psi}_{\cO^{\otimes} \rightarrow \Span(\FF_{\cT})}(\bC,\bD)\\
    &\simeq
    \coprod_{(\bC,\bD) \in \cO_T \times \cO_S} \prod_{U \in \Orb(S)} \cO(\bC_U;D_U)
  \end{align*}
  In particular, if $\cO^{\otimes}$ has one color, then 
  \[
    \Map^{\psi}_{\Env_I(\cO)_V \rightarrow \FF_{I,V}}(iT,iS) \simeq \prod_{U \in \Orb(S)} \cO(T_U).
  \]
\end{lemma}

\subsubsection{Trivial $\cT$-operads}\label{Trivial subsubsection}
We now construct a simple family of $\cT$-operads:
\begin{construction}
  Given $\cC$ a $\cT$-$\infty$-category, we define the $\cT$-operad
  \[
    \triv(\cC)^{\otimes} \deq L_{\Op_{\cT}}\prn{\Tot \cC \rightarrow \cT^{\op} \rightarrow \Span(\FF_{\cT})}.\qedhere
  \]
\end{construction}
The following property follows by unwinding definitions.
\begin{proposition}\label{Triv prop}
  $U$ implements an equivalence
  \[
    \Alg_{\triv(\cC)}(\cO) \xrightarrow{\;\;\sim\;\;} \Fun_{\cT}(\cC,U\cO);
  \]
  in particular, $\triv(-)^{\otimes}\colon \Cat_{\cT} \rightarrow \Op_{\cT}$ is a fully faithful left adjoint to $U$.
\end{proposition}
In order to state a corollary, set the notation $\uSigma_{\cT} \deq \uFF_{\cT,*}^{\simeq}$, the latter denoting the $\cT$-space core of \cref{Core example}.
We acquire the following identification from uniqueness of left adjoints and \cite[Cor~2.4.5]{Nardin}.
\begin{corollary}\label{Triv and Nardin-Shah}
  The equivalence $\Op_{\cT} \simeq \Fbrs(\Tot \uFF_{\cT,*})$ identifies $\triv(\cC)^{\otimes}$ with the trivial $\cT$-$\infty$-operads of \cite[Cor~2.4.5]{Nardin};
  in particular, $\Tot_{\cT} \triv^{\otimes}_{\cT} \simeq \uSigma_{\cT}$, and the functor $\Tot \Tot_{\cT} \triv(\cC)^{\otimes} \rightarrow \Tot \Tot_{\cT} \triv(*_{\cT})^{\otimes} \simeq \Tot \uSigma_{\cT}$ is unstraightened from the right Kan extension
  \[
\begin{tikzcd}[ampersand replacement=\&]
	{\cT^{\op}} \& \Cat \\
	{\Tot \uSigma_{\cT}}
	\arrow["\cC", from=1-1, to=1-2]
	\arrow[from=1-1, to=2-1]
	\arrow[""{name=0, anchor=center, inner sep=0}, "{\Tot_{\cT} \triv(\cC)}"'{pos=0.3}, dashed, from=2-1, to=1-2]
	\arrow[shorten <=6pt, shorten >=2pt, Rightarrow, from=0, to=1-1]
\end{tikzcd}
  \]
\end{corollary}
\subsection{The underlying \tcT-symmetric sequence}\label{Symmetric sequence subsection}
We now study a forgetful functor to \emph{symmetric sequences}.
We begin in \cref{C-symmetric subsubsection} by defining the multi-colored variant of $\cT$-symmetric sequences, then move on in \cref{Underlying C-symmetric subsubsection} to construct the underlying $\fC$-symmetric sequence, verifying that it is monadic in the one-color setting. 
We reframe this somewhat in \cref{Other perspectives subsubsection} to introduce the \emph{$n$-ary $B_{\cT} \Sigma_n$-space} construction.
\subsubsection{$\fC$-symmetric sequences}\label{C-symmetric subsubsection}
\begin{definition}\label{Sigma V construction}
Let $\cD$ be a $\cT$-$\infty$-category.
The \emph{$\infty$-category of $\cT$-symmetric sequences in $\cD$} is $\Fun_{\cT}(\uSigma_{\cT}, \cD)$.
In the case $\cC = \ucS_{\cT}$, we refer to $\Fun_{\cT}(\uSigma_{\cT},\ucS_{\cT}) \simeq \Fun(\tot \uSigma_{\cT},\cS)$ simply as the \emph{$\infty$-category of $\cT$-symmetric sequences}.
  
More generally, if $\fC$ is a $\cT$-coefficient system of sets, we set the notation $\uSigma_{\fC} \deq \Tot_{\cT} \triv(\fC)^{\otimes}$;
we refer to $\Fun_{\cT}(\uSigma_{\fC},\cD)$ and $\Fun(\Tot \uSigma_{\fC},\cS)$ as the $\infty$-categories of $\fC$-symmetric sequences in $\cD$ and $\fC$-symmetric sequences, respectively.
\end{definition}

\begin{observation}\label{Sigma observation}
  For any adequate triple $(\cX,\cX_b,\cX_f)$, the inclusion 
  \[
    \cX \hookrightarrow \Span_{b,f}(\cX)
  \]
  induces an equivalence on cores.
  In particular, choosing $(\uFF_{\cT}, \uFF_{\cT}^{s.i.}, \uFF_{\cT})$, we find that the inclusion $(-)_+:\uFF_{\cT} \rightarrow \uFF_{\cT,*}$ induces an equivalence
  \[
    \uFF_{\cT}^{\simeq} \simeq \uFF_{\cT,*}^{\simeq} \simeq \uSigma_{\cT}.
  \]
  In particular, unwinding definitions, we find that
  \[
      \Sigma_V \deq \uSigma_{\cT,V} \simeq \FF_V^{\simeq} \simeq \coprod_{S \in \FF_V} B\Aut_V S
  \]
  and the restriction map $\Sigma_{V} \rightarrow \Sigma_{W}$ is induced by the forgetful maps $B\Aut_V S \rightarrow B\Aut_W S$.
\end{observation}
We may similarly explicitly understand $\uSigma_{\fC}$.
\begin{observation}
  A map of coefficient systems $\fC \rightarrow \fD$ induces a map of $\cT$-operads $\triv(\fC)^{\otimes} \rightarrow \triv(\fD)^{\otimes}$, i.e. a cocartesian fibration of $\cT$-$\infty$-categories $\uSigma_{\fC} \rightarrow \uSigma_{\fD}$.
  Applying this in the case $\fD = *_{\cT}$, we acquire a canonical natural cocartesian fibration of $\cT$-$\infty$-categories
  \[
    \uSigma_{\fC} \rightarrow \uSigma_{\cT};
  \]
  taking $V$-values yields a natural cocartesian fibration 
  \[
    \Sigma_{\fC,V} \rightarrow \Sigma_V
  \]
  whose straightening has discrete values 
  \[
    S \mapsto \fC^S \times \fC^V
  \]
  with functoriality along $S$-automorphisms given by permuting the factors of the $\fC^S$ part (see \cref{Triv and Nardin-Shah}).
  In particular, the classical Grothendieck construction describes $\Sigma_V$ as having;
  \begin{itemize}
    \item Objects: profiles $(V, S \in \FF_V, (\bC;D) \in \fC^S \times \fC^V)$ with orbit $V$
    \item Morphisms: $V$-equivariant automorphisms of $S$ whose action fixes $\bC$, i.e. color-preserving $S$-automorphisms.
  \end{itemize}
  In short, we find that
  \[
    \Sigma_{\fC,V} = \coprod_{\stackrel{S \in \FF_V}{(\bC;D) \in \fC^S \times \fC^V}} B\Aut_V(\bC).
  \]
  Moreover, unwinding definitions, we find that the restriction functor $\Res_W^V\colon \Sigma_{\fC,V} \rightarrow \Sigma_{\fC,W}$ is induced from the forgetful maps $B\Aut_V(\bC) \rightarrow B\Aut_W(\bC)$.
\end{observation}

\subsubsection{The underlying $\fC$-symmetric sequence}\label{Underlying C-symmetric subsubsection}
We will restrict to the following setting.
\begin{definition}
  Given $\fC\colon \cT^{\op} \rightarrow \Set$ a coefficient system of sets, we define the full subcategory of \emph{$\fC$-colored $\cT$-operads} as the pullback
  \[
    \begin{tikzcd}[ampersand replacement=\&]
      {\Op_{\cT}^{\fC}} \& {\Op_{\cT}} \\
      {\cbr{\fC}} \& {\CoFr^{\cT}(\Set)}
      \arrow[hook', from=1-1, to=1-2]
      \arrow[from=1-1, to=2-1]
      \arrow["\lrcorner"{anchor=center, pos=0.125}, draw=none, from=1-1, to=2-2]
      \arrow[from=1-2, to=2-2, "\pi_0 U"]
      \arrow[hook, from=2-1, to=2-2]
    \end{tikzcd}
  \]
\end{definition}

For instance, $\Op_{\cT}^{\oc} = \Op_{\cT}^{*_{\cT}}$ as full subcategories.
Thus the following construction recovers an \emph{underlying $\cT$-symmetric sequence} on one-color $\cT$-operads.
\begin{construction}\label{SSeq observation}
    Given $\cO^{\otimes} \in \Op_{\cT}^{\fC}$, 
    there is a structure map
    \[
      \Env_{\cO} \triv(\fC) \simeq \triv(\fC)^{\otimes} \times_{\cO^{\otimes}} \Ar^{\act, /\el}(\cO^{\otimes}) \rightarrow \triv(\fC)^{\otimes},
    \]
    where $\Ar^{\act, /\el}(\cO^{\otimes}) \subset \Ar^{\act}(\cO^{\otimes})$ is the full subcategory of active arrows whose codomain is elementary;
    this is an inert-cocartesian fibration by pullback-stability of inert-cocartesian fibrations \cite[Obs~2.1.7]{Barkan}.
    The \emph{underlying $\fC$-symmetric sequence of $\cO^{\otimes}$} is the unstraightening
    \[
        \cO_{\sseq}^{\otimes} \deq \Un_{\triv(\fC)} \Env_{\cO} \triv(\fC) \in \Fun(\tot \uSigma_{\fC}, \Cat).
    \]
    Unwinding definitions, we find that there exists a cartesian square  
    \[
        \begin{tikzcd}
            \cO(\bC;D) \arrow[r] \arrow[d] \arrow[rd, "\lrcorner" very near start, phantom]
            & \Env_{\cO} \triv(\fC) \arrow[d] \arrow[r,equals]
            & \tot \uSigma_{\fC} \times_{\cO^{\otimes}} \Ar^{\act,/\el}(\cO) \arrow[d]\\
            \cbr{(\bC;D)} \arrow[r, hook]
            & \triv(\fC)^{\otimes} \arrow[r,equals]
            & \tot \uSigma_{\fC}
        \end{tikzcd}
    \]
    so that $\cO_{\sseq}^{\otimes}$ is indeed an $\fC$-symmetric sequence.
    The associated functor is denoted
    \begin{align*}
      \sseq\colon \Op^{\fC}_{\cT} &\rightarrow \Fun(\tot \uSigma_{\fC},\cS).\qedhere
    \end{align*}
\end{construction}

We will often use the following to reduce questions about $\cT$-operads to $\cT$-symmetric sequences.
\begin{proposition}\label{Symmetric sequence proposition}
     Suppose a functor of $\cT$-operads $\varphi:\cO^{\otimes} \rightarrow \cP^{\otimes}$ satisfies the following conditions:
    \begin{enumerate}[label={(\alph*)}]
      \item $\varphi$ induces surjective maps $\pi_0 \cO_V \rightarrow \pi_0 \cP_V$ for all $V \in \cT$, and
      \item for all $V \in \cT$, all $S \in \FF_V$, all $\bC \in \cO_S$, and all $D \in \cO_V$, the map $\varphi$ induces equivalences $\varphi:\cO(\bC;D) \xrightarrow{\sim} \cP(\varphi \bC;\varphi D)$.  
    \end{enumerate}
    Then $\varphi$ is an equivalence of $\cT$-operads;
    in particular, the functor
    \[
        \sseq:\Op_{\cT}^{\fC} \rightarrow \Fun(\tot \uSigma_{\fC}, \cS)
    \]
    is conservative.
\end{proposition}
In particular, specializing to $\fC = *_{\cT}$, we find out that $\cbr{\cO(S) \mid S \in \uFF_{\cT}}$ is jointly conservative.
To prove this, we proceed by reduction to the following observation.
\begin{observation}\label{Tot is conservative}
  If $\cC \rightarrow \cD$ is an equivalence of categories over $\cE$, then it preserves and reflects cocartesian lifts of arrows in $\cE$;
  in particular, if $\varphi:\cO^{\otimes} \rightarrow \cP^{\otimes}$ is a morphism of $\cT$-operads who induces an equivalence $\tot \varphi\cln \cO^{\otimes} \rightarrow \cP^{\otimes}$ between the total $\infty$-categories of the associated functors to $\Span(\FF_{\cT})$, then its inverse is also a morphism of $\cT$-operads.
  Said another way, we've observed that the functor $U\cln \Op_{\cT} \rightarrow \Cat_{/\Span(\FF_{\cT})}$ is fully faithful on cores, hence it is conservative, so $\tot\cln \Op_{\cT} \rightarrow \Cat$ is conservative.
  Similar arguments show that $\Tot \Tot_{\cT} \cln \Op_{\cT} \rightarrow \Cat_{\cT} \rightarrow  \Cat$ is conservative.
\end{observation}

\begin{proof}[Proof of \cref{Symmetric sequence proposition}]
    The second statement follows immediately from the first, since morphisms of $\fC$-colored $\cT$-operads are $\pi_0$-isomorphisms by assumption.
    Fixing $\varphi$ satisfying (a) and (b), we will prove that $\varphi$ is an equivalence of $\cT$-operads.
    Using \cref{Tot is conservative}, it suffices to prove that $\tot \varphi$ is an equivalence of $\infty$-categories.

    By the Segal condition for colors, we have an equivalence of arrows
    \[
    \begin{tikzcd}
        \pi_0 \cO_S \arrow[d,"\varphi_S"] \arrow[r,"\simeq", phantom]
        & \prod_{V \in \Orb(S)} \pi_0 \cO_V \arrow[d,"\prod \varphi_V"]\\
         \pi_0 \cP_S \arrow[r,"\simeq",phantom]
        & \prod_{V \in \Orb(S)} \pi_0 \cP_V
    \end{tikzcd}
    \]
    Since $\pi_0 \cO \simeq \coprod_S \pi_0 \cO_S$, (a) implies that $\varphi$ is essentially surjective.
    Furthermore, the Segal condition for multimorphisms yields equivalences of arrows
    \[
      \begin{tikzcd}[column sep = tiny]
          \Map_{\cO^{\otimes}}(\bC,\bD) 
                \arrow[d,"\varphi"] 
                \arrow[r,"\simeq", phantom]
            &  \mspc\coprod\limits _{f:\pi \bC \leftarrow S \rightarrow \pi \bD} \mspc \Map_{\cO}^{f}(\bC;\bD) 
                \arrow[d,"\coprod \varphi "]
                \arrow[r,"\simeq", phantom]
                & \coprod\limits_f \prod\limits_{V \in \Orb(\pi(D))} \mspc \Map_{\cO}^{f_V}(\bC_{f^{-1}_V} ;D_V) 
                \arrow[d,"\coprod \prod \varphi"]
                \arrow[r,"\simeq", phantom]
                & \coprod\limits_f \prod\limits_{V} \cO(\bC_{f^{-1} V};D_V) 
                \arrow[d,"\coprod \prod \varphi(T_V)"]\\
            \Map_{\cP^{\otimes}}(\varphi \bC, \varphi \bD)
                \arrow[r,"\simeq", phantom]
            & \coprod\limits_{f:\pi \bC \leftarrow S \rightarrow \pi \bD} \mspc \Map_{\cP}^f(\varphi \bC;\varphi \bD)
                \arrow[r,"\simeq", phantom]
                &\coprod\limits_f \prod\limits_{V \in \Orb(S)} \mspc \Map_{\cP}^{f'}(\varphi \bC_{f^{-1} V},\varphi D_V)
                \arrow[r,"\simeq", phantom]
                & \coprod\limits_f \prod\limits_{V} \cP(\varphi \bC_{f^{-1} V};\varphi D_V).
        \end{tikzcd}
    \]
    the right arrow is an equivalence by (b), so the leftmost arrow is an equivalence, hence $\varphi$ is fully faithful.
\end{proof}

The author learned the $U_\circ$ portion of the following argument from Thomas Blom.
\begin{corollary}\label{Monadic sseq corollary}
    The functor $\sseq_{\cT}\colon \Op_{\cT}^{\oc} \rightarrow \Fun(\tot \uSigma_{\cT},\cS)$ is monadic and preserves sifted colimits. 
\end{corollary}
\begin{proof}
    By \cite[Cor~4.2.2]{Barkan}, $\Op_{\cT}^{\red}$ and $\Fun(\tot \uSigma_{\cT},\cS)$ are presentable, so by Barr-Beck \cite[Thm~4.7.3.5]{HA} and the adjoint functor theorem \cite[Cor~5.5.2.9]{HTT}, it suffices to prove that $\sseq$ is conservative and preesrves limits and sifted colimits.
    Conservativity is \cref{Symmetric sequence proposition}, and (co)limits in functor categories are computed pointwise by \cite[Prop~5.1.2.2]{HTT}, so it suffices to prove that $\cO \mapsto \cO(S)$ preseres limits and sifted colimits.
    We separate this into manageable chunks via the following diagram:
    \[\begin{tikzcd}[ampersand replacement=\&, column sep=large]
	{\Op^{\oc}_{\cT}} \&\&\& \cS \\
	{\Cat^{\mathrm{Int-cocart,core-iso}}_{/\Span(\FF_{\cT})}} \& {\Cat^{\mathrm{core-iso}}_{/\Span(\FF_{\cT})}} \& {\Cat^{\mathrm{core-iso}}_{/\Span(\FF_{V})}} \& {\Fun\prn{\prn{\Span(\FF_{V})^{\simeq}}^{\times 2}, \cS}}
	\arrow["{{\cO \mapsto\cO(S)}}", from=1-1, to=1-4]
	\arrow["{{U_{\mathrm{Seg}}}}", from=1-1, to=2-1]
	\arrow["{{U_{\mathrm{cocart}}}}", from=2-1, to=2-2]
	\arrow["{\Res_V^{\cT}}", from=2-2, to=2-3]
	\arrow["{{U_{\circ}}}", from=2-3, to=2-4]
	\arrow["{\ev_{\Ind_V^{\cT} S, V}}"{description}, from=2-4, to=1-4]
\end{tikzcd}\]
$\pi$ and $\ev_{\Ind_V^{\cT} S,V}$ preserve (co)limits since they are evaluation of functor categories \cite[Prop~5.1.2.2]{HTT}.  
    $U_{\mathrm{Cocart}}$ preserves limits and sifted colimits by \cite[Cor~2.1.5]{Barkan}.
    $U_{\mathrm{Seg}}$ preserves limits and sifted colimits, as each commute with finite products.
    $\Res_V^{\cT}$ preserves limits and sifted colimits, as it is a left and right adjoint.

    By \cite[Prop~3.12]{Haugseng_segal}, $U_{\circ}$ is equivalent to the forgetful functor 
    \[
      \Alg(\cS_{/\Span(\FF_{\cT})^{\simeq},\Span(\FF_{\cT})^{\simeq}}) \rightarrow \cS_{/\Span(\FF_{\cT})^\simeq, \Span(\FF_{\cT})^\simeq},
    \]
    where $\cS^{\otimes}_{/Y,Y}$ is a monoidal structure on $\cS_{/Y,Y} \simeq \cS_{Y \times Y} \simeq \Fun(Y\times Y, \cS)$. 
    This functor preserves limits and sifted colimits by \cite[Prop~3.2.3.1]{HA}, completing the argument.
\end{proof}

In particular, this constructs a left adjoint
\[
    \Fr:\Fun_{\cT}(\uSigma_{\cT},\ucS_{\cT}) = \Fun(\tot \uSigma_{\cT},\cS) \rightarrow \Op_{\cT}^{\oc}
\]
to $\sseq$.
We lift this to a $\cT$-adjunction in the following construction.

\def\uFr{\underline{\Fr}}
\begin{construction}\label{Fr T-functor construction}
    The functor $\sseq$ is associated with a $\cT$-functor $\usseq$ as in the following diagram
\[
  \begin{tikzcd}
    &&&& \textcolor{rgb,255:red,0;green,14;blue,204}{\Ar^{\act,/\el}(\cO^{\otimes})} \\
    \textcolor{rgb,255:red,0;green,14;blue,204}{\cO^{\otimes}} & \textcolor{rgb,255:red,0;green,14;blue,204}{\triv_{\cT}^{\otimes}} & \textcolor{rgb,255:red,0;green,14;blue,204}{\cO^{\otimes}} & \textcolor{rgb,255:red,0;green,14;blue,204}{\triv_{\cT}^{\otimes}} & \textcolor{rgb,255:red,0;green,14;blue,204}{\cO^{\otimes}} \\
    {\uOp_{\cT}^{\oc}} & {\uOp_{\cT, \triv_{\cT}^{\otimes}/}} && {\uFun_{\cT}\prn{\Infl_{e}^{\cT} \Lambda_2^2,\Op_{\cT}} \times_{\uOp_{\cT}} \cbr{\triv_{\cT}^{\otimes}}} \\
    {\uFun_{\cT}\prn{\uSigma_{\cT},\ucS_{\cT}}} & {\Op_{\cT,/\triv_{\cT}^{\otimes}}} && {\uFun_{\cT}\prn{\Infl_e^{\cT} \Delta_1^2, \uOp_{\cT}} \times_{\uOp_{\cT}} \cbr{\triv_{\cT}^{\otimes}}} \\
    \textcolor{rgb,255:red,0;green,14;blue,204}{\sseq \cO^{\otimes}} & \textcolor{rgb,255:red,0;green,14;blue,204}{\Env_{\cO}\triv} && \textcolor{rgb,255:red,0;green,14;blue,204}{\Env_{\cO}\triv} & \textcolor{rgb,255:red,0;green,14;blue,204}{\Ar^{\act,/\el}(\cO^{\otimes})} \\
    & \textcolor{rgb,255:red,0;green,14;blue,204}{\triv_{\cT}^{\otimes}} && \textcolor{rgb,255:red,0;green,14;blue,204}{\triv_{\cT}^{\otimes}} & \textcolor{rgb,255:red,0;green,14;blue,204}{\cO^{\otimes}}
    \arrow["s"{description}, color={rgb,255:red,0;green,14;blue,204}, from=1-5, to=2-5]
    \arrow["\in"{marking, allow upside down}, color={rgb,255:red,0;green,14;blue,204}, draw=none, from=2-1, to=3-1]
    \arrow[""{name=0, anchor=center, inner sep=0}, color={rgb,255:red,0;green,14;blue,204}, from=2-2, to=2-3]
    \arrow["\in"{marking, allow upside down}, color={rgb,255:red,0;green,14;blue,204}, draw=none, from=2-2, to=3-2]
    \arrow[""{name=1, anchor=center, inner sep=0}, color={rgb,255:red,0;green,14;blue,204}, from=2-4, to=2-5]
    \arrow["\in"{marking, allow upside down}, shift left=5, color={rgb,255:red,0;green,14;blue,204}, draw=none, from=2-4, to=3-4]
    \arrow[hook, from=3-1, to=3-2]
    \arrow["\usseq"', from=3-1, to=4-1]
    \arrow[from=3-2, to=3-4]
    \arrow[from=3-4, to=4-4]
    \arrow[Rightarrow, no head, from=4-2, to=4-1]
    \arrow["U", from=4-4, to=4-2]
    \arrow["\in"{marking, allow upside down}, color={rgb,255:red,0;green,14;blue,204}, draw=none, from=5-1, to=4-1]
    \arrow["\in"{marking, allow upside down}, color={rgb,255:red,0;green,14;blue,204}, draw=none, from=5-2, to=4-2]
    \arrow[""{name=2, anchor=center, inner sep=0}, color={rgb,255:red,0;green,14;blue,204}, from=5-2, to=6-2]
    \arrow["\in"{marking, allow upside down}, color={rgb,255:red,0;green,14;blue,204}, draw=none, from=5-4, to=4-4]
    \arrow[""{name=3, anchor=center, inner sep=0}, color={rgb,255:red,0;green,14;blue,204}, from=5-4, to=5-5]
    \arrow[""{name=4, anchor=center, inner sep=0}, color={rgb,255:red,0;green,14;blue,204}, from=5-4, to=6-4]
    \arrow["\lrcorner"{anchor=center, pos=0.125}, color={rgb,255:red,0;green,14;blue,204}, draw=none, from=5-4, to=6-5]
    \arrow["s"{description}, color={rgb,255:red,0;green,14;blue,204}, from=5-5, to=6-5]
    \arrow[color={rgb,255:red,0;green,14;blue,204}, from=6-4, to=6-5]
    \arrow[curve={height=-30pt}, shorten >=26pt, maps to, from=2-1, to=0]
    \arrow[curve={height=-30pt}, shorten <=26pt, shorten >=13pt, maps to, from=0, to=2-5]
    \arrow[shift left=5, curve={height=-30pt}, shorten <=8pt, shorten >=8pt, maps to, from=1, to=3]
    \arrow[shorten <=8pt, maps to, from=2, to=5-1]
    \arrow[shorten <=16pt, shorten >=16pt, maps to, from=4, to=2]
  \end{tikzcd}
\]    
    By \cite[Prop~7.3.2.1]{HA}, the pointwise left adjoints $\Fr$ lifts to a $\cT$-adjunction
    \[
        \usseq:\uOp_{\cT}^{\oc} \leftrightarrows \uFun_{\cT}(\uSigma_{\cT},\ucS_{\cT}):\uFr,
    \]  
    i.e. $\uFr$ is compatible with restriction.
\end{construction}

\subsubsection{Other persiectives on $\cT$-symmetric sequences}\label{Other perspectives subsubsection}
\begin{remark}\label{First O(n) remark}
  Let $\cO_{G \times \Sigma_n,\Gamma_n} \subset \cO_{G \times \Sigma_n}$ be the full subcategory spanned by $[G \times \Sigma_n / \Gamma_S]$ for $\phi_S:H \rightarrow \Sigma_n$ a map with associated graph subgroup $\Gamma_S = \cbr{(h,\phi_S(h)) \mid h \in H} \subset H \times \Sigma_{n} \subset G \times \Sigma_n$.
    This possesses an evident forgetful functor $\cO^{\op}_{G \times \Sigma_n / \Gamma_S} \rightarrow \cO^{\op}_G$ taking $[G \times \Sigma_n / \Gamma_S] \rightarrow [G/H]$;
    in \cite[Ex~4.3.7]{Nardin}, this was shown to be a cocartesian fibration factoring through an equivalence
    \[
      \coprod_{n \in \NN} \cO^{\op}_{G \times \Sigma_n / \Gamma_S} \simeq \Tot \uSigma_G \rightarrow \cO_G^{\op}
    \]
    taking $[G \times \Sigma_n / \Gamma_S] \mapsto (H,S)$, and hence taking the $G$-space presented by $\cO_{G \times \Sigma_n / \Gamma_S}$ equivalently onto the summand $B_G \Sigma_n \subset \uSigma_G$, the classifying $G$-space for equivariant principle $\Sigma_n$-bundles.  
 
    More generally, we say that an atomic orbital $\infty$-categories $\cT$ is \emph{EI} if the inclusions $\Aut(V) \hookrightarrow \End(V)$ are equivalences.
    If $\cT$ is EI and admits a weakly initial object $e \in \cT$, we may define the $\cT$-subspace $B_{\cT} \Sigma_n \subset \uSigma_{\cT}$ as corresponding with the $\cT$-sets $S$ whose restriction $\Res_e^{\cT} S$ is a set with $n$ elements;
    then, we acquire a splitting
    \[
      \uSigma_{\cT} \simeq \coprod_{n \in \NN} B_{\cT} \Sigma_n.
    \]
    Under the above equivalence, given $\cO^{\otimes} \in \Op_{\cT}^{\oc}$, we define the \emph{$n$-ary $B_{\cT} \Sigma_n$-space}
    \[
      \ucO(n)\colon B_{\cT} \Sigma_n \subset \uSigma_{\cT} \xrightarrow{\sseq \cO} \uSigma_{\cT}.
    \]
    For instance, if $\cF \subset \cO_G$ is a family of subgroups, then $\cF$ is EI with a weakly initial object, and $B_{\cF} \Sigma_n$ is the classifying $\cF$-space for $\cF$-genuine $G$-equivariant principal $\Sigma_n$-bundles.
    In the case $\cT = \cO_G$, $\ucO(n)$ is characterized by its $\Gamma_S$-fixed points $\ucO(n)^{\Gamma_S} \simeq \cO(S)$, with restriction functors along $\Gamma_{\Res_K^H S} \subset \Gamma_{S}$ corresponding with restriction map $\cO(S) \rightarrow \cO(\Res_K^H S)$.
\end{remark}
We will see in \cref{O(n) nerve remark} that this intertwines with a nerve functor from operad objects in topological $G$-spaces.
We will also need the following notation.
\begin{notation}
  Given an orbit $V \in \cT$, and a finite $V$-set $S \in \FF_V$, we may define a natural ``$S$-ary'' $V$-space in $\cT$-symmetric sequences 
  \[
    \underline{(-)}(S)\colon \Op^{\oc}_{\cT} \rightarrow \Fun_{\cT}(\uSigma_{\cT},\ucS_{\cT}) \rightarrow \Fun_{V}(\uSigma_V,\ucS_V) \xrightarrow{\ev_S} \cS_V.
  \]
  More generally, we may define the analogous $V$-space for an $\cO$-profile:
  \[
    \underline{(-)}(\bC;D)\colon \Op^{\fC}_{\cT} \rightarrow \Fun_{\cT}(\uSigma_{\fC},\ucS_{\cT}) \rightarrow \Fun_{V}(\uSigma_{\fC,V},\ucS_V) \xrightarrow{\ev_{(\bC;D)}} \cS_V.\qedhere
  \]
\end{notation}

\subsection{The monad for \texorpdfstring{$\cO$}{O}-algebras}\label{Monad subsection}
We now take a detour into studying the free $\cO$-algebra monad.
Our main application for this is the following theorem.
\begin{theorem}[{``Equivariant \cite[Thm~4.1.1]{Hinich}''}]\label{Unital conservativity corollary}
  A map of $\cT$-operads $\varphi\colon \cO^{\otimes} \rightarrow \cP^{\otimes}$ is an equivalence if and only if it satisfies the following conditions:
  \begin{enumerate}[label={(\alph*)}]
    \item the $\cT$-functor $U(\varphi)\colon \cO \rightarrow \cP$ is essentially surjective, and
    \item the pullback functor $\varphi^*\colon \Alg_{\cP}(\ucS_{\cT}) \rightarrow \Alg_{\cO}(\ucS_{\cT})$ is an equivalence of $\infty$-categories.
  \end{enumerate}
\end{theorem}

Fix $\cO^{\otimes}$ a one-object $\cT$-operad, fix $\cC^{\otimes}$ a distributive $\cO$-monoidal category in the sense of \cref{Distributivity observation} (e.g. it may be presentably $\cO$-monoidal) and let $\triv^\otimes_{\cT} \rightarrow \cC^{\otimes}$ be the functor of operads associated with a $\cT$-object $X \in \Gamma \cC$.
Denote by $X^{\otimes}\colon \Env_{\cO} \triv^{\otimes}_{\cT} \rightarrow \cC^{\otimes}$ the associated $\cO$-symmetric monoidal functor, and denote by
\[
    \cO_{\sseq}(X)\colon \Env_{\cO} \triv_{\cT} \rightarrow \cC
\]
the underlying $\cT$-functor.
Given $Y$ a $V$-space and $X \in \cC_V$, we will write $Y \cdot X$ for the indexed colimit of the constant $Y$-indexed diagram $Y \rightarrow *_V \rightarrow \Res_V^{\cT} \cC$ at $X$.

\begin{proposition}[{``Equivariant \cite[Lem~2.4.2]{Schlank}''}]\label{Monad proposition}
  The forgetful $\cT$-functor $U\colon \uAlg_{\cO}(\cC) \rightarrow \cC$ is monadic, and the associated monad $T_{\cO}$ acts on $X \in \Gamma^{\cT} \cC$ by the indexed colimit
  \begin{align*}
    T_{\cO} X &\deq \ucolim \cO_{\sseq}(X),\\
              &\simeq \ucolim_{S \in \uSigma_{\cT}} \ucO(S) \cdot X^{\otimes S}.
  \end{align*}
\end{proposition}
\begin{proof}
  Monadicity is precisely \cite[Cor~5.1.5]{Nardin}, so it suffices to compute the associated monad.
  By \cite[Rem~4.3.6]{Nardin}, the left adjoint $\Fr\colon \cC \rightarrow \Alg_{\cO}(\cC)$ is computed on $X$ by $\cT$-operadic left Kan extension of the corresponding map $\triv^{\otimes} \xrightarrow X \cC^{\otimes}$ along the canonical inclusion $\triv^{\otimes} \rightarrow \cO^{\otimes}$, and the underlying $\cT$-functor of this is computed by the composite $\cT$-left Kan extension
   \[
     \begin{tikzcd}[row sep=small]
       {\Env_{\cO}  \triv} & {\uSigma_{\cT} \times_{\cO^{\otimes}} \Ar^{\act, /\el}(\cO)} && \cC \\
	& {\uSigma_{\cT}} \\
	\cO & {*_{\cT}}
	\arrow[Rightarrow, no head, from=1-1, to=1-2]
	\arrow[""{name=0, anchor=center, inner sep=0}, "{{{T_{\cO} X}}}"', curve={height=18pt}, dashed, from=3-2, to=1-4]
	\arrow[Rightarrow, no head, from=3-1, to=3-2]
	\arrow[from=1-1, to=3-1]
	\arrow[from=1-2, to=2-2]
	\arrow[from=2-2, to=3-2]
	\arrow["X", from=1-2, to=1-4]
	\arrow["{\widetilde T_{\cO}X}"{description}, curve={height=6pt}, dashed, from=2-2, to=1-4]
	\arrow[""{name=1, anchor=center, inner sep=0}, shift left, shorten <=4pt, shorten >=36pt, Rightarrow, from=1-2, to=0]
	\arrow[shorten <=8pt, shorten >=5pt, Rightarrow, from=1, to=0]
    \end{tikzcd}
    \]
  $\cT$-left Kan extension diagrams to $\uAst_{\cT}$ are $\cT$-colimit diagrams by definition, so the underlying $\cT$-object is
  \[
    T_{\cO}X \simeq \ucolim \cO_{\sseq}(X). 
  \]
  Additionally, the $\cT$-left Kan extension $\widetilde T_{\cO} X$ has values given by the indexed colimit
   \begin{align*}
        \widetilde T_{\cO} X(S) 
        &\simeq \ucolim_{\pr_1(T,x \in \cO(T)) \rightarrow S} X^{\otimes T};
   \end{align*}
   in fact, the inclusion $\ucO(S) \simeq \cbr{S} \times_{\cO^{\otimes}} \Ar^{\act, /\el}(\cO) \subset \prn{\uSigma_{\cT} \times_{\cO^{\otimes}}\Ar^{\act, /\el}(\cO)}^{/\underline{S}}$ is $\cT_{/V}$-final, so it induces an equivalence
   \begin{align*}
        \widetilde T_{\cO} X(S) 
        &\simeq \ucolim_{x \in \ucO(S)} X^{\otimes S}\\
        &\simeq \ucO(S) \cdot X^{\otimes S}
   \end{align*}
  and the result follows by composition of $\cT$-left Kan extensions.
\end{proof}
\begin{remark}
    In view of \cref{First O(n) remark}, we may rewrite \cref{Monad proposition} in the case $\cT$ is EI with a weakly initial object as
    \[
      T_{\cO} X \simeq \coprod_{n \in \NN} \ucolim_{S \in B_{\cT} \Sigma_n} \ucO(S) \cdot X^{\otimes S}.
    \]
    We would like to interpret this in a more traditional way, so define the $B_{\cT} \Sigma_n$-space $X^{n}$ by
    \[
      X^n\colon B_{\cT} \Sigma_n \subset \uSigma_{\cT} \subset \uFF_{\cT}^{\op} \xrightarrow{S \mapsto X^S} \ucS_{\cT}.
    \]
    In the case $\cT = \cO_G$, this is characterized by its graph subgroup fixed points $\prn{X^n}^{\Gamma_S} \simeq X^S$.
    The $B_{\cT} \Sigma_n$-space corresponding with $S \mapsto \ucO(S) \cdot X^{\otimes n}$ is $\ucO(n) \cdot X^n$.
    Using the notation $\prn{-}_{h_{\cT} \Sigma_n}$ for $B_{\cT} \Sigma_n$-indexed colimits, we may then write the formula
    \[
      T_{\cO} X \simeq \coprod_{n \in \NN} \prn{\ucO(n) \times X^n}_{h_{\cT} \Sigma_n}.
    \]
    For instance, when $\cO = \EE_V^{\otimes}$, one may check that this agrees with the monad $\mathbb{K}_V$ for free algebras over the $V$-Steiner operad considered in \cite{Guillou-May}, so it satisfies an approximation theorem to $\Omega^V \Sigma^V$.
\end{remark}
By \cite[Prop~3.2.5]{Nardin}, the Cartesian $\cT$-symmetric monoidal structure on $\uCoFr^{\cT}(\cC)$ is distributive whenever $\cC$ is a cocomplete Cartesian closed category.
In this setting, we may easily characterize the associated monad.
\begin{corollary}\label{CoFr monad}
  Suppose $\cC$ a cocomplete cartesian closed $\infty$-category.
  Then, the forgetful functor 
  \[
    \Alg_{\cO}(\uCoFr^{\cT}(\cC)) \rightarrow \CoFr^{\cT}(\cC)
  \]
  is monadic, and the associated monad $T_{\cO}$ has fixed points
  \begin{align*}
    \prn{T_{\cO} X}^V 
    &\simeq \coprod_{S \in \FF_V} \prn{\cO(S) \cdot \prn{X^{S}}^V}_{h\Aut_V(S)}\\
    &\simeq \coprod_{S \in \FF_V} \prn{\cO(S) \cdot \prod_{U \in \Orb(S)} X^U}_{h\Aut_V(S)}
  \end{align*}
\end{corollary} 
\begin{proof}
  This follows from \cref{Monad proposition} by combining the fixed points of indexed colimits formula of \cref{Fixed points of colimit} with the description of $\Sigma_V$ in \cref{Sigma observation}. 
\end{proof}
In fact, we may say more;
on the summand corresponding with $S$, the restriction map on $(T_{\cO} X)^V \rightarrow (T_{\cO} X)^U$ is induced from the restriction map
\[
  \cO(S) \cdot \prn{X^S}^V \rightarrow \cO(\Res_U^V S) \cdot \prn{X^S}^V \rightarrow \cO(\Res_U^V S) \cdot \prn{X^{\Res_U^V}S}^U  
\]
\begin{corollary}\label{Conservativity corollary}
  The functor $\Alg_{(-)}(\ucS_{\cT})\colon \Op_{\cT}^{\oc} \rightarrow \Cat$ is conservative.
\end{corollary}
\begin{proof}
  Suppose $\varphi\colon \cO \rightarrow \cP$ induces an equivalence $\Alg_{\cP}(\ucS_{\cT}) \xrightarrow{\sim} \Alg_{\cO}(\ucS_{\cT})$.
  Then $\varphi$ induces a natural equivalence $T_{\cO} \implies T_{\cP}$ respecting the summand decomposition in \cref{CoFr monad}.
  Choosing $X = S \in \FF_V$, note that the $V$-equivariant automorphisms embed as a summand $\Aut_V(S) \subset \End_V(S) \simeq \prn{S^S}^V$, yielding a natural coproduct decomposition
  \begin{align*}
    \prn{\cO(S) \times \prn{S^{\times S}}^V}_{h\Aut_V S}
     &\simeq \prn{\cO(S) \times \Aut_V S}_{h\Aut_V S} \sqcup J_{\cO,S}\\
     &\simeq \cO(S) \sqcup J_{\cO,S}
  \end{align*}
   for some $J_{\cO,S}$;
   hence the summand-preserving equivalence $T_\varphi\colon T_{\cO} S \implies T_{\cP} S$ implies that $\varphi(S)\colon \cO(S) \rightarrow \cP(S)$ is an equivalence for all $S$, i.e. $\sseq \varphi\colon \sseq \cO \rightarrow \sseq \cP$ is an equivalence of $\cT$-symmetric sequences.
  Thus \cref{Symmetric sequence proposition} implies that $\varphi$ is an equivalence.
\end{proof}
\begin{remark}
 \cref{CoFr monad} agrees with the free Segal $\Tot \cO^{\otimes}$-object monad of \cite[Cor~8.12]{Chu} in view of \cref{Span segal condition lemma,Soundly extendable lemma,CMon I cat corollary}, so $\cO^{\otimes}$-algebras in the Cartesian structure on $\uCoFr^{\cT} \cC$ are interpretable as \emph{$\cO^{\otimes}$-monoids} (c.f. \cite[Prop~2.4.2.5]{HA} in view of \cref{Pointed prop,Two models for segal objects}).
  We do not study this further at present, as we will cover the general Cartesian case in forthcoming work \cite{Tensor}.
\end{remark}

To finish the section, we repeat the above work without the one-color assumption.
\begin{observation}\label{Multi-color observation}
   Analogously to \cite{Hinich}, let $f\colon \fC \rightarrow \cO$ be a $\cT$-functor from a coefficient system of sets, and let $\triv(\fC) \rightarrow \cO^{\otimes}$ be the corresponding map of $\cT$-operads.
   Then, \cite[Thm~5.1.4]{Nardin} constructs a left $\cT$-adjoint to the pullback functor $\uAlg_{\cO}(\cD) \rightarrow \uAlg_{\triv(\fC)}(\cC) \simeq \uFun_{\cT}(\fC,\cD)$, whose associated $\cT$-functor has value on the $\cO$-algebra $X$ given by the $\cT$-left Kan extension
  \[
    \begin{tikzcd}[ampersand replacement=\&, row sep=small]
      {\Env_{\cO}  \triv(\fC)} \& {\uSigma_{\fC} \times_{\cO^{\otimes}} \Ar^{\act, /\el}(\cO)} \&\& \cD \\
      \& {\uSigma_{\fC}} \\
      \fC \& \fC
      \arrow[Rightarrow, no head, from=1-1, to=1-2]
      \arrow[from=1-1, to=3-1]
      \arrow["X", from=1-2, to=1-4]
      \arrow[from=1-2, to=2-2]
      \arrow["{{\widetilde T_{\cO}X}}"{description}, curve={height=6pt}, dashed, from=2-2, to=1-4]
      \arrow[from=2-2, to=3-2]
      \arrow[Rightarrow, no head, from=3-1, to=3-2]
      \arrow[""{name=0, anchor=center, inner sep=0}, "{{{{T_{\cO} X}}}}"', curve={height=18pt}, dashed, from=3-2, to=1-4]
      \arrow[""{name=1, anchor=center, inner sep=0}, shift left, shorten <=5pt, shorten >=37pt, Rightarrow, from=1-2, to=0]
      \arrow[shorten <=11pt, shorten >=7pt, Rightarrow, from=1, to=0]
    \end{tikzcd}
  \]
  By an analogous argument to \cref{Monad proposition}, we have
  \[
    \widetilde T_{\cO} X(\bC,D) \simeq \ucO(f\bC;fD) \cdot \bigotimes_U^{\pi f \bC} X_{C_U}; 
  \]
  moreover, when $\cD \simeq \uCoFr^{\cT}(\cC)$ for $\cC$ a cocomplete cartesian closed $\infty$-category, we have
  \begin{align*}
    \prn{T_{\cO} X}_D^V &\simeq 
    \coprod_{(\bC,D) \in \Sigma_{\fC,V}} \prn{\cO(f\bC;fD) \times \prod_{U \in \pi \bC} X_{C_U}^U}_{h\Aut_{\fC,V} \bC}.
  \end{align*}
  Momentarily choose $\cC = \cS$, and note that, if $X_{C_U} = \Res_U^V S$ for $S \in \FF_V \subset \cS_V$ for all $U$, then 
  \[
    \prod_{U \in \pi \bC} X_{C_U}^U \simeq \prn{S^{\times \pi\bC}}^V \simeq \Map^V(\pi \bC,S),
  \]
  with $\Aut_{\fC,V} \bC$-action given by the composite map $\Aut_{\fC,V} \bC \rightarrow \End_V(\bC) \rightarrow \End\Map^V(\pi \bC,S)$.
  In particular, choosing $X$ to be the functor $\fC \rightarrow \CoFr^V(\cS)$ which is constant at $\pi \bC$,
  we acquire a natural equivalence
  \begin{align*}
    \prn{\cO(f\bC;fD) \times \prod_{U \in \pi \bC} X_{C_U}^U}_{h\Aut_{\fC,V} \bC}
    &\simeq \prn{\cO(f\bC;fD) \times \prn{\pi \bC^{\times \pi \bC}}^V}_{h\Aut_{\fC,V} \bC}\\
    &\simeq \prn{\cO(f\bC;fD) \times \Aut_{\fC,V} \bC}_{h\Aut_{\fC,V} \bC} \sqcup J_{\cO, (\bC;D)}\\
    &\simeq \cO(f\bC;fD) \sqcup J_{\cO, (\bC;D)}
  \end{align*}
  for some object $J_{\cO, (\bC;D)} \in \cS$.

  Now, note that there is a unique symmetric monoidal functor $\Fr_{\cC}\colon \cS \rightarrow \cC$ under the Cartesian structure, and the induced map $\uCoFr^{\cT}(\cS) \rightarrow \uCoFr^{\cT}(\cC)$ preserves fiberwise products and indexed coproducts.
  In particular, we acquire a natural splitting
  \begin{equation}\label{Multi-color splitting}
    \Fr_{\cC}\cO(f\bC;fD) \sqcup J' \simeq T_{\cO} \Fr_{\cC}(\underline{\pi \bC}).\qedhere
  \end{equation}
\end{observation}
We now conclude a proof of \cref{Unital conservativity corollary}.
When $\varphi$ is an equivalence, conditions (a) and (b) are obvious,
Conversely, assume conditions (a) and (b);
it suffices to argue that $\varphi(\bC;D) \rightarrow \cP(\varphi \bC;\varphi D)$ is an equivalence for all $(\bC;D) \in \cO_S \times \cO_V$ by \cref{Symmetric sequence proposition}.
This follows from the following stronger proposition.
\begin{proposition}\label{Fr monad proposition}
  Suppose $\cC$ is a presentable and cartesian closed $\infty$-category and $\varphi\colon \cO^{\otimes} \rightarrow \cP^{\otimes}$ a map of $\cT$-operads whose pullback functor $\Alg_{\cP}(\uCoFr^{\cT}(\cC)) \rightarrow \Alg_{\cO}(\uCoFr^{\cT}(\cC))$ is an equivalence.
  Then, for all $(\bC;D) \in \cO_S \times \cO_V$, the induced map
  \[
    \Fr_{\cC}\cO(\bC;D) \rightarrow \Fr_{\cC}\cP(\varphi\bC;\varphi D)
  \]
  is an equivalence, where $\Fr_{\cC}\colon \cS \rightarrow \cC$ is the (unique) left adjoint sending $* \in \cS$ to the terminal object of $\cC$.
\end{proposition}
\begin{proof}
  We will study the sequence of adjunctions on algebras in $\cT$-spaces associated with the sequence
  \[
    \triv(\pi_0 \cO)^{\otimes} \xrightarrow{\gamma} \cO^{\otimes} \xrightarrow{\varphi} \cP^{\otimes}
  \]
  using \cref{Multi-color observation}. 
  Condition (b) guarantees that $\varphi^*$ is an equivalence \emph{over $\Fun_{\cT}(\pi_0 \cO,\uCoFr^{\cT}(\cC))$}:
  \[
    \begin{tikzcd}[ampersand replacement=\&, column sep=tiny]
      {\Alg_{\cP}(\uCoFr^{\cT}(\cC))} \&\& {\Alg_{\cO}(\uCoFr^{\cT}(\cC))} \\
      \& {\Fun_{\cT}(\pi_0 \cO, \uCoFr^{\cT}(\cC))}
      \arrow["\sim", from=1-1, to=1-3]
      \arrow["{\prn{\gamma \varphi}^*}"', from=1-1, to=2-2]
      \arrow["{\gamma^*}", from=1-3, to=2-2]
    \end{tikzcd}
  \]
  this induces a natural equivalence between the associated monads for $(\gamma \varphi)^*$ and $\gamma^*$ respecting the splitting of \cref{Multi-color splitting} for each $(\bC,D)$, and hence yields an equivalence $\varphi\colon \Fr_{\cC}\cO(\bC;D) \rightarrow \Fr_{\cC}\cP(\varphi \bC;\varphi \bD)$, as desired.
\end{proof}

\subsection{\tcO-algebras in \tI-symmetric monoidal \texorpdfstring{$d$}{d}-categories}\label{d-operads subsection}
Recall that a space $X$ is said to be \emph{$d$-truncated} if it is empty or $\pi_n(X,x) = *$ for all $x \in X$ and $n > d$;
in particular, $X$ is $(-1)$-truncated precisely if it is either empty or contractible.
In \cref{Truncations of SMCs subsubsection}, we applied this to mapping spaces to define \emph{$\cT$-symmetric monoidal $d$-categories}.
In this section, we define a compatible notion of \emph{$\cT$-$d$-operads}, centered on the following result.
\begin{proposition}\label{Operad mapping fiber prop}
    Let $\cO^{\otimes}$ be a $\cT$-operad and let $d \geq -1$.
    Then, the following conditions are equivalent:
    \begin{enumerate}[label={(\alph*)}]
        \item $\cO(S)$ is $d$-truncated for all $S \in \FF_V$.
        \item The $\cT$-functor $\Env \cO \rightarrow \uFF_{\cT}$ has $d$-truncated mapping fibers.
    \end{enumerate}
\end{proposition}
\begin{proof}
    Let $\psi\colon T \rightarrow S$ be a map of $\cT$-sets over $V$.
    Then, by \cref{Envelope omnibus corollary}, we have a equivalences
    \begin{equation}\label{Mor equation}
      \begin{split}
        \Mor^{\psi}_{\Env \cO \rightarrow \uFF_{\cT}}(\Env \cO) 
        &\simeq \coprod_{\bC \in \cO_T, \bD \in \cO_S} \Map^{\psi}_{\Env \cO \rightarrow \uFF_{\cT}}(\bC,\bD)\\
        &\simeq \coprod_{\bC \in \cO_T, \bD \in \cO_S} \prod_{U \in \Orb(S)} \Map^{\psi}_{\Env \cO \rightarrow \uFF_{\cT}}(\bC_U,D_U)\\
        &\simeq \coprod_{\bC \in \cO_T, \bD \in \cO_S} \prod_{U \in \Orb(S)} \cO(\bC_U;D_U),
      \end{split}
    \end{equation}
    natural in $\cO^{\otimes}$.
    First, in the case $d = -1$, note that conditions (a) and (b) both imply that $\cO$ has at most one color, so \cref{Mor equation} specializes to
    \[
        \Mor^{\psi}_{\Env \cO \rightarrow \uFF_{\cT}}(\Env \cO) \simeq \prod_{U \in \Orb(S)} \cO(S).
    \]
    Thus it suffices to note that a product is $-1$-truncated if and only if its factors are.

    Next, in the case $d \geq 0$, note that a coproduct of spaces is $d$-truncated if and only if its factors are;
    hence \cref{Mor equation} shows that (b) is equivalent to the condition that $\prod_{U \in \Orb(S)} \cO(\bC_U;D_U)$ is $d$-truncated for all $S,\bC,\bD$.
    In fact, the equation 
    \[
      \cO(S) \simeq \coprod_{(\bC,D) \in \cO_S \times \cO_V} \cO(\bC;D)
    \]
    shows that this (b) equivalent to the condition that $\cO(S)$ is $d$-truncated for all $S \in \FF_V$, as desired.
\end{proof}

We define the full subcategory of \emph{$d$-operads} $\Op_{\cT,d} \subset \Op_{\cT}$
to be spanned by $\cT$-operads satisfying the condition that $\cO(S)$ is $(d-1)$-truncated for all $S \in \FF_V$ as in \cref{Operad mapping fiber prop}.
The following corollary follows from \cref{Operad mapping fiber prop} and the mapping fiber truncation characterizations of \cref{T-fiber truncatedness corollary}.
\begin{corollary}
\label{Truncated operad prop}
    Let $\cO^{\otimes}$ be a $\cT$-operad and let $d \geq 1$.
    The following conditions are equivalent:
    \begin{enumerate}[label={(\alph*)}]
      \item $\cO^{\otimes}$ is a $\cT$-$d$-operad, and
        \item $\Env \cO^{\otimes}$ is a $\cT$-symmetric monoidal $d$-category.
    \end{enumerate}
    Furthermore, the following conditions are equivalent:
    \begin{enumerate}[label={(\alph*')}]
      \item $\cO^{\otimes}$ is a $\cT$-$0$-operad, and
        \item the $\cT$-functor $\Env \cO \rightarrow \uFF_{\cT}$ is a replete $\cT$-subcategory inclusion.
    \end{enumerate}
\end{corollary}

In general, these form a well-behaved subcategory.
\begin{corollary}\label{Operad truncation corollary}
   The inclusion $\Op_{\cT,d} \hookrightarrow \Op_{\cT}$ has a left adjoint $h_{d}$ satisfying
  \[
    \prn{h_{d} \cO}(S) \simeq \tau_{\leq d-1} \cO(S). 
  \]  
   Furthermore, when $d \geq 1$, this fits into the following diagram
    \[
        \begin{tikzcd}
            \Op_{\cT} \arrow[r,"h_{d}"] \arrow[d]
            & \Op_{\cT,d} \arrow[d] \\
            \Cat^{\otimes}_{\cT} \arrow[r] \arrow[r,"h_{d}"]
            & \Cat^{\otimes}_{\cT,d}
        \end{tikzcd}
    \]
    In particular, when $\cC^{\otimes}$ is a $\cT$-symmetric monoidal $d$-category, the canonical map $\cO^{\otimes} \rightarrow h_{d} \cO^{\otimes}$ induces an equivalence
    \[
        \Alg_{\cO}(\cC) \simeq \Alg_{h_{d} \cO}(\cC).
    \]
\end{corollary}
\begin{proof}
  By \cite[Prop~4.2.1]{Barkan}, the image of the fully faithful functor $\Op_{\cT} \hookrightarrow \Cat_{\cT, /\uFF_{\cT}^{\cT-\sqcup}}^{\otimes}$ is spanned by the \emph{equifibered} $\cT$-symmetric monoidal $\infty$-categories, i.e. $\cC^{\otimes}$ such that, given $T \rightarrow S$ a map of finite $\cT$-sets, the associated diagram
  \[\begin{tikzcd}
	{\cC_T} & {\cC_S} \\
	{\FF_{T}} & {\FF_S}
	\arrow[from=1-1, to=1-2]
	\arrow[from=1-1, to=2-1]
	\arrow[from=1-2, to=2-2]
	\arrow[from=2-1, to=2-2]
\end{tikzcd}\]
  is cartesian.
  We separately argue in the case $d \geq 1$ and $d = 0$ that the image of this is closed under $h_{\cT,d}$;
  this will imply that $h_{d} \Env^{/\uFF_{\cT}} \cO^{\otimes}$ corresponds with a $\cT$-$d$-operad $h_{d} \cO^{\otimes}$, which computes the left adjoint to the inclusions $\Op_{\cT,d} \subset \Op_{\cT}$ by fully faithfulness of $\Env^{/\uFF_{\cT}} \cO^{\otimes}$.
  In particular, the counit $\varepsilon\colon \cO^{\otimes} \rightarrow h_{d} \cO^{\otimes}$ will induce the counit $\Env \cO \rightarrow h_d \Env\cO$, which by \cref{Mor equation} shows that $\cO(S) \rightarrow h_d \cO(S)$ is the Postnikov $(d-1)$-truncation map.

  To prove compatibility with equifibrations, we first consider the case $d \geq 1$.
  In this case, since $h_{\cT,d}:\Cat_{\cT}^{\otimes} \rightarrow \Cat_{\cT,d}^{\otimes}$ is applied pointwise, it preserves equifibrations, so $h_{\cT,d} \Env^{/\uFF_{\cT}} \cO^{\otimes}$ corresponds with a $d$-operad $h_{\cT,d} \cO^{\otimes}$.

  The case $d = 0$ is similar, except that we are tasked with replacing equifibered $\cT$-symmetric monoidal functors with an equifibered (replete) subcategory.
  In fact, replete subcategories are precisely $(-1)$-truncated maps in $\Cat_1$, so we may do this by taking the pointwise $(-1)$-truncation functor and applying \cite[Prop~5.5.6.5]{HTT} to see that the result is equifibered.
\end{proof}

We acquire a simple lift of \cite[Prop~5.5]{Blumberg-op}.
\begin{corollary}\label{Ninfty poset corollary}
    Let $\cP^{\otimes}$ be a $\cT$-$d$-operad.
    \begin{enumerate}
      \item if $d \geq 1$, then $\Alg_{\cO}(\cP)$ is a $d$-category; hence $\Op_{\cT, d}$ is a $(d+1)$-category.
      \item if $d = 0$, then $\Alg_{\cO}(\cP)$ is either empty or contractible;
        hence $\Op_{\cT, 0}$ is a poset.
  \end{enumerate}
\end{corollary}
\begin{proof}
  In each case, the second statement follows from the first by noting that the mapping spaces in $\Op_{\cT}$ are $\Alg_{\cO}(\cP)^{\simeq}$.
  For the first statements, note that
  \[
    \Alg_{\cO}(\cP) \simeq \Alg_{h_{d} \cO}(\cP) \simeq \Fun^{\otimes}_{\cT, /\uFF_{\cT}^{\cT-\sqcup}}(\Env h_{d} \cO^{\otimes}, \Env \cP^{\otimes});
  \]
  if $d \geq 1$, then this is a subcategory of a $d$-category, so it's a $d$-category.
  If $d = 0$, then this category is either empty or contractible since we verified that the map $\Env \cP^{\otimes} \rightarrow \uFF_{\cT}^{\cT-\sqcup}$ is monic.
\end{proof}

\begin{corollary}\label{0-operads subterminal}
  $\cP^{\otimes}$ is a $\cT$-0-operad if and only if it's a sub-terminal object of $\Op_{\cT}$. 
\end{corollary}
\begin{proof}
  The mapping space criterion of monomorphisms shows that this is equivalent to the condition that
  \[
    \Alg_{h_{0} \cO}(\cP)^{\simeq} \simeq \Alg_{\cO}(\cP)^{\simeq} \rightarrow \Alg_{\cO}(\Comm_{\cT}^{\otimes})^{\simeq} \simeq * 
  \]
  is a monomorphism, i.e. $\Alg_{h_0 \cO}(\cP)^{\simeq} \in \cbr{\emptyset,*}$;
  this follows from \cref{Ninfty poset corollary}.

  On the other hand, \cref{Monadic sseq corollary} (together with Kan extensions) constructs a \emph{free $\cT$-operad on $*_S$} characterized by the property
  \[
    \Alg_{\Fr_S(*)}(\cO)^{\simeq} \simeq \cO(S);
  \]
  thus the mapping space criterion for a subterminal $\cT$-operad $\cO^{\otimes}$ implies that $\cO(S)$ is either empty or contractible for all $S$, so $\cO^{\otimes}$ is a $\cT$-0-operad.
\end{proof}

\begin{corollary}\label{Unslicing ff corollary}
  Let $I \leq J$ be related weak indexing categories.
  Then, the unslicing functor 
  \[
    \Op_I \simeq \Op_{J, /\cN_{I \infty}^{\otimes}} \rightarrow \Op_J
  \]
  is fully faithful.
\end{corollary}
\begin{proof}
  Fully faithful functors satisfy two-out-of-three, so we may replace $\Op_I \rightarrow \Op_J$ with the composite unslicing functor $\Op_I \rightarrow \Op_J \rightarrow \Op_{\cT}$, and assume $I = \FF_{\cT}$.
  The corollary is then equivalent to the statement that $\cN_{I \infty}^{\otimes} \rightarrow \Comm_{\cT}^{\otimes}$ is a monomorphism \cite[\S~5.5.6]{HTT}.
  In fact, by \cref{Ninfty sseq example}, $\cN_{I \infty}^{\otimes}$ is a $\cT$-0-operad, so this follows from \cref{0-operads subterminal}.
\end{proof}

We finish the subsection with a recognition result for $h_{d}$-equivalences;
we say that a map $\varphi\cln \cO^{\otimes} \rightarrow \cP^{\otimes}$ is an \emph{$(d-1)$-equivalence} if any of the following equivalent conditions hold.
\begin{proposition}
  Let $\varphi\cln \cO^{\otimes} \rightarrow \cP^{\otimes}$ be a morphism of $\cT$-operads.
  Then, the following are equivalent:
  \begin{enumerate}[label={(\alph*)}]
    \item The underlying $\cT$-functor $U\varphi\cln \cO \rightarrow \cP$ is essentially surjective and for all $V \in \cT$ and $S \in \FF_{\cT}$, the induced map $\tau_{\leq (d-1)}\cO(S) \rightarrow \tau_{\leq (d-1)}\cP(S)$ is an equivalence.
    \item $\varphi$ is an $h_{d}$-equivalence.
    \item The underlying $\cT$-functor $U\varphi\colon  \cO \rightarrow \cP$ is essentially surjective and for all $\cT$-symmetric monoidal $d$-categories $\cC$, the pullback $\cT$-symmetric monoidal functor
      \[
        \uAlg^{\otimes}_{\cP}(\cC) \rightarrow \uAlg_{\cO}^{\otimes}(\cC)
      \]
      is an equivalence.
    \item The underlying $\cT$-functor $U\varphi\colon  \cO \rightarrow \cP$ is essentially surjective and the pullback functor
      \[
        \Alg_{\cP}(\ucS_{\cT, \leq d-1}) \rightarrow \Alg_{\cO}(\ucS_{\cT, \leq d-1})
      \]
      is an equivalence.
  \end{enumerate}
\end{proposition}
\begin{proof}
  Suppose (a);
  in view of \cref{Symmetric sequence proposition}, to prove (b), we're tasked with proving that the maps $h_{d} \cO(\bC;D) \rightarrow h_{d} \cP(\bC;D)$ are equivalences.
  But by the natural equivalence
  \[
    \cO(S) \simeq \coprod_{(\bC,D) \in \cO_S \times \cO_V} \cO(\bC;D),
  \]
  it suffices to verify that $h_{d} \cO(S) \rightarrow h_{d} \cP(S)$ is an equivalence for each $S$.
  This follows from (a) by \cref{Operad truncation corollary}.

  Suppose (b);
  by the factorization 
  \[
    \Cat^{\otimes}_{\cT,d} \hookrightarrow \Op_{\cT,d} \hookrightarrow \Op_{\cT}
  \]
  of \cref{Operad truncation corollary}, given $\cC \in \Cat^{\otimes}_{\cT,d}$, the top map in the following is an equivalence
  \[
    \begin{tikzcd}
      {\uAlg_{h_{d} \cP}(\cC)} & {\uAlg_{h_{d} \cO}(\cC)} \\
      {\uAlg_{\cP}(\cC)} & {\uAlg_{\cO}(\cC)}
      \arrow["\sim", from=1-1, to=1-2]
      \arrow["\simeq"{marking, allow upside down}, draw=none, from=1-1, to=2-1]
      \arrow["\simeq"{marking, allow upside down}, draw=none, from=1-2, to=2-2]
      \arrow[from=2-1, to=2-2]
    \end{tikzcd}
  \]
  the bottom arrow is an equivalence from two-out-of-three, and (c) follows from \cref{Underlying conservative proposition}.
  Furthermore, (c) implies (d) by setting $\cC^{\otimes} \deq \ucS_{\cT, \leq d-1}^{\cT-\times}$.
  
  Suppose (d).
  The unique symmetric monoidal left adjoint $\cS \rightarrow \cS_{\leq d-1}$ is $\tau_{\leq d-1}$, so \cref{Fr monad proposition} implies that $\tau_{\leq d-1}\cO(S) \rightarrow \tau_{\leq d-1}\cP(S)$ is an equivalence, i.e. (a).
\end{proof}

\subsection{Arity support and Borelification}\label{Arity support subsection}
We now introduce a support stratification in terms of \emph{arity}.
\begin{construction}
    Given $\cO \in \Op_{\cT}$, the \emph{arity support of $\cO$} is the collection of maps $A\cO \subset \FF_{\cT}$ defined by
    \[
      A\cO := \cbr{\psi:T \rightarrow S \;\; \middle| \;\; \prod_{U \in \Orb(S)} \cO(T \times_U S) \neq \emptyset} \subset \FF_{\cT}\qedhere
    \]
\end{construction}

\def\Sub{\operatorname{Sub}}
\begin{observation}
  There exists no map of spaces $X_1 \times X_2 \rightarrow Y_1 \times Y_2$ if and only if $X_1,X_2 \neq \emptyset$ and $Y_i = \emptyset$ for some $i$.
  In particular,
  \begin{itemize}
    \item the existence of the composition map $\Map^{\psi}_{\cO}(T;S) \times \Map^{\psi'}_{\cO}(R;T) \rightarrow \Map^{\psi \circ \psi'}_{\cO}(R;S)$ implies that $A\cO \subset \FF_{\cT}$ is closed under composition;
    \item existence of identity operations implies that $A\cO \subset \FF_{\cT}$ contains identity arrows of its elements; and
    \item functoriality of $\Map_{(-)}^{\psi}(T;S)$ implies that $A$ forms a functor
      \[
        A\colon \Op_{\cT} \rightarrow \Sub_{\Cat}(\FF_{\cT}).\qedhere
      \]
  \end{itemize}
\end{observation}
We will frequently compute $A$ by noting that it factors as 
\[
  A\colon \Op_{\cT} \xrightarrow{\sseq_{\cT}} \Fun(\tot \uSigma_{\cT},\cS) \rightarrow \Sub_{\Cat}(\FF_{\cT}). 
\]

\begin{example}\label{Ninfty support example}
  For all $I \in \wIndCat_{\cT}$, we have $A\cN_{I \infty} = I$, so $\wIndCat_{\cT} \subset A(\Op_{\cT})$.
\end{example}

\begin{proposition}\label{Windex image proposition}
    For all $\cO^{\otimes} \in \Op_{\cT}$, the subcategory $A\cO \subset \FF_{\cT}$ is a weak indexing category;
    hence
    \[
      A(\Op_{\cT}) = \wIndCat_{\cT} \subset \Sub_{\Cat}(\FF_{\cT}).
    \]
\end{proposition}
\begin{proof}
    The second statement follows from the first by \cref{Ninfty support example}, so it suffices to prove that $\cO^{\otimes} \in \Op_{\cT}$ satisfies \cref{Restriction stable condition,Windex segal condition,Automorphism condition}.
    Indeed, \cref{Restriction stable condition} follows by unwinding definitions using existence of the arity restriction map of \cref{Restriction map}.
    Similarly, \cref{Windex segal condition} follows immediately by definition.
    Lastly, \cref{Automorphism condition} follows by existence of the $\Aut_V(S)$-action of \cref{Sigma action map}.
\end{proof}

\begin{corollary}\label{Ninfty T-0-operad corollary}
  A $\cT$-operad is a $\cT$-0-operad if and only if it's a weak $\cN_\infty$-operad.
\end{corollary}
\begin{proof}
  By \cref{Ninfty sseq example}, $\cN_{I \infty}^{\otimes}$ is a $\cT$-0-operad, so fix $\cO^{\otimes}$ a $\cT$-0-operad.
  By definition, $\pi_{\cO}$ factors as
  \[
    \cO^{\otimes} \xrightarrow{can} \Span_{A\cO}(\FF_{\cT}) \rightarrow \Span(\FF_{\cT}),
  \]
  i.e. there is a map $\varphi\cln \cO^{\otimes} \rightarrow \cN_{A\cO}^{\otimes}$.
  Moreover, for all $S$, there exists an abstract equivalence $\cO(S) \simeq \cN_{A \cO}(S)$, and since $\cO(S) \in \cbr{*,\emptyset}$, every endomorphism of $\cO(S)$ is an equivalence.
  This implies that $\cO(S) \rightarrow \cN_{A \cO}(S)$ is an equivalence for all $S \in \uFF_{\cT}$, and the result follows from \cref{Symmetric sequence proposition}.
\end{proof}

\begin{corollary}\label{All mapping spaces to ninfty prop}
  Given $\cO^{\otimes}$ a $\cT$-operad, there is an equivalence $h_{0} \cO^{\otimes} \simeq \cN_{A \cO \infty}^{\otimes}$.
  Hence, for any weak indexing category $I$, there is a natural equivalence
    \begin{equation}\label{All mapping spaces to ninfty equation}
        \Alg_{\cO}(\cN_{I \infty}) \simeq \begin{cases}
            * & A\cO \leq I,\\
            \emptyset & \mathrm{otherwise}.
        \end{cases}
    \end{equation}
\end{corollary}
\begin{proof}
The first statement follows by combining \cref{Ninfty T-0-operad corollary,Ninfty support example} with the fact that $A\cO  = Ah_{d}\cO$.
  We've already shown that $\Alg_{\cO}(\cN_{I \infty}) \simeq \Alg_{\cN_{A\cO}}(\cN_{I \infty})$ is either empty or contractible in \cref{Ninfty poset corollary}, so it suffices to characterize when there exists a map $\cN_{I \infty}^{\otimes} \rightarrow \cN_{J \infty}^{\otimes}$, i.e. a factorization $\Span_I(\FF_{\cT}) \subset \Span_J(\FF_{\cT}) \subset \Span(\FF_{\cT})$;
  this occurs if and only if $I \leq J$, yielding the corollary.
\end{proof}

The following generalization of the indexing systems theorems of \cite{Bonventre,Gutierrez,Nardin,Rubin} then immediately follows from \cref{Windex image proposition,Unslicing ff corollary,All mapping spaces to ninfty prop}.
\begin{corollary}\label{Windex image corollary}
    The functor of admissible maps admits a fully faithful right adjoint
    \begin{equation}\label{Windex image corollary equation}
    \begin{tikzcd}
	{\Op_{\cT}} && {\wIndSys_{\cT}}
	\arrow[""{name=0, anchor=center, inner sep=0}, "A" above, curve={height=-15pt}, from=1-1, to=1-3]
	\arrow[""{name=1, anchor=center, inner sep=0}, "{\cN^{\otimes}_{(-)\infty}}" below, curve={height=-15pt}, hook', from=1-3, to=1-1]
	\arrow["\dashv"{anchor=center, rotate=-90}, draw=none, from=0, to=1]
    \end{tikzcd}
    \end{equation}
    whose image consists of the weak $\cN_\infty$-operads;
    furthermore, the following are equal full subcategories of $\Op_{\cT}$:
    \[
        \Op_I = \Op_{\cT, /\cN_{I \infty}} = A^{-1}(\wIndCat_{\cT, \leq I}).
    \]
\end{corollary} 
\begin{observation}\label{Property P observation}
  By fully faithfulness of $\cN_{(-)\infty}^{\otimes}$, the adjunction associated with $A$ restricts to
  \[
    \begin{tikzcd}
      {\Op_{\cT, \EE_{\infty}^{\otimes}/}} && {\IndSys_{\cT}}
	\arrow[""{name=0, anchor=center, inner sep=0}, "A" above, curve={height=-15pt}, from=1-1, to=1-3]
	\arrow[""{name=1, anchor=center, inner sep=0}, "{\cN^{\otimes}_{(-)\infty}}" below, curve={height=-15pt}, hook', from=1-3, to=1-1]
	\arrow["\dashv"{anchor=center, rotate=-90}, draw=none, from=0, to=1]
    \end{tikzcd}
  \]
\end{observation}

Given $I \leq J$ a related pair of weak indexing sysems, let $E_I^J:\wIndCat_{\cT, \leq I} \rightarrow \wIndCat_{\cT, \leq J}$ be the evident inclusion, with right adjoint $\Bor_I^J = (-) \cap \FF_I:\wIndCat_{\cT, \leq J} \rightarrow \wIndCat_{\cT, \leq kI}$. 
These are push-pull adjunctions;
following in form, we write the corresponding \emph{unslicing functor} as
\[
  E_{I}^J\cln \Op_I \simeq \Op_{J, /\cN_{I \infty}^{\otimes}} \rightarrow \Op_J.
\]
This has a right adjoint 
\[
  \Bor_I^J\cln \Op_J \rightarrow \Op_{J, /\cN_{I \infty}^{\otimes}} \simeq \Op_I
\]
given by pullback along the unique map $\cN_{I \infty}^{\otimes} \rightarrow \cN_{J \infty}^{\otimes}$.
These map to push-pull along the inclusion $\iota_I^J \cln \Span_I(\FF_{\cT}) \rightarrow \Span_J(\FF_{\cT})$ along $\Tot\colon \Op_{I} \rightarrow \Cat_{/\Span_I(\FF_{\cT})}$ and similar for $J$ \cite[\S~4]{Barkan}.
Hence they intertwine with $A$, i.e.
\[
  E_I^J A \cO = A E_I^J \cO; \hspace{50pt} \Bor_I^J A \cO = A \Bor_I^J \cO.
\]

\begin{corollary}\label{I-borelification corollary}
  For $I \leq J$ weak indexing systems, the functor $E_I^J\cln \Op_I \rightarrow \Op_J$ is an inclusion of a colocalizing $\cT$-subcategory 
    \[\begin{tikzcd}
	{\uOp_{I}^{\otimes}} && {\uOp_{J}^{\otimes}}
	\arrow[""{name=0, anchor=center, inner sep=0}, "{E^J_{I}}"{description}, curve={height=-12pt}, hook', from=1-1, to=1-3]
	\arrow[""{name=1, anchor=center, inner sep=0}, "{\Bor^J_I}"{description}, curve={height=-12pt}, from=1-3, to=1-1]
	\arrow["\dashv"{anchor=center, rotate=-90}, draw=none, from=0, to=1]
\end{tikzcd}\]
    whose terminal object is $\cN_{I \infty}^{\otimes}$.
    Furthermore, there is are equivalences
    \begin{align*}
        E_I^{I'} \cN_{J \infty}^{\otimes} &\simeq \cN_{E_I^{I'} J \infty}^{\otimes}\\
        \Bor_I^{I'} \cN_{J \infty}^{\otimes} &\simeq \cN_{\Bor_I^{I'} J \infty}^{\otimes}.\qedhere
    \end{align*}
\end{corollary}
\begin{proof}
  The first sentence follows by the above argument.
  The computations follow by examining the structure spaces of the resulting $\cT$-operads.
\end{proof}

\subsection{The genuine operadic nerve}\label{Discrete genuine nerve subsection}
We now concern ourselves with comparisons to other notions of equivariant operads.
Throughout this subsection, we assume $\cT$ is a 1-category;
for instance, we may specialize to $V$-operads for $V \in \cT$ an orbit.
We begin in \cref{Nerve subsubsection} by reviewing Bonventre's genuine operadic nerve $N^{\otimes}$;
we detour in \cref{Restriction nerve subsubsection} to verify that it is compatible with restriction.
Then, in \cref{Genuine nerve subsubsection} we show that $N^{\otimes}$ possesses a conservative total right derived functor of $\infty$-categories.
We end in \cref{Discrete nerve subsubsection} by noting that this restricts to an equivalence on $\cT$-1-operads and describing the corresponding discrete theory of algebras.
In all sections, we assume that $\cT$ is an atomic orbital 1-category.
\subsubsection{The 1-categorical nerve}\label{Nerve subsubsection}
Recall the $\cT$-space $\uSigma_{\fC}$ of \cref{Sigma V construction}.
\cite{Bonventre-nerve} introduced a specialization of the following.
\begin{definition}\label{Discrete genuine definition}
  A $\fC$-colored genuine $\cT$-operad in a symmetric monoidal 1-category $\cV$ the data of:
  \begin{enumerate}
    \item a $\fC$-symmetric sequence $\cO(-;-)\colon \tot \uSigma_{\fC} \rightarrow \cV$,
    \item for all $V \in \cT$ and $C \in \fC_V$, a distinguished ``identity'' element $1_C \in \cO(C;C)$, and
    \item for all composable data $\prn{(\bC;D), (\bB_U;C_U)_{U \in \Orb(S)}}$ lying over a map $T \rightarrow S$,
      a Borel $\Aut_V(S) \times \prod_{U \in \Orb(S)} \Aut_U(T_U)$-equivariant ``composition'' map
      \[
        \gamma:\cO(\bC;D) \otimes \bigotimes_{U \in \Orb(S)}\cO (\bB_U; C_U) \rightarrow \cO\prn{\bB;D}
      \]
  \end{enumerate}
  subject to the following compatibilities:
  \begin{enumerate}[label={(\alph*)}]
    \item (restriction-stability of the identity) for all $U \rightarrow V$ and $C \in \fC_V$, the restriction map 
      \[
        \Res_U^V\colon \cO(C;C) \rightarrow \cO\prn{\Res_U^V C;\Res_U^V C}
      \]
      sends $1_{C}$ to $1_{\Res_U^V C}$;
    \item (restriction-stability of composition) for all $U \rightarrow V$, the following commutes
      \[\begin{tikzcd}[column sep=large, row sep=small]
	{\cO(\bC;D) \otimes \mspc \bigotimes\limits_{U \in \Orb(S)} \mspc \cO(\bB_U;C_U)} & {\cO(\bB;D)} \\
  {\cO\prn{\Res_W^V \bC;\Res_W^V D} \otimes \mspc \bigotimes\limits_{U \in \Orb(S)} \mspc \cO\prn{\Res_W^V \bB_U;C_U}} & {\cO\prn{\Res_W^V \bB;D}}
	\arrow["\gamma", from=1-1, to=1-2]
	\arrow["{\Res_V^W}", from=1-1, to=2-1]
	\arrow["{\Res_V^W}", from=1-2, to=2-2]
	\arrow["\gamma", from=2-1, to=2-2]
\end{tikzcd}\]
    \item (unitality) for all $S \in \FF_V$, the following diagram commutes
      \[\begin{tikzcd}
	{\cO(\bC;D)} & {\cO(\bC;D) \otimes \mspc \bigotimes\limits_{U \in \Orb(S)} \mspc \cO(C_U;C_U)} \\
	{\cO(D;D) \otimes \cO(\bC;D)} & {\cO(\bC;D)}
	\arrow["{\prn{\id, \prn{\cbr{1_U}}}}", from=1-1, to=1-2]
	\arrow["{\prn{1_V,\id}}"', from=1-1, to=2-1]
	\arrow[Rightarrow, no head, from=1-1, to=2-2]
	\arrow["\gamma", from=1-2, to=2-2]
	\arrow["\gamma"{description}, from=2-1, to=2-2]
\end{tikzcd}\]
    \item (associativity) For all $S \in \FF_V$, $(T_U) \in \FF_S$ writing $T \deq \coprod\limits_U^S T_U$, and $(R_W) \in \FF_T$ writing $R \deq \coprod\limits_W^T R_W$, the following diagram commutes
      \[\begin{tikzcd}[row sep=small]
        {\prn{\cO\prn{\bC;D} \otimes \mspc \bigotimes\limits_{U \in \Orb(S_U)} \mspc \cO\prn{\bB_U;C_U}} \otimes \mspc \bigotimes\limits_{\stackrel{U \in \Orb(S)}{W \in \Orb(T_U)}} \mspc \cO\prn{\bA_W;B_W}} & {\cO\prn{\bB;D} \otimes \mspc \bigotimes\limits_{W \in \Orb(T)} \mspc \cO\prn{\bA_W;B_W}} \\
	{\cO(\bC;D) \otimes \mspc \bigotimes\limits_{U \in \Orb(S)} \prn{\cO(\bB_U;C_U) \otimes \mspc \bigotimes\limits_{W \in \Orb(T_U)} \mspc \cO(\bA_W;B_W)}} \\
	{\cO(\bC;D) \otimes \mspc \bigotimes\limits_{U \in \Orb(S)} \mspc \cO\prn{\bA_U;C_U}} & {\cO\prn{\bA;D}}
  \arrow[from=1-1, to=1-2,"\gamma"]
  \arrow[from=1-2, to=3-2,"\gamma"]
	\arrow[Rightarrow, no head, from=2-1, to=1-1]
  \arrow[from=2-1, to=3-1,"\gamma \; "']
  \arrow[from=3-1, to=3-2,"\gamma"]
\end{tikzcd}\]
\end{enumerate}
  A morphism of $\fC$-colored discrete $\cT$-operads in $\cV$ is a map of $\fC$-symmetric sequences in $\cV$ preserving each $1_C$ and intertwining $\gamma$;
  we refer to the resulting 1-category as $g\Op^{\fC}_{\cT}(\cV)$.
\end{definition}
\begin{construction}
  Given a map of coefficient systems $f\colon \fC \rightarrow \fC'$, there is an induced map of $\cT$-operads $\triv\prn{\fC}^{\otimes} \rightarrow \triv\prn{\fC'}^{\otimes}$ yielding a $\cT$-functor $f\colon \uSigma_{\fC} \rightarrow \uSigma_{\fC'}$, and hence a precomposition functor 
  \[
    f^*\colon \Fun\prn{\Tot \uSigma_{\fC'}, \cV} \rightarrow \Fun\prn{\Tot \uSigma_{\fC}, \cV}.
  \]
  These are the cocartesian transport functors of a cocartesian fibration, which we call $\mathrm{SSeq}^\bullet_{\cT}$.
  A \emph{morphism of colored discrete operads} $\cO^{\otimes} \rightarrow \cP^{\otimes}$ is a morphism of their underlying objects of $\mathrm{SSeq}^{\bullet}_{\cT}$ which is compatible with identities and composition.
  We refer to the resulting 1-category as $g\Op_{\cT}(\cV)$.
\end{construction}

We write $s\Op^{\fC}_{\cT} \deq g\Op^{\fC}_{\cT}(s\Set)$ and $s\Op_{\cT} \deq g\Op_{\cT}(s\Set)$.
In \cite{Bonventre}, a model structure was given to $s\Op^{\oc}_G$;
this was later shown to be Quillen equivalent to several other model categorical variations on $G$-operads (e.g. \cite[Tab~1]{Bonventre_dendroidal}).
Complementarily, \cite{Bonventre-nerve} constructed a \emph{genuine operadic nerve} functor of 1-categories
\[
  N^{\otimes}\cln s\Op_G \rightarrow s\Set^+_{/\prn{\Tot \uFF_{G,*}, \mathrm{Ne}}}
\]
whose restriction $g\Op_G(\mathrm{Kan})$ lands in fibrant objects in Nardin-Shah's model structure \cite[\S~2.6]{Nardin}, and hence presents $G$-operads.

Moreover, $g\Op_{\cT}(\mathrm{Kan})$ agrees with the \emph{fibrant simplicial colored $\cT$-operads} of \cite[Def~2.5.4]{Nardin};
Nardin-Shah \cite{Nardin} construct an analogous nerve functor
\[
  N^{\otimes}\cln g\Op_\cT(\mathrm{Kan}) \rightarrow s\Set^+_{/\prn{\Tot \uFF_{\cT,*}}, \mathrm{NE}}
\]
whose specialization to $\cT = \cO_G$ agrees with Bonventre's nerve.

These nerve functors can be understood as taking $\cO \in g\Op^{\fC}_{\cT}(\mathrm{Kan})$ to the $\mathrm{Kan}$-enriched category over $\Tot \uFF_{\cT,*}$ with $\Ob \prn{ \cO_S} = \fC_S$ and with mapping spaces
\[
  \Map_{\cO^{\otimes}}(\bC,\bD) \deq \coprod_{S \leftarrow \pi_{\cO}\bC \rightarrow \pi_{\cO}\bD} \prod_{U \in \Orb(\pi_{\cO}(\bD))} \cO(\bC_U; D_U),
\]
with evident composition and mapping down to $\Map_{\Tot \uFF_{\cT,*}}(\pi_{\cO} \bC, \pi_{\cO} \bD)$ via the evident forgetful map.

\subsubsection{Restriction and the nerve}\label{Restriction nerve subsubsection}
\begin{construction}
  Let $W \in \cT$ be a distinguished object.
  Then, the \emph{restriction functor} 
  \[
    \Res_W^{\cT} \cln g\Op_{\cT}(\cV) \rightarrow g\Op_{W}(\cV) \deq g\Op_{\cT_{/W}}(\cV)
  \]
  acts on underlying symmetric sequences via pullback along the map $\Tot \Res_W^{\cT} \uSigma_{\fC} \rightarrow \Tot \uSigma_{\fC}$, with the data $1_V$ and $\gamma$ defined in $\Res_W^{\cT} \cO^{\otimes}$ by specialization from $\cO^{\otimes}$.  
\end{construction}
We define restriction $\Res_W^{\cT}\cln \Cat_{s\Set, /\Tot \uFF_{\cT,*}} \rightarrow \Cat_{s\Set, /\Tot \uFF_{W,*}}$ by pullback along $\Tot \uFF_{W,*} \rightarrow \Tot \uFF_{\cT,*}$.
\begin{proposition}
  Let $\cO$ be a one-color simplicial genuine $\cT$-operad.
  There is a natural isomorphism of simplicial categories $N^{\otimes} \Res_W^{\cT} \cO \simeq \Res_W^{\cT} N^{\otimes}\cO$ over $\Tot \uFF_{\cT,*}$.
\end{proposition}
\begin{proof}
  We may construct a functor $N^{\otimes} \Res_W^{\cT} \cO^{\otimes} \rightarrow N^{\otimes} \cO^{\otimes}$ over $\Tot \uFF_{\cT,*}$ sending the object over a $(V \rightarrow W)$-set $S_{V \rightarrow W}$ to its underlying $V$-set $S$ and acting on mapping spaces by taking coproducts of the tautological equivalence $\Res_W^{\cT} \cO(S_{V \rightarrow W}) \simeq \cO(S_V)$.
  This constructs a natural diagram
  \[
    \begin{tikzcd}[row sep=small]
      {\Tot_{\cT} N^{\otimes} \Res_W^{\cT} \cO^{\otimes}} \\
  & {\Tot_{\cT} \Res_W^{\cT} N^{\otimes} \cO^{\otimes}} & {\Tot_{\cT} N^{\otimes} \cO^{\otimes}} \\
	& {\Tot \uFF_{W,*}} & {\Tot \uFF_{\cT,*}}
	\arrow[dashed, from=1-1, to=2-2,"F"]
	\arrow[curve={height=-18pt}, from=1-1, to=2-3]
	\arrow[curve={height=18pt}, from=1-1, to=3-2]
	\arrow[from=2-2, to=2-3]
	\arrow[from=2-2, to=3-2]
	\arrow[from=2-3, to=3-3]
	\arrow[from=3-2, to=3-3]
\end{tikzcd}
  \]
  Since $\pi_{N^{\otimes} \Res_W^{\cT} \cO^{\otimes}}$ and $\pi_{\Res_W^{\cT} N^{\otimes} \cO^{\otimes}}$ are $\pi_0$-isomorphisms, $F$ is as well, so it is essentially surjective.
  It follows by unwinding definitions that $F$ is fully faithful, and hence an equivalence.
\end{proof}
The above restriction functor implements restriction of $\cT$-operads, so we have the following. 
\begin{corollary}
  There is a natural equivalence of $W$-operads $\Res_W^{\cT} N^{\otimes} \cO \simeq N^{\otimes} \Res_W^{\cT} \cO$.
\end{corollary} 

The main reason we went to this trouble is for the following example.
\begin{example}
  Let $G$ be a finite group and $V$ be a real orthogonal $G$-representation.
  Let $D_V$ be a genuine $G$-operad which is equivalent to the little $V$-disks operad (see \cite[\S~3.9]{Horev}).
  Then, given $K \subset H \subset G$, and $S \in \FF_K$, we have a tautological equivalence
  \[
    \Res_H^G D_V(S) \simeq \Conf_S^K(\Res_K^G V) \simeq \Conf_S^K(\Res_K^H \Res_H^G V) \simeq D_{\Res_H^G V}(S)
  \]
  which intertwines with the composition rule in $D_V$;
  writing $\EE_V^{\otimes} \deq N^{\otimes} D^V$, we acquire an equivalence
  \[
    \Res_H^G \EE^{\otimes}_V \simeq \EE_{\Res_H^G V}^{\otimes}\qedhere.
  \]
\end{example}

\subsubsection{The conservative $\infty$-categorical lift}\label{Genuine nerve subsubsection}\,
$N^{\otimes}$ has homotopical properties.
\begin{proposition}
    $N^{\otimes}$ preserves and reflects weak equivalences between one-color genuine $G$-operads in Kan complexes. 
\end{proposition}
\begin{proof}
    By \cite[Thm~II,Prop~4.31]{Bonventre}, the functor $U:s\Op^{\oc}_G \rightarrow \Fun(\uSigma_G,\sSet)$ is monadic and $s\Op^{\oc}_G$ admits the transferred model structure from the projective model structure on $\Fun\prn{\Tot \uSigma_G,\sSet_{\mathrm{Quillen}}}$;
    in particular, $U$ preserves and reflects weak equivalences.

    It is not hard to see that $\sseq$ is right-derived from a functor
    \[
      \ssseq:\sSet^{+,\oc}_{/(\uFF_{\cT},Ne)} \rightarrow \Fun\prn{\Tot \uSigma_G,\sSet_{\mathrm{Quillen}}}_\mathrm{Proj}
    \]
    setting $\cO_{\sseq}(S) := \pi_{\cO}^{-1}(\Ind_H^G S \rightarrow G/H)$;
    by \cref{Symmetric sequence proposition} $\sseq$ is conservative, so $\ssseq$ preserves and reflects weak equivalences between fibrant objects.
    Hence it suffices unwind definitions and note that the following diagram commutes
    \[
        \begin{tikzcd}
          s\Op^{\oc}_G \arrow[r,"N^{\otimes}"] \arrow[dr,"U" below] 
          & \sSet^{+,{\oc}}_{/(\uFF_{G},Ne)} \arrow[d,"\ssseq"]\\
            & \Fun(\Tot \uSigma_G,\sSet)
        \end{tikzcd}
    \]
\end{proof}
In fact, the one-color assumption was unnecessary.
We say that a map of genuine simplicial $G$-operads $\varphi\colon \cO \rightarrow \cP$ is a \emph{weak equivalence} if it is an isomorphism on coefficient systems and for all profiles $(\bC;D)$, the map $\cO(\bC;D) \rightarrow \cP(\varphi \bC;\varphi D)$
is a weak equivalence.
These weak equivalences satisfy two-out-of-three (in fact, two-out-of-six) by the same property for isomorphisms and for weak equivalences of simplicial sets.
\begin{proposition}
  $N^{\otimes}$ preserves and reflects weak equivalences between arbitrary genine $G$-operads in Kan complexes.
\end{proposition}
\begin{proof}
  It is not hard to see that $N^{\otimes}$ preserves and reflects the property of \emph{inducing isomorphism on coefficient systems of colors}, so we may fix a coefficient system of sets of colors $\fC$ and verify that 
  \[
    N^{\otimes}_{\fC}\colon g\Op_{G}^{\fC}(\mathrm{Kan}) \rightarrow s\Set^+_{/(\uFF_{G,*}, NE)}
  \]
  preserves and reflects weak equivalences.
  Thankfully, we have the same tools as in the one-color case;
  \cref{Symmetric sequence proposition} constructs a functor $s\sseq:\sSet^{+,\fC}_{/(\uFF_{\cT},Ne)} \rightarrow \Fun\prn{\Tot \uSigma_{\fC}, \sSet_{\mathrm{Quillen}}}_{\mathrm{Proj}}$ which preserves and reflects weak equivalences between fibrant objects, and $N^{\otimes}_{\fC}$ is a functor over $\Fun \prn{\Tot \uSigma_{\fC},\sSet_{\mathrm{Quillen}}}$;
  by two-out-of-three for weak equivalences, $N^\otimes$ preserves and reflects weak equivalences between fibrant objects.
\end{proof}

Defining the $\infty$-category $g\Op_G \deq g\Op_G(\mathrm{Kan})[\mathrm{weq}^{-1}]$, we acquire a multi-color version of \cref{Nerve cool theorem} by functoriality of Hammock localization.

\begin{remark}\label{O(n) nerve remark}
  In \cite{Bonventre}, another nerve functor $\iota_*\colon g\Op^{\oc}_e(\sSet)^{BG} \rightarrow g\Op^{\oc}_G(\sSet)$ was constructed and shown to furnish a Quillen equivalence for a model structure on $g\Op^{\oc}_e(\sSet)^{BG}$ whose weak equivalences and fibrations are preserved and reflected by the \emph{graph subgroup} fixed points $\prod_{n \in \NN} \cbr{\cO(n)^{\Gamma} \mid \Gamma \in \cO_{G \times \Sigma_n,\Gamma}}$ for the Quillen model structure on $\sSet$.
  In particular, under the equivalence $\Un_{\cO_G}\cO_{G \times \Sigma_n, \Gamma} \simeq B_G \Sigma_n$, an object $\cO \in g\Op_e(\sSet)^{BG}$ has an $n$-ary $B_G \Sigma_n$ space $\cO(n)$.
  In fact, unwinding definitions using \cite[Rmk~4.38]{Bonventre}, we find that there is an equivalence
  \[
    \underline{N^{\otimes} \iota_* \cO}(n) \simeq \cO(n).\qedhere
  \]
\end{remark}

\subsubsection{The discrete genuine nerve is an equivalence}\label{Discrete nerve subsubsection}
Note that the fully faithful inclusion of discrete simplicial sets $\Set \hookrightarrow \sSet$ is product-preserving, so it induces a fully faithful functor $g\Op_{\cT}(\Set) \hookrightarrow g\Op_{\cT}(\sSet)$.
We refer to these as \emph{discrete genuine $\cT$-operads}.
We're concerned with relating this to $\cT$-1-categories, beginning with the following.

\begin{observation}\label{1-nerve}
  For all $\cO \in g\Op_{\cT}(\Set)$, $N^{\otimes} \cO$ is a $\cT$-1-operad.
\end{observation}
Conversely, from the data of a $\cT$-1-operad $\cO$, the data of a discrete genuine $\cT$-operad $\cO(-)$ is supplied by \cref{Discretization remark}.

\begin{proposition}\label{Discrete operads proposition}
  $N^{\otimes}$ descends to a functor $g\Op_{\cT}(\Set) \rightarrow \Op_{\cT,1}$ with quasi-inverse $\cO(-)$.
\end{proposition}
\begin{proof}
  By \cref{1-nerve}, $N^{\otimes}$ restricts as above.
  Thus it suffices to prove that the compositions $g\Op_{\cT}(\Set) \rightarrow g\Op_{\cT}(\Set)$ and $\Op_{\cT,1} \rightarrow \Op_{\cT,1}$ are naturally equivalent to the identity;
  this follows immediately after unwinding definitions. 
\end{proof}

Now having an explicit combinatorial model for $\cT$-1-operads, we focus on algebras using the following.
\begin{construction}
  Let $\cO^{\otimes}$ be a $\cT$-operad and $\cP \subset \cO$ a full $\cT$-subcategory.
  Then, we define the full $\cT$-subcategory $\Tot_{\cT} \cP^{\otimes} \subset \Tot_{\cT} \cO^{\otimes}$ to be spanned by the tuples $\bC \in \cO_S$ such that, for each $U \in \Orb(S)$, $C_U \in \cP$.
  $\cP^{\otimes}$ is a $\cT$-operad and $\cP^{\otimes} \rightarrow \cO^{\otimes}$ a map of $\cT$-operads \cite[\S~2.9]{Nardin};
  we call this the \emph{full $\cT$-suboperad spanned by $\cP$}.

  In particular, if $X \in \Gamma^{\cT} \cO$ is a $\cT$-object in $\cO$, we define the \emph{endomorphism $\cT$-operad $\End_X^{\otimes} \subset \cC^{\otimes}$ of $X$} to be the full $\cT$-suboperad of $\cO^{\otimes}$ spanned by $\cbr{X}$.
\end{construction}

\begin{observation}\label{End observation}
  Suppose $\cC^{\otimes}$ is an $I$-symmetric monoidal $\infty$-category and $X \in \Gamma^{\cT} \cC$.
  Then, $\End_X$ has underlying $\cT$-symmetric sequence $\End_X(S) \simeq \Map_{\cC_V}(X^{\otimes S}_V,X_V)$ for $S \in \uFF_I$, identity element $1_V = \id_{X_V}$, and composition map given by composition of maps
  \[
    \gamma(\mu_S; (\mu_{T_U})) \cln X^{\otimes T}_V \simeq \bigotimes_U^S X^{\otimes T_U}_U \xrightarrow{\bigotimes_U^S \mu_{T_U}} X^{\otimes S}_V \xrightarrow{\mu_S} X_V.
  \]
  In particular, if $\cC^{\otimes}$ is an $I$-symmetric monoidal $d$-category, then $\End_X \cC^{\otimes}$ is a $\cT$-$d$-operad.
\end{observation}

In general, an $\cO$-algebra in $\cC^{\otimes}$ may be viewed as the information of its underlying object $X$ together with the factored map $\cO^{\otimes} \rightarrow \End_X^{\otimes} \hookrightarrow \cC^{\otimes}$.
The following proposition follows by unwinding definitions.
\begin{proposition}\label{1-algebra proposition}
  If $\cC^{\otimes}$ is a $\cT$-symmetric monoidal 1-category and $X,Y$ are $\cO$-algebras in $\cC^{\otimes}$, then the hom set $\Hom_{\Alg_{\cO}(\cC)}(X,Y) \subset \Hom_{\cC}(X,Y)$ consists of those maps such that the following diagram of $\cT$-operads commutes:
  \[
    \begin{tikzcd}[row sep=tiny]
      & \End_X^{\otimes} \arrow[dd]\\
      \cO^{\otimes} \arrow[ru] \arrow[rd]\\
      & \End_Y^{\otimes}
    \end{tikzcd}
  \]
\end{proposition} 

For the sake of comparison, we will propose one more model for discrete $I$-commutative algebras.
\begin{definition}\label{Strict I-commutative algebra}
  Let $I$ be a one-color weak indexing category.
  Then, a \emph{strict $I$-commutative algebra in $\cC$} is the data of a $\cT$-object $X$ together with $\Aut_V S$-invariant maps $\mu_S:X^{\otimes S}_V \rightarrow X_V$ for all $S \in \FF_{I,V}$ subject to the following conditions:
  \begin{enumerate}
    \item (restriction-stability) The functor $\Res_U^V$ takes $\mu_S$ to $\mu_{\Res_U^V S}$.
    \item (identity) $\mu_{*_V}$ is the identity for all $V$.
    \item (commutativity) for all $S$-tuples $(T_U) \in \FF_{I,S}$, writing $T = \coprod\limits_U^S T_U$, the following diagram commutes:
      \begin{equation}\label{I-calg associativity diagram}
      \begin{tikzcd}
	{\bigotimes\limits_U^S X^{\otimes T_U}_U} & {X^{\otimes S}_V} \\
	{X^{\otimes T}_V} & {X_V}
	\arrow["{\prn{\mu_{T_U}}}", from=1-1, to=1-2]
	\arrow["\simeq"{marking, allow upside down}, draw=none, from=1-1, to=2-1]
	\arrow["{\mu_S}"{description}, from=1-2, to=2-2]
	\arrow["{\mu_T}", from=2-1, to=2-2]
\end{tikzcd}
  \end{equation}
  \end{enumerate}
\end{definition}

\begin{remark}\label{Unitality remark}
  In the case that $I$ is unital, we acquire a form of unitality from the commutativity condition;
  choosing $S = S' \sqcup *_V$, and choosing $T_U$ to be empty for all summands other than the distinguished fixed point and $*_V$ for the distinguished fixed point, we acquire a unitality diagram
      \[\begin{tikzcd}[row sep=tiny]
	& {X^{\otimes S \sqcup *_V}_V} \\
	{X_V} && {X_V}
	\arrow[from=1-2, to=2-3]
	\arrow[from=2-1, to=1-2]
	\arrow[Rightarrow, no head, from=2-1, to=2-3]
\end{tikzcd}\]
\end{remark}

\begin{proposition}
  If $\cC^{\otimes}$ is a $\cT$-symmetric monoidal 1-category, then the categories of $I$-commutative algebras and strict $I$-commutative algebras in $\cC$ agree. 
\end{proposition}
\begin{proof}
  This follows from \cref{End observation}, noting that $\Map(\cN_{I \infty}^{\otimes}, \End_X^{\otimes}) \simeq \Map(\cN_{I\infty}^{\otimes}, \Bor_I^{\cT} \End_X^{\otimes})$ and unwinding definitions using \cref{Discrete operads proposition}.
\end{proof}

Let $X,Y$ be $I$-commutative algebras and $f\cln X \rightarrow Y$ a morphism between their underlying $\cT$-objects.
For the rest of this subsection, we assume familiarity with the techniques of \cite{Windex}.
We will say that $f$ \emph{intertwines at $S \in \FF_{I,V}$} if the following diagram commutes:
\[
  \begin{tikzcd}[row sep = small]
	{X_V^{\otimes S}} & {X_V} \\
	{Y_V^{\otimes S}} & {Y_V}
	\arrow[from=1-1, to=1-2]
	\arrow[from=1-1, to=2-1]
	\arrow[from=1-2, to=2-2]
	\arrow[from=2-1, to=2-2]
\end{tikzcd}
\]
Define the collection $\uFF_{t(f)} \subset \uFF_I$ by
\[
  \FF_{t(f),V} \deq \left\{S \;\; \middle| \;\; f \text{ intertwines at } S\right\} \subset \FF_{I,V}
\]
The fact that $f$ is a map of $\cT$-objects implies that $\uFF_{t(f)}$ is restriction stable.
Hence $\uFF_{t(F)} \subset \uFF_I$ is a full $\cT$-subcategory.
\begin{proposition}
  $\uFF_{t(f)}$ is a one-color weak indexing system. 
\end{proposition}
\begin{proof}
  It follows by unwinding definitions that $c(t(f)) = \cT$, so we're left with proving that $\uFF_{t(f)}$ is closed under self-indexed coproducts.
  To that end, fix $S \in \FF_{t(f),V}$ and $(T_U) \in \FF_{t(f),S}$ and write $T \deq \coprod_U^S T_U$. 
  By the associativity condition, we're tasked with proving that the outer rectangle of the following diagram commutes
  \[
    \begin{tikzcd}[row sep=small]
      {X_V^{\otimes T}} & {\bigotimes_U^S X_U^{T_U}} & {X_V^{\otimes S}} & {X_V} \\
      {Y_V^{\otimes T}} & {\bigotimes_U^S Y_U^{T_U}} & {Y_V^{\otimes S}} & {Y_V}
      \arrow["\simeq"{description}, draw=none, from=1-1, to=1-2]
      \arrow[from=1-1, to=2-1]
      \arrow[from=1-2, to=1-3]
      \arrow[from=1-2, to=2-2]
      \arrow[from=1-3, to=1-4]
      \arrow[from=1-3, to=2-3]
      \arrow[from=1-4, to=2-4]
      \arrow["\simeq"{description}, draw=none, from=2-1, to=2-2]
      \arrow[from=2-2, to=2-3]
      \arrow[from=2-3, to=2-4]
    \end{tikzcd}
\]
  The left inner rectangle is commutative by definition;
  the right inner rectangle is commutative by the assumption $S \in \FF_{t(f),V}$;
  the middle inner rectangle is commutative by taking a (pointwise) $S$-indexed tensor product of the commutativity diagrams for each $T_U$. 
\end{proof}
Recall the \emph{sparse} $V$-sets of \cref{I-commutative monoids subsection}.
\begin{corollary}\label{Indexing systems map of calgs}
  Let $I$ be an almost essentially unital weak indexing system.
  Then,
  \begin{enumerate}
    \item $f$ is a map of $I$-commutative algebras if and only if it intertwines at all sparse $I$-admissible $V$-sets.
    \item If $I$ is an indexing system, then $f$ is a map of $I$-commutative algebras if and only if it intertwines at $2 \cdot *_V$  and at all $I$-admissible \emph{transitive} $V$-sets for all $V \in \cT$.
  \end{enumerate}
\end{corollary}
\begin{proof}
  By color-borelification, we may assume $I$ is almost-unital.
  In each case, it suffices to show that the applicable $V$-sets generate $\uFF_I$ as a weak indexing category;
  this is \cref{Sparse generation prop}.
\end{proof}

\begin{corollary}\label{Chan comparison corollary}
  If $\cC$ is a $G$-symmetric monoidal 1-category and $I$ is an indexing system, then $I$-commutative algebras in $\cC$ are equivalent to \cite[Def~5.6]{Chan}'s ``$I$-commutative monoids'' over $\cC$.
\end{corollary}
To prove this, suppose $X$ is a $G$-object equipped with the data of \cite[Def~5.6]{Chan}, i.e. a unit element $\eta\colon * \rightarrow X_G$, a binary multiplication $+\colon X_G \otimes X_G \rightarrow X_G$, and for all $I$-admissible transitive $H$-sets $[H/K]$, a map $\mu_K^H\colon N_K^H X_K \rightarrow X_H$.
We let $X_H$ have the restricted commutative monoid structure.
Given $S \in \FF_{I,H}$, we define the map $\mu_S\colon X_H^{\otimes S} \rightarrow X_H$ by
\[
  \mu_S \deq \sum_{[H/K] \in \Orb(S)} \mu_K^H;
\]
this is well defined by condition (3) of \cite[Def~5.6]{Chan}.
\begin{lemma}\label{Associativity lemma}
  \cref{I-calg associativity diagram} commutes for $\mu_S$.
\end{lemma}
\begin{proof}
  To verify this, we must verify that the outer square in the following diagram commutes.
  \[
\begin{tikzcd}[ampersand replacement=\&]
	{\bigotimes\limits_{[H/K] \in \Orb(S)} N_K^H \bigotimes\limits_{[K/J] \in \Orb(T_K)}N_J^KX_J} \& {\bigotimes\limits_{[H/K] \in \Orb(S)} N_K^H \bigotimes\limits_{[K/J] \in \Orb(T_K)}X_K} \& {\bigotimes\limits_{[H/K] \in \Orb(S)} N_K^HX_K} \\
	{\bigotimes\limits_{[H/K] \in \Orb(S)}  \bigotimes\limits_{[K/J] \in \Orb(T_K)}N_K^HN_J^KX_J} \& {\bigotimes\limits_{[H/K] \in \Orb(S)}  \bigotimes\limits_{[K/J] \in \Orb(T_K)}N_K^HX_K} \\
	{\bigotimes\limits_{[H/K] \in \Orb(S)}  \bigotimes\limits_{[K/J] \in \Orb(T_K)}N_J^HX_J} \& {\bigotimes\limits_{[H/K] \in \Orb(S)}  \bigotimes\limits_{[K/J] \in \Orb(T_K)}X_H} \& {\bigotimes\limits_{[H/K] \in \Orb(S)} X_H} \\
	{\bigotimes\limits_{[H/J] \in \Orb(T)}N_J^HX_J} \& {\bigotimes\limits_{[H/J] \in \Orb(T)} X_H} \& {X_H}
	\arrow["{\prn{N_K^H(\mu_J^K)}}", from=1-1, to=1-2]
	\arrow["\simeq"{marking, allow upside down}, draw=none, from=1-1, to=2-1]
	\arrow["{\prn{N_K^H(+)}}", from=1-2, to=1-3]
	\arrow["\simeq"{marking, allow upside down}, draw=none, from=1-2, to=2-2]
	\arrow["{(\mu_K^H)}", from=1-3, to=3-3]
	\arrow["{\prn{(N_K^H \mu_J^K)}}", from=2-1, to=2-2]
	\arrow["\simeq"{marking, allow upside down}, draw=none, from=2-1, to=3-1]
	\arrow["{\prn{(\mu_K^H)}}", from=2-2, to=3-2]
	\arrow["{\prn{(\mu_J^H)}}", from=3-1, to=3-2]
	\arrow["\simeq"{marking, allow upside down}, draw=none, from=3-1, to=4-1]
	\arrow["{\prn{+}}", from=3-2, to=3-3]
	\arrow["\simeq"{marking, allow upside down}, draw=none, from=3-2, to=4-2]
	\arrow["{+}", from=3-3, to=4-3]
	\arrow["{\prn{\mu_J^H}}", from=4-1, to=4-2]
	\arrow["{+}", from=4-2, to=4-3]
\end{tikzcd}
  \]
  The top left square commutes by definition and the bottom left square commutes by condition (1) of \cite[Def~5.6]{Chan}.
  The middle left square commutes by condition (2) of \cite[Def~5.6]{Chan}.
  The top right square commutes by the fact that $\mu_K^H$ is a monoid homomorphism.
  The bottom right square is the commutativity law for the monoid $X_H$.
\end{proof}

Moreover, conditions (3-4) of \cite[Def~5.6]{Chan} implies the following.
\begin{lemma}\label{Restriction lemma}
  $\Res_H^G$ takes $\mu_S$ onto $\mu_{\Res_H^G S}$.
\end{lemma}

\begin{proof}[Proof of \cref{Chan comparison corollary}]
  Given $X$ an $I$-commutative algebra, we let the commutative monoid structure on $X(G)$ have multiplication $\mu_{2\cdot *_G}$ and unit $\mu_{\emptyset_G}$, and we define $\mu_K^H \deq \mu_{[K/H]}$.
  Conversely, given $X$ satisfying \cite[Def~5.6]{Chan}, we let $\mu_S$ be defined as above.
  We have 2 tasks:
  \begin{enumerate}[label={(\roman*)}]
    \item verify that the above data yields a well-defined functor $G\colon \CAlg_I^{\mathrm{Chan-Hoyer}}(\cC) \rightarrow \CAlg_I(\cC)$; and
    \item verify that $G$ is fully faithful and essentially surjective.
  \end{enumerate}
  First, note that \cref{Associativity lemma,Restriction lemma,Indexing systems map of calgs} together imply that $G$ is a fully faithful functor with the above signature, so we're left with verifying that it is essentially surjective;
  that is, we have to check that $(\mu_K^H)$ satisfy \cite[Def~5.6]{Chan}, after which we may simply note that $\mu_S = \sum_{[H/K] \in \Orb(S)} \mu_{[H/K]}$ for essential surjectivity.
  In fact, condition (1) follows from \cref{Unitality remark} for the distinguished fixed point $*_H \subset N_K^H \Res_K^H *_H$, condition (2) follows from  \cref{I-calg associativity diagram} for the $[H/K]$-set $[K/L]$, and conditions (3) and (4) both follow from restriction-stability of $\mu$.
\end{proof}

\section{Equivariant Boardman-Vogt tensor products}\label{Operads section}
Using the language of fibrous patterns, in \cref{BV fibrous subsection} we define the \emph{Boardman Vogt tensor product}, and we show that it's closed and compatible with the Segal envelope in \cref{Alg is adjoint prop,BV Day pattern}.
Following this, in \cref{BV operad subsection} we specialize this to $\Op_{\cT}$;
moreover, we characterize the $\obv$-unit of $\Op_I$ and leverage this to compute the $\cT$-$\infty$-categories underlying operads of algebras.
Finally, in \cref{Inflation subsection}, we define the inflation adjunction $\Infl_e^{\cT} \cln \Op_{\cT} \rightleftarrows \Op\cln \Gamma^{\cT}$ and characterize its relationship with the Boardman-Vogt tensor product.

\subsection{Boardman-Vogt tensor products of fibrous patterns} \label{BV fibrous subsection}
\begin{definition}
  A \emph{magmatic pattern} is the data of a soundly extendable algebraic pattern $\base$ together with a functor $\wedge\colon \base \times \base \rightarrow \base$ which is compatible with Segal objects.
\end{definition}

\begin{construction}
  Let $(\base,\wedge)$ be a magmatic pattern.
    Then, the \emph{$\base$-Boardman-Vogt tensor product} is the bifunctor $- \obv- \cln \Fbrs(\base) \times \Fbrs(\base) \rightarrow \Fbrs(\base)$ defined by
    \[
      \fO \obv \fP \deq L_{\Fbrs}\prn{\fO \times \fP \rightarrow \base \times \base \xrightarrow{\wedge} \base}.\qedhere
    \]
\end{construction}
\def\fB{\base}
\def\se{\mathrm{se}}
\def\BiFun{\mathrm{BiFun}}
We defined this in order to have a mapping out property with respect to the following construction.
\begin{definition}
  Let $(\base,\wedge)$ be a magmatic pattern and $\fO,\fP,\fQ$ fibrous $\base$-patterns.
  Then, a \emph{bifunctor of fibrous $\base$ patterns} $\fO \times \fP \rightarrow \fQ$ is a commutative diagram in $\Cat$
  \[
    \begin{tikzcd}
      \fO \times \fP \arrow[r] \arrow[d]
      & \fQ \arrow[d]\\
      \base \times \base \arrow[r,"\wedge"]
      & \base
    \end{tikzcd}
  \]
  whose top horizontal arrow lies in $\AlgPatt$,\footnote{The lift to $\AlgPatt$ is unique, since each structure map in an algebraic pattern is a replete subcategory inclusion, hence a monomorphism in $\Cat$.} where $\fO \times \fP \rightarrow \base \times \base$ is induced by the structure maps of $\fO$ and $\fP$.
  The collection of bifunctors fits into a full subcategory
  \[
      \BiFun_{\base}(\fO,\fP;\fQ) \subset \Fun(\Delta^1 \times \Delta^1,\Cat).\qedhere
  \]
\end{definition}

\begin{example}\label{Trivial bifunctor example}
  Let $\fO,\fP$ be fibrous $\base$-patterns, and consider $\base$ to be a fibrous $\base$-pattern via the identity.
  Then, the $\infty$-category of bifunctors $\fO \times \fP \rightarrow \base$ is contractible, as it is equivalent to composite arrows $\fO \times \fP \rightarrow \base \times \base \rightarrow \base$.
\end{example}

\begin{observation}\label{Bifunctors observation}
    There are natural equivalences
    \begin{align*}
        \BiFun_{\base}(\fO,\fP;\fQ) 
        &\simeq \Fun^{\Int-\cocart}_{/\base \times \base}(\fO \times \fP, \wedge^*\fQ)\\ 
        &\simeq \Fun^{\Int-\cocart}_{/\base}(\wedge_!(\fO \times \fP), \fQ)\\
        &\simeq \Fun^{\Int-\cocart}_{/\base}(\fO \obv \fP, \fQ).\qedhere
    \end{align*}
\end{observation}
As in \cite[Prop~2.19]{Boardman} and the variety of recontextualizations of their ideas (e.g. \cite{Weiss, HA}), we recognize this as \emph{$\fO$-algebras in $\fP$-algebras}, making $\obv$ into a closed tensor product, using the following.
  \begin{construction}\label{BV construction}
    Fix $(\base,\wedge)$ a magmatic pattern, let $F\cln \fO \times \fP \rightarrow \fQ$ be a bifunctor of fibrous $\base$-patterns, and let $\fC$ be a fibrous $\fQ$-pattern.
    We have a diagram
    \[
        \fO \xleftarrow{p} \fO \times \fP \xrightarrow{F} \fQ;
    \]
    admitting push-pull adjunctions $p^* \dashv p_*$ and $L_{\Fbrs} F_! \dashv F^*$ on fibrous patterns, with compatible adjunctions on Segal objects by \cref{Fibrous patterns are strong Segal morphisms,Overcategory of fibrous patterns prop,Push-pull along projection corollary}.
    We define the pattern
    \[
        \uAlg^{\otimes}_{\fP/\fQ}(\fC) \deq p_*F^*\fC \in \Fbrs(\fO);
    \]
    this is the \emph{fibrous $\fO$-pattern of $\fP$-algebras in $\fC$ over $\fQ$.}
    In most cases, we will have $\fQ = \fO = \base$, in which case the information of a bifunctor $\base \times \fP \rightarrow \base$ is simply that of a fibrous $\base$-pattern $\fP$ by \cref{Trivial bifunctor example}.
    In this case, we simply write
    \[
        \uAlg^{\otimes}_{\fP}(\fC) \deq \uAlg^{\otimes}_{\fP/\base}(\fC) \in \Fbrs(\base);
    \]
    this is the \emph{fibrous $\base$-pattern of $\fP$-algebras in $\fC$.}
\end{construction}

    In the case $\fQ = \fO = \base$, the above diagram refines to
    \[
        \base \xleftarrow{p} \base \times \fP \xrightarrow{\id \times \pi} \base \times \base \xrightarrow{\wedge} \base,
    \]
    so the functor $\fP \mapsto \uAlg_{\fP}^{\otimes}(\fC)$ has a left adjoint computed by $L_{\Fbrs}\wedge_!\prn{\id \times \pi}_!p^*$;
    explicitly, this is computed on $\fP'$ by the fibrous localization of the diagonal composite
    \[
  \begin{tikzcd}[column sep = huge]
	{\fP' \times \fP} & {p^*  \fP'} \\
	& {\base \times \fP} \\
	& {\base \times \base} & \base
	\arrow["{\pi_{\fQ} \times \id}"{description}, from=1-1, to=2-2]
	\arrow["{\pi_{\fQ} \times \pi_{\fP}}"', curve={height=12pt}, from=1-1, to=3-2]
	\arrow["\simeq"{description}, draw=none, from=1-2, to=1-1]
	\arrow[from=1-2, to=2-2]
	\arrow["{\id \times \pi_{\fP}}"{description}, from=2-2, to=3-2]
	\arrow[from=2-2, to=3-3]
	\arrow["\wedge"', from=3-2, to=3-3]
\end{tikzcd}
\]
By definition, this is precisely $\fP' \otimes^{\BV} \fP$, so we've proved the following.
\begin{proposition}\label{Alg is adjoint prop}
    The functor $(-) \obv \fO\colon \Fbrs(\base) \rightarrow \Fbrs(\base)$ is left adjoint to $\uAlg_{\fO}^{\otimes}(-)$.
\end{proposition}

We additionally spell out a few useful characteristics $\obv$ and $\uAlg_{\fO}^{\otimes}(-)$ here.
First, we describe functoriality.
\begin{observation}
  Fix the fibrous $\base$-pattern $\fQ$.
  Suppose we have bifunctors of fibrous $\base$-patterns 
  \[
    F\colon \fO \times \fP \longrightarrow \fQ \longleftarrow \fO \times \fP'\colon G
  \]
  together with a morphism of fibrous $\base$-patterns $\varphi\colon \fP \rightarrow \fP'$ making the following diagram commute:
  \[
    \begin{tikzcd}[row sep = tiny]
      & \fO \times \fP \arrow[ld,"p"'] \arrow[dr,"F"] \arrow[dd, "\varphi" description]\\
      \fO && \fQ\\
      & \fO \times \fP' \arrow[ru, "G"'] \arrow[lu,"p'"]
    \end{tikzcd}
  \]
  The left triangle possesses a Beck-Chevalley transformation
  \[
    p^* \varphi_! \implies \id_! p^{\prime *} = p^{\prime *},
  \]
  which possesses a mate natural transformation $p'_* \implies p_* \varphi^*$;
  precomposing with $G^*$, this yields a ``pullback'' natural transformation
  \[
    \Alg_{\fP'/\fQ}^{\otimes}(-) \implies \Alg_{\fP/\fQ}^{\otimes}(-).\qedhere
  \]
\end{observation}
.We observe that, in all of the work above, we may have instead assumed that $\fC \in \Seg_\base(\Cat)$, in which case all of our constructions land in $\Seg_\base(\Cat)$.
Spelled out, this yields the following.
\begin{proposition}\label{Functoriality of alg prop}
    Fix $\fO,\fP,\fQ,\fC$ as in \cref{BV construction}. Then
    \begin{enumerate}
        \item if $\fC$ is a Segal $\fQ$-$\infty$-category, then $\uAlg^{\otimes}_{\fP/\fQ}(\fC)$ is a Segal $\fO$-$\infty$-category;
        \item if $\fC \rightarrow \fD$ is a morphism of Segal $\fQ$-$\infty$-categories, then the induced map $\uAlg^{\otimes}_{\fP/\fQ}(\fC) \rightarrow \uAlg^{\otimes}_{\fP/\fQ}(\fD)$ is a morphism of Segal $\fO$-$\infty$-categories, i.e. it preserves cocartesian arrows; and
        \item if $\fP \rightarrow \fP'$ is a morphism of fibrous $\base$-patterns and $\fC$ is a Segal $\fQ$-$\infty$-category, then the induced map of fibrous patterns
        \[
            \uAlg_{\fP'/\fQ}^{\otimes}(\fC) \rightarrow \uAlg_{\fP/\fQ}^{\otimes}(\fC)
        \]
        is a functor of Segal $\fO$-$\infty$-categories.
    \end{enumerate}
\end{proposition}

In analogy to \cite{Barkan_Segal} we show that this tensor product is compatible with Segal envelopes.
\begin{proposition}\label{BV Day pattern}
    The following diagram commutes
    \[
        \begin{tikzcd}
            \Fbrs(\base)^2 \arrow[rr,"\obv"] \arrow[d,"\Env"]
            && \Fbrs(\base) \arrow[d,"\Env"]\\
            \Fun(\base,\Cat)^2 \arrow[r,"\circledast"]
            & \Fun(\base,\Cat) \arrow[r,"L_{\Seg}"]
            & \Seg_{\base}(\Cat)
        \end{tikzcd}
    \]
\end{proposition}
\begin{proof}
    Fix $\fC$ a Segal $\base$-$\infty$-category.
    Then, there are natural equivalences
    \begin{align}
        \Fun_{\Seg_{\base}(\Cat)}
        \prn{\Env\prn{\fO \obv \fP},\fC}
        &\simeq \Fun^{\Int-\cocart}_{/\base}
            \prn{\wedge_! \fO \times \fP,\fC}
            \nonumber\\
        &\simeq \Fun^{\Int-\cocart}_{/\base \times \base}
            \prn{\fO \times \fP, \wedge^* \fC}
            \nonumber\\
        &\simeq \Fun^{\cocart}_{/\base \times \base}
            \prn{\Env_{\base \times \base}(\fO \times \fP), \wedge^* \fC}
            \nonumber\\
        &\simeq \Fun^{\cocart}_{/\base \times \base}
            \prn{\Env_\base(\fO) \times \Env_\base(\fP), \wedge^* \fC}
            \label{Product of segal equivalence}\\
        &\simeq \Fun^{\cocart}_{/\base}
        \prn{L_{\Seg} \wedge_!\prn{\Env_\base\prn{\fO) \times \Env_\base(\fP}},\fC}
            \nonumber\\
        &\simeq \Fun_{\Seg_{\base}(\Cat)}
        \prn{L_{\Seg} \prn{\Env_\base(\fO) \circledast \Env_\base(\fP)}, \fC}
            \label{Final equivalence for env coherence theorem}
     \end{align}
     Equivalence \cref{Product of segal equivalence} is \cref{Product of envelopes observation};
     \cref{Final equivalence for env coherence theorem} follows by symmetric monoidality of the Grothendieck construction \cite[Thm~B]{Ramzi}.
    The result then follows by Yoneda's lemma.
\end{proof}
We derive a uniqueness statement for $\obv$ by an analogous argument to \cite{Barkan_Segal}.
\begin{corollary}\label{Sliced BV Day pattern}
  $\obv$ is the unique bifunctor on $\Fbrs(\fB)$ making the following diagram commute:
  \[
    \begin{tikzcd}[ampersand replacement=\&, column sep=large]
      {\Fbrs(\base)^2} \&\&\& {\Fbrs(\base)} \\
      {\prn{\Fun(\base, \Cat)_{/\sA_{\base}}}^2} \& {\Fun(\base, \Cat)_{/\sA_{\base} \circledast \sA_{\base}}} \& {\Fun(\base, \Cat)_{/\sA_{\base}}} \& {\Seg(\base)_{/\sA_{\base}}}
      \arrow["\obv", from=1-1, to=1-4]
      \arrow["{\Env^2}"', from=1-1, to=2-1, hook]
      \arrow["\Env", hook', from=1-4, to=2-4]
      \arrow["\circledast"', from=2-1, to=2-2]
      \arrow["{\Env(\wedge)_!}"', from=2-2, to=2-3]
      \arrow["{L_{\Seg}}"', from=2-3, to=2-4]
    \end{tikzcd}
  \]
\end{corollary}
\begin{proof}
  Commutativity of the diagram follows by \cref{Functoriality of alg prop} and uniqueness of $\obv$ follows from the fact that the right vertical functor is fully faithful, hence a monomorphism in $\Cat$. 
\end{proof}

\subsection{Boardman-Vogt tensor products of \tV-operads}\label{BV operad subsection}
Recall that $\Op_{\cT} \simeq \Fbrs(\Span(\FF_{\cT}))$.
We specialize the results of \cref{BV fibrous subsection} to the case that $\cT$ has a terminal object.
\begin{construction}
  Fix an object $V \in \cT$.
  We show in \cref{Cartesian product is Segal} that the Cartesian product in $\FF_{V}$ endows $\Span(\FF_{V})$ with the structure of a magmatic pattern via the \emph{smash product} 
  \[
    \wedge \deq \Span(\times)\colon \Span(\FF_{V}) \times \Span(\FF_{V}) \rightarrow \Span(\FF_{V});
  \]
  we refer to the resulting bifunctor as the \emph{Boardman-Vogt tensor product of $V$-operads} 
  \[
    \cO^{\otimes} \obv \cP^{\otimes} \deq L_{\Op_{\cT}}\prn{\cO^{\otimes} \times \cP^{\otimes} \rightarrow \Span(\FF_{V}) \times \Span(\FF_{V}) \xrightarrow{\wedge} \Span(\FF_{V})}.
  \]
  The \emph{$V$-operad of $\cO$-algebras in $\cC^{\otimes}$} is given by the right adjoint $\uAlg_{\cO}^{\otimes}(\cC) \in \Op_{\cT}$ to the Boardman-Vogt tensor product constructed in \cref{Alg is adjoint prop}.
\end{construction}

\cref{Functoriality of alg prop} immediately implies the following. 
\begin{corollary}\label{Symmetric monoidal push-pull corollary}
  Fix $\cO^{\otimes} \rightarrow \cP^{\otimes}$ a map of $V$-operads and $\cC^{\otimes} \rightarrow \cD^{\otimes}$ a map of $V$-symmetric monoidal $\infty$-categories.
  Then, $\uAlg_{\cO}^{\otimes}(\cC)$ is a $V$-symmetric monoidal category, and the canonical lax $V$-symmetric monoidal functors
  \[
    \uAlg_{\cP}^{\otimes}(\cC) \rightarrow \uAlg_{\cO}^{\otimes}(\cC),\hspace{40pt} \uAlg_{\cO}^{\otimes}(\cC) \rightarrow \uAlg_{\cO}^{\otimes}(\cD)
  \]
  are $V$-symmetric monoidal.
\end{corollary}

Using this, in the one-color case we view $\cO^{\otimes} \obv \cP^{\otimes}$-algebras as \emph{homotopy-coherently interchanging $\cO$-algebra and $\cP$-algebra structures on a common $\cT$-object}.
This takes an easy to understand form in the 1-categorical case by the following argument.

\begin{observation}
  Suppose $\cC^{\otimes}$ is an $I$-symmetric monoidal 1-category and $\cO^{\otimes},\cP^{\otimes}$ are one-color $\cT$-operads.
  Then, an $\cO^{\otimes} \obv \cP^{\otimes}$-algebra structure on a $\cT$-object $X \in \Gamma^{\cT} \cC$ is equivalently viewed as a pair of maps $\cP^{\otimes} \rightarrow \End_X^{\otimes}(\cC)$ and $\cO^{\otimes} \rightarrow \End_X^{\otimes}\prn{\uAlg^{\otimes}_{\cP}(\cC)}$ via \cref{1-algebra proposition}.
  In particular, this consists of pairs of $\cO$-algebra and $\cP$-algebra structures on $X$ subject to the \emph{interchange relation} that, for all $\mu_S \in \cO(S)$ and $\mu_T \in \cO(T)$, the following diagram commutes.
\[\begin{tikzcd}
{\bigotimes\limits_U^S X^{\otimes \Res_U^V T}_V} & {X_{V}^{\otimes S \times T}} & {\bigotimes\limits_W^T X^{\otimes \Res_W^V S}_V} &&& {X^{\otimes T}_V} \\
{X^{\otimes S}_V} &&&&& {X_V}
\arrow["\simeq"{description}, draw=none, from=1-1, to=1-2]
\arrow["\prn{\Res_U^V \mu_T}"', from=1-1, to=2-1]
\arrow["\simeq"{description}, draw=none, from=1-2, to=1-3]
\arrow["\prn{\Res_W^V \mu_S}"{description}, from=1-3, to=1-6]
\arrow["\mu_{T}", from=1-6, to=2-6]
\arrow["\mu_{S}"{description}, from=2-1, to=2-6]
\end{tikzcd}\]
  A morphism of $\cO^{\otimes} \obv \cP^{\otimes}$-algebras is equivalently expressed as a morphism of underlying $\cT$ objects $X \rightarrow Y$ causing the following to commute:
  \[
    \begin{tikzcd}[ampersand replacement=\&, row sep=tiny]
      \& {\End^{\otimes}_X(\cC)} \&\& {\End_X^{\otimes} \uAlg_{\cO}^{\otimes}(\cC)} \\
      {\cO^{\otimes}} \&\& {\cP^{\otimes}} \\
      \& {\End^{\otimes}_Y \cC} \&\& {\End_Y^{\otimes} \uAlg_{\cO}^{\otimes}(\cC)}
      \arrow[from=1-2, to=3-2]
      \arrow[from=1-4, to=3-4]
      \arrow[from=2-1, to=1-2]
      \arrow[from=2-1, to=3-2]
      \arrow[from=2-3, to=1-4]
      \arrow[from=2-3, to=3-4]
    \end{tikzcd}
  \]
  By faithfulness of the forgetful functor $\Alg_{\cO}(\cC)_V \rightarrow \cC_V$, this is simply a morphism of underlying $\cT$-objects which is separately an $\cO$-algebra and $\cP$-algebra map.
\end{observation}

\cref{BV Day pattern} specializes to the following.
\begin{corollary}\label{BV Env corollary}
  The $V$-symmetric monoidal envelope intertwines with the mode structure:
  \[
    \Env\prn{\cO^{\otimes} \obv \cP^{\otimes}} \simeq \Env\prn{\cO^{\otimes}} \otimes^{\mathrm{Mode}} \Env\prn{\cP^{\otimes}}.
  \]
\end{corollary}
In particular, \cite[Thm~E]{Barkan_Segal} shows that this property identifies the non-equivariant Boardman-Vogt tensor product, so we acquire the following.
\begin{corollary}\label{Underlying tensor product}
  When $\cT \simeq *$, $\obv$ is naturally equivalent to the Boardman-Vogt tensor product of \cite{Barkan_Segal,HA,Hinich}.
\end{corollary}
Additionally, we may characterize the $\obv$-unit.
\begin{proposition}\label{Triv algberas prop}\label{Triv algebras prop}
    For all $\cO^{\otimes} \in \Op_{V}$, we have $\cO^{\otimes} \simeq \cO^{\otimes} \obv \triv_{V}^{\otimes}$;
    hence there exists a natural equivalence
    \[
        \uAlg_{\triv_{V}}^{\otimes}(\cO) \rightarrow \cO^{\otimes}.
    \]
\end{proposition}
\begin{proof}
  The first statement implies the second by the usual folklore argument:
\begin{align*}
    \Map(\cO^{\otimes}, \uAlg_{\triv_{V}}^{\otimes}(\cP))
    &\simeq \Map\prn{\cO^{\otimes} \obv \triv_{V}^{\otimes}, \cP^{\otimes}},\\ 
    &\simeq  \Map(\cO^{\otimes}, \cP^{\otimes}),
\end{align*}
  so Yoneda's lemma yields a natural equivalence $\uAlg_{\triv_{V}}^{\otimes}(\cP) \simeq \cP^{\otimes}$.
  The same argument in reverse shows that the second statement implies the first.

  By the expression $\triv_V^{\otimes} \simeq L_{\Op_{\cT}}(*_{\cT} \rightarrow \Span(\FF_{\cT}))$, bifunctors $\triv^{\otimes}_{V} \times \cO \rightarrow \cP$ correspond canonically with functors of $\cT$-operads $\cO \rightarrow \cP$;
  put another way, using the bifunctor presentation for algebras of \cref{Bifunctors observation}, this demonstrates that the forgetful natural transformation
\[
    \Alg_{\cO \otimes^{\BV} \triv_V}(\cP) \rightarrow \Alg_{\cO}(\cP)
\]
is a natural equivalence for all $\cP^{\otimes} \in \Op_{V}$;
Yoneda's lemma then demonstrates that $\cO^{\otimes} \obv \triv^{\otimes}_{V} \simeq \cO^{\otimes}$. 
\end{proof}

Using this, we have a sequence of natural equivalences
\begin{align*}
    U \uAlg_{\cO}^{\otimes}(\cP)
    &\simeq \uAlg_{\triv_{V}} \uAlg_{\cO}^{\otimes}(\cP)\\
    &\simeq \uAlg_{\cO \otimes \triv_{V}}(\cP)\\
    &\simeq \uAlg_{\cO}(\cP);
\end{align*}
in particular, we've proved the following corollary.
\begin{corollary}\label{Alg underlying corollary}
    There exists a natural equivalence
    \[
        U\uAlg^{\otimes}_{\cO}(\cP) \simeq \uAlg_{\cO}(\cP).
    \]
\end{corollary}
We've shown in \cref{BV Day pattern} that $\Env$ intertwines $\obv$ with $\circledast$, and we've now seen that $\triv^{\otimes}_{V}$ is the $\obv$-unit.
In fact, $\Env$ intertwines units.
\begin{proposition}\label{Env unit lemma}
  $\Env_I(\triv_{\cT})$ is the $\circledast$-unit in $\CMon_I(\Cat)^{\circledast}$.
\end{proposition}
\begin{proof}
  By \cref{Mode SMC}, when $\cC^{\times}$ is cartesian, the free object $\Fr_I(*) \in \CMon_I(\cC)$ is the $\circledast$-unit, so
  \begin{align*}
    \Fun^{\otimes}_I(\Env_I(\triv_{\cT})^{\otimes}, \cD^{\otimes})
    &\simeq \Alg_{\triv_{\cT}}(\cD^{\otimes}) &\ref{Env is maps of operads corollary}\\
    &\simeq \Gamma^{\cT} \cD &\ref{Triv prop}\\
    &\simeq \Fun^{\otimes}_I(\Fr_I *, \cD^{\otimes})\\
    &\simeq \Fun^{\otimes}_I\prn{1^{\circledast},\cD^{\otimes}};
  \end{align*}
  hence the result follows from Yoneda's lemma.
\end{proof}

In forthcoming work \cite{Tensor}, we will use \cref{BV Env corollary,Env unit lemma} and a variant of Barkan-Steinebrunner's strategy \cite{Barkan} to lift $\obv$ to a canonical symmetric monoidal structure.

\subsection{Inflation and the Boardman-Vogt tensor product}\label{Inflation subsection}
Recall the adjunction $\Infl_e^{\cT}\colon \Cat \rightleftarrows \Cat_{\cT}\colon \Gamma^{\cT}$ of \cref{T-categories subsection}.
We briefly discuss an operadic version of this and relate it to $\obv$.
\begin{construction}
  Given $\cO^{\otimes}$ a $\cT$-operad, and $V \in \cT$, we form the \emph{$V$-value operad} 
  \[
    \Gamma^V \cO^{\otimes} \deq i_V^{*} \cO^{\otimes},
  \]
  where $i_V\cln \Span(\FF) \hookrightarrow \Span(\FF_{\cT})$ is the map of patterns extending the coproduct preserving functor $\FF \hookrightarrow \FF_{\cT}$ sending $* \mapsto *_V$.
  Using this, we may set
  \[
  \Gamma^{\cT} \cO^{\otimes} \deq \lim_{V \in \cT} \cO^{\otimes},
  \]
  noting that this recovers $\Gamma^V$ if $V$ is terminal in $\cT$.
\end{construction}

\begin{remark}
  In the case that $\cC^{\otimes}$ is a $\cT$-symmetric monoidal $\infty$-category, the structure map of the operad $\Gamma^V \cC$ is the pullback of a cocartesian fibration, so it is a cocartesian fibration, i.e. it presents a symmetric monoidal $\infty$-category;
  unwinding definitions, this agrees with the construction $\Gamma^V \cC$ of \cref{Symmetric monoidal evaluation construction}.
  Since the forgetful functor $\Cat \rightarrow \Op$ is a right adjoint, it preserves limits, so the two constructions of $\Gamma^{\cT} \cC$ also agree. 
\end{remark}  

In \cref{I-infinity operad prop}, we show that $\varphi\colon \cT^{\op} \times \Span(\FF) \rightarrow \Span_{I^\infty}(\FF_{\cT})$ induces an equivalence
\[
\Op_{I^\infty} \simeq \Fun(\cT^{\op},\Op).
\]
In particular, this yields the following.
\begin{proposition}
  The functor $\Gamma^{\cT}\colon \Op_{I^\infty} \rightarrow \Op$ has a fully faithful left adjoint $\Infl^{\cT}\colon \Op \rightarrow \Op_{I^\infty}$ whose image is spanned by the $I^{\infty}$-operads whose corresponding functors $\cT^{\op} \rightarrow \Op$ are constant.
\end{proposition}

The map of patterns $i_V$ induces a push-pull adjunction $E_{I^\infty}^{\cT} \cln \Op_{I^\infty} \rightleftarrows \Op_{\cT} \cln \Bor_{I^\infty}^{\cT}$, and we will write $\Infl^{\cT} \cln \Op \rightleftarrows \Op_{\cT} \cln \Gamma^{\cT}$ for the composite adjunction as well.
\begin{proposition}\label{Infl BV}
  There exists a natural equivalence $\Infl_e^{V} \cO^{\otimes} \obv \Infl_e^{V} \cP^{\otimes} \simeq \Infl_e^{V} \prn{\cO^{\otimes} \obv \cP^{\otimes}}$.  
\end{proposition}
\begin{proof}
  We can verify that $\Infl_e^{\cT}$ is product-preserving, so we acquire a zigzag of maps
  \begin{align*}
    \Infl_e^{V}\cO^{\otimes} \obv \Infl_e^{V} \cP^{\otimes} 
    \xleftarrow{\;\;\;\; \eta_{\Op_{V}} \;\;\;\;}&
      \;\; \wedge_! \prn{\Infl_e^{V}\cO^{\otimes} \times \Infl_e^{V} \cP^{\otimes}}\\
    \simeq& 
      \;\; \wedge_! \Infl_e^{V} \prn{\cO^{\otimes} \times \cP^{\otimes}}\\ 
    \simeq& 
      \;\; \Infl_e^{V}\wedge_! \prn{\cO^{\otimes} \times \cP^{\otimes}}\\
    \xrightarrow{\;\;\;\; \Infl_e^{V} \eta_{\Op} \;\;\;\;} & 
      \;\; \Infl_e^{V}\prn{\cO^{\otimes} \obv \cP^{\otimes}},
  \end{align*}
  with $\eta_{\Op_{V}}$ an $L_{\Op_{V}}$-equivalence.
  We're tasked with proving that $\eta_{\Op}$ is an $L_{\Op_{V}}$-equivalence;
  then, the desired equivalence can be gotten by applying $L_{\Op_{V}}$ and inverting arrows as needed.
  In fact, if $\cQ^{\otimes}$ is a $V$-operad, then pullback along $\eta_{\Op}$ furnishes an equivalence
  \begin{align*} 
    \Fun_{/\Span(\FF_V)}^{\Int-\cocart} \prn{\Infl_e^{V} \wedge_! \prn{\cO^{\otimes} \times \cP^{\otimes}}, \cQ^{\otimes}}
    &\simeq \Fun_{/\Span(\FF)}^{\Int-\cocart} \prn{\wedge_! \prn{\cO^{\otimes} \times \cP^{\otimes}}, \Gamma^{V} \cQ^{\otimes}}\\
    &\simeq \Fun_{/\Span(\FF)}^{\Int-\cocart} \prn{\cO^{\otimes} \obv \cP^{\otimes}, \Gamma^{V} \cQ^{\otimes}}\\
    &\simeq \Fun_{/\Span(\FF_{V})}^{\Int-\cocart} \prn{\Infl_e^{V} \prn{\cO^{\otimes} \obv \cP^{\otimes}}, \cQ^{\otimes}},
  \end{align*}
  so $\Infl_e^{V} \eta_{\Op}$ is an $L_{\Op_{V}}$-equivalence, yielding the desired natural equivalence.
\end{proof}

\begin{example}
  Let $G$ be a finite group and $n_G$ the trivial $n$-dimensional orthogonal $G$-representation.
  Note that the bottom map 
  \[
    \begin{tikzcd}
      {\EE_{n_G}(m \cdot *_H)} & {\EE_{n_G}(m \cdot *_K)} \\
      {\Conf_{m \cdot *_H}^{H}(n_G)} & {\Conf_{m \cdot *_H}^{H}(n_G)}
      \arrow[from=1-1, to=1-2]
      \arrow["\simeq"{marking, allow upside down}, draw=none, from=1-1, to=2-1]
      \arrow["\simeq"{marking, allow upside down}, draw=none, from=1-2, to=2-2]
      \arrow["\sim"', from=2-1, to=2-2]
    \end{tikzcd}
  \]
  is an equivalence for all $K \subset H \subset G$, as it intertwines the tautological identification of each side with $\Conf_m(\RR^n)$.
  In particular, the map $\EE^{\otimes}_{n_G} \rightarrow \EE^{\otimes}_{\infty_G} \simeq \EE^{\otimes}_\infty$ witnesses $\EE_{n_G}$ as an $I^\infty$-operad in the image of $\Infl_e^G$;
  unwinding definitions, we have an equivalence $\Infl_e^G \EE^{\otimes}_n \simeq \EE_{n_G}$.
\end{example}

In general, we define the $\cT$-operad $\EE^{\otimes}_n \deq \Infl_e^{\cT} \EE_n^{\otimes}$.

\begin{corollary}[Trivially eqivariant Dunn additivity]\label{En corollary}
  There is an equivalence $\EE^{\otimes}_n \obv \EE^{\otimes}_m \simeq \EE^{\otimes}_{n+m}$.
\end{corollary}
\begin{proof}
  By \cref{Underlying tensor product,Infl BV}, it suffices to construct an equivalence of operads $\EE^{\otimes}_n \obv \EE^{\otimes}_m \simeq \EE^{\otimes}_{n+m}$;
  this is nonequivariant Dunn additivity \cite[Thm~5.1.2.2]{HA}.
\end{proof}
Moreover, we acquire compatibility between $\Gamma^{\cT}$ and $\cT$-operads of algebras.
\begin{corollary}\label{Gamma alg}
  There exists a natural equivalence of operads
  \[
    \Gamma^{V} \uAlg^{\otimes}_{\Infl_e^{V} \cO}(\cC) \simeq \Alg_{\cO}^{\otimes}(\Gamma^{V} \cC)
  \]  
\end{corollary}
\begin{proof}
  Once more, given $\cP^{\otimes} \in \Op$, there is a string of natural equivalences
  \begin{align*}
    \Alg_{\cP} \Gamma^{V} \uAlg_{\Infl_e^{V} \cO}^{\otimes}(\cC)
    &\simeq \Alg_{\Infl_e^{V} \cP} \uAlg_{\Infl_e^{V} \cO}^{\otimes}(\cC)\\
    &\simeq \Alg_{\Infl_e^{V} \cP \otimes \Infl_e^{V} \cO}(\cC)\\
    &\simeq \Alg_{\Infl_e^{V} \prn{\cP \otimes \cO}}(\cC)\\
    &\simeq \Alg_{\cP \otimes \cO}(\Gamma^{V} \cC)\\
    &\simeq \Alg_{\cP} \Alg_{\cO}^{\otimes}(\Gamma^{V} \cC),
  \end{align*}
  so the result follows by Yoneda's lemma.
\end{proof}

A similar statement to \cref{Infl BV} for $\triv(-)^{\otimes}$ follows by either symbol pushing or examining the various localizations;
we take the former approach, constructing a string of natural equivalences
\begin{align*}
  \Alg_{\Infl_e^{V} \triv\prn{\cC}}(\cO)
  &\simeq \Alg_{\triv\prn{\cC}}(\Gamma^{V} \cO)\\
  &\simeq \Fun(\cC,\Gamma^{V} \cO)\\
  &\simeq \Fun_{\cT}(\Infl_e^{V} \cC, \cO)\\
  &\simeq \Alg_{\triv\prn{\Infl_e^{V} \cC}}(\cO).
\end{align*}
That is, we've proved the following.
\begin{proposition}\label{Infl triv}
  There is a canonical natural equivalence
  \[
    \Infl_e^{V} \triv\prn{\cC}^{\otimes} \simeq \triv\prn{\Infl_e^{V} \cC}^{\otimes}.
  \]
\end{proposition}

\begin{remark}
  \cref{BV operad subsection,Inflation subsection} collected results about Boardman-Vogt tensor products of $V$-operads, which implies the corresponding results for $G$-operads as $\cO_G$ has a terminal object.
  Nevertheless, for the sake of equivariance under families, we would like to prove the corresponding results for $\cT$-operads.
  Unwinding the arguments, it would suffice to lift $\prn{\Op_V, \obv}$ to an $\AA_2$-$\cT$-$\infty$-category, and thus develop a \emph{Boardman-Vogt tensor product of $\cT$-operads} which restricts to our construction.
  In fact, to do so simply requires constructing coherent natural equivalences 
  \[
    \Res_U^V \prn{\cO^{\otimes} \obv \cP^{\otimes}} \simeq \Res_U^V \cO^{\otimes} \obv \Res_U^V \cP^{\otimes}
  \]
  for all $U\rightarrow V \in \cT$.
  Inspired by the uniqueness of \cref{Sliced BV Day pattern}, two strategies come to mind:
  \begin{enumerate}
    \item Much of the work of \cite{Barkan_equifibered} is likely to hold for $\cT$-commutative monoids;
      in particular, one may expect that an equifibered map between envelopes of $\cT$-operads canonically lifts to a map over $\uFF_{\cT}^{\cT-\sqcup}$, which would imply that the \emph{unsliced} envelope $\Op_{\cT} \rightarrow \Cat_{\cT}^{\otimes}$ is a replete subcategory inclusion, and hence monic.
      Thus \cref{BV Env corollary} and restriction-stability of $\otimes^{\Mode}$ would yield restriction-stability of $\obv$.
    \item Alternatively, one may note that, in the nonequivariant case, $\Comm^{\otimes} \in \Op$ is an idempotent algebra.
      If $\Comm_{V}^{\otimes} \in \Op_V$ is an idempotent algebra for all $V$, then their envelopes $\uFF_{V}^{V-\sqcup}$ will be idempotent algebras under the mode structure by \cref{Sliced BV Day pattern}, compatibly with restriction (as the unit maps each live in a contractible mapping space).
      This would yield a symmetric monoidal structure on $\uCat_{\cT, /\uFF_{\cT}^{\cT-\sqcup}}^{\otimes}$ under which $\uOp_{\cT}$ would be a symmetric monoidal full $\cT$-subcategory.
  \end{enumerate}
  The author hopes to return fulfill the second strategy in forthcoming work \cite{Tensor}.
\end{remark}

\begin{appendix}
\section{Burnside algebraic patterns: the atomic orbital and global cases}
\stoptocwriting % stop toc from writing 
The following appendix is not written to be particularly original;
most of its contents appear as straightforward technical extensions of beloved works in higher algebra, and they are included for the sake of mathematical completeness.
The contents herein do not depend on the results of the main body of this paper.

\subsection{\texorpdfstring{$I$}{I}-operads as fibrous patterns}\label{Operads are fibrous subsection}
This subsection deviates only slightly from \cite[\S~5.2]{Barkan}, so we suggest that the reader first read their work.
We're interested in proving a global equivariant generalization of \cref{Operads are fibrous theorem}, so we begin with the relevant patterns.
We assume familiarity with the terminology of finite pointed $\cT$-sets and $\cP$-sets of \cite{Cnossen_stable,Nardin}.

As noted in \cite{Windex}, $\FF_{\cT}$ is an extensive category in the sense of \cite[Def~2.2.1]{Cnossen_tambara}, an \emph{extensive span pair} $(\FF_{\cT},I_{\cP})$ equivalent to an atomic orbital subcategory $\cP \subset \cT$ (i.e. an indexing category), and a \emph{weakly extensive span pair} $(\FF_{\cT},I)$ equivalent to a one-color weak indexing category $I \subset \FF_{\cT}$.
In the case $(\FF_{\cT},I)$ is a weakly extensive span pair, we write
\[
  \Span_I(\FF_{\cT}) \deq \Span_{all,I}(\FF_{\cT};\cT^{\op})
\]
for the resulting pattern.
Moreover, given $\cP \subset \cT$ an atomic orbital $\infty$-category, we write $\Span_{\cP}(\FF_{\cT}) \deq \Span_{\FF_{\cT}^{\cP}}(\FF_{\cT})$ and
\[
  \Tot \uFF_{\cT,*}^{\cP} \deq \Span_{s.i.,tdeg}(\Tot \uFF_{\cT}^{\cP,\vee},\cT^{\op}),
\]
where $(-)^\vee\colon \Cat_{/\cC}^{\cocart} \rightarrow \Cat_{/\cC}^{\cart}$ is the \emph{dual cartesian fibration} construction, $\Tot \uFF_{\cT}^{\cP, \vee, s.i.} \subset \Tot \uFF_{\cT}^{\cP,\vee}$ is the wide subcategory of morphisms $f\colon (S \rightarrow U) \rightarrow (T \rightarrow V)$ whose associated morphism $f_\circ$ is a summand inclusion:
\[
 \begin{tikzcd}[ampersand replacement=\&, row sep=small]
	S \&\& T \\
	\& {T  \times_V U} \\
	U \&\& V
	\arrow["{f_s}", from=1-1, to=1-3]
	\arrow["{f_{\circ}}"{description}, dashed, from=1-1, to=2-2]
	\arrow[from=1-1, to=3-1]
	\arrow[from=1-3, to=3-3]
	\arrow[from=2-2, to=1-3]
	\arrow["{`}", from=2-2, to=3-1]
	\arrow["{f_t}", from=3-1, to=3-3]
\end{tikzcd} 
\]
and $\Tot \uFF_{\cT}^{\cP,\vee,tdeg} \subset \Tot \uFF_{\cT}^{\cP,\vee}$ the wide subcategory with $f_t$ homotopic to the identity.

The upshot of this is that we acquire a map of adequate quadruples
\[
  \prn{\Tot \uFF_{\cT}^{\cP, \vee}, (s.i.,tdeg), \cT^{\op}}
  \rightarrow \prn{\FF_{\cT}, (all,\FF_{\cT}^{\cP}), \cT^{\op}}
\]
lying over the source map 
\[
  s\colon \Tot \uFF_{\cT}^{\cP,{\vee}} \rightarrow \FF_{\cT},
\]
yielding a map of algebraic patterns
\[
  \varphi\colon  \Tot \uFF_{\cT,*}^{\cP} \rightarrow \Span_{\cP}(\FF_{\cT}).
\]
We will prove the following theorem.
\begin{theorem}\label{Span fibrous theorem}
  The map of patterns $\varphi\colon \Tot \uFF^{\cP}_{\cT,*} \rightarrow \Span(\FF_{\cT})$ induces equivalences of categories 
  \begin{align*}
    \Seg_{\Span_{\cP}(\FF_{\cT})}(\cC) 
    &\simeq \Seg_{\Tot \uFF^{\cP}_{\cT,*}}(\cC),\\
    &\simeq \CMon_{\cP}(\cC);\\
    \Fbrs(\Span_{\cP}(\FF_{\cT})) 
    &\simeq \Fbrs(\Tot \uFF^{\cP}_{\cT,*}).
  \end{align*}
  Moreover, in the case $\cT = \cP$, there is an additional equivalence
  \[
    \Fbrs(\Tot \uFF_{\cT,*}) \simeq \Op_{\cT,\infty},
  \]
  the latter denoting Nardin-Shah \cite{Nardin}'s $\infty$-category of $\cT$-$\infty$-operads.
\end{theorem}

\subsubsection{The pattern $\Tot \uFF_{\cT,*}^{\cP}$}%
We may explicitly describe the Segal conditions for $\Tot \uFF_{\cT,*}^{\cP}$.
\begin{lemma}[{\cite[Obs~5.2.9]{Barkan}}]\label{Segal condition lemma}
  Fix $\brk{S \rightarrow U}$ an object in $\uFF_{\cT,*}^{\cP}$.
  Then, there are equivalences 
  \begin{align}
    \label{Segal condition eq1}\prn{\prn{\Tot \uFF_{\cT,*}^{\cP}}^{\el}_{\brk{S \rightarrow U}/}}^{\op}
    &\simeq \cT \times_{\Tot \uFF_{\cT}^{\cP}} \Tot \uFF^{\cP,\vee, s.i.}_{\cT,/\brk{S \rightarrow U}}\\
    \label{Segal condition eq1.5} &\simeq \cT \times_{\FF_{\cT}^{\cP}} \uFF^{\cP, \vee}_{\cT,/\brk{S \rightarrow  U}}\\
    \label{Segal condition eq2}   &\simeq \cT \times_{\FF_{\cT}^{\cP}} \FF^{\cP}_{\cT,/S}.
  \end{align}
  Furthermore, the full subcategory of $\prn{\cT \times_{\cT} \FF^{\cP}_{\cT,/S}}^{\op}$ consisting of morphisms $f\colon T \rightarrow S$ such that $f$ itself is the inclusion of an orbit is an initial subcategory equivalent to the set $\mathrm{Orb}(S)$.
\end{lemma}
\begin{proof}
  \cref{Segal condition eq1,Segal condition eq2} follows by definition.
  For \cref{Segal condition eq1.5}, this follows by noting that whenever $\brk{U = U} \rightarrow \brk{S \rightarrow V}$ is a morphism in $\uFF^{\cP}_{\cT}$ out of an orbit, the associated morphism $U \rightarrow S \times_V U$ is a summand inclusion, as it's split by the projection $S \times_V U \rightarrow U$ and $\cP$ is atomic.
  For the remaining statement, the inclusion $\mathrm{Orb}(S) \hookrightarrow \cT \times_{\FF_{\cT}^{\cP}} \FF^{\cP}_{\cT,/S}$ has a right adjoint sending $f:T \rightarrow S$ to $f(T) \subset S$, so it is initial.
\end{proof}

Moreover, the pattern is reasonably well behaved.
\begin{lemma}[{\cite[Cor~5.2.10]{Barkan}}]\label{Sound lemma}
  The pattern $\Tot \uFF^{\cP}_{\cT,*}$ is sound.
\end{lemma}
\begin{proof}
  We verify the conditions of \cite[Prop~3.3.23]{Barkan}.
  First, we must verify that $\prn{\Tot \uFF_{\cT}^{\cP, \vee, si}}_{/S} \hookrightarrow \Tot \uFF^{\cP, \vee}_{\cT,/S}$ is fully faithful, 
  i.e. if there is a pair of $\Tot \uFF_{\cT}^{\cP,\vee}$-morphisms
  \[
    \begin{tikzcd}
      S_2 \arrow[r, "f_s"] \arrow[d]
      & S_1 \arrow[r, "g_s"] \arrow[d]
      & S_0 \arrow[d]\\
      U_2 \arrow[r, "f_t"]
      & U_1 \arrow[r, "g_t"]
      & U_0
    \end{tikzcd}
  \]
  such that the associated maps $gf_{\circ}\colon S_2 \rightarrow S_0 \times_{U_0} U_2$ and $g_{\circ}\colon S_1 \rightarrow S_0 \times_{U_0} U_1$ are summand inclusions, the map $S_2 \rightarrow S_1 \times_{U_1} U_2$ is a summand inclusion.
  Here, we use the orbitality property that a morphism in $\FF_{\cT}^{\cP}$ is a summand inclusion if and only if it's a section;
  noting that $gf_{\circ}$ may be decomposed as 
  \[
    S_2 \xrightarrow{\;\; f_{\circ}\;\;} S_1 \times_{U_1} U_2 \xrightarrow{\;\; g_{\circ} \times_{U_1} U_2 \;\;} S_0 \times_{U_0} U_1 \times_{U_1} U_2 \simeq S_0 \times_{U_0} U_2. 
  \]
  if $r$ is a retract for $gf_{\circ}$, then $r \circ \prn{g_\circ \times_{U_1} U_2}$ is a retract for $f_\circ$, so $f$ lies in $\prn{\Tot \uFF_{\cT}^{\cP, \vee, s.i.}}$, as desired.

  Last, we must verify that
  \[
    \Tot \uFF_{\cT, /\brk{S \rightarrow U}}^{\cP, \vee, si, \el} \hookrightarrow
    \Tot \uFF_{\cT, /\brk{S \rightarrow U}}^{\cP, \vee, \el}
  \]
  is final for all $\brk{S \rightarrow U} \in \uFF^{\cP, \vee}_{\cT}$;
  in fact, it is an equivalence by \cref{Segal condition lemma}.
\end{proof}

From this, we may prove the following proposition.
\begin{proposition}\label{Pointed prop}
  There is an equivalence of $\infty$-categories over $\cC$
   \[
     \Seg_{\Tot \underline{\FF}^{\cP}_{\cT,*}}(\cC) \simeq \Fun_{\cT}^{\cP-\oplus}(\uFF_{\cT,*}^{\cP}, \uCoFr^{\cT}(\cC)).
   \]
   Moreover, when $\cP = \cT$, there is an equivalence of $\infty$-categories
    \[
      \Fbrs(\Tot \underline{\FF}_{\cT,*}) \simeq \Op_{\cT, \infty},
    \]
    the latter denoting Nardin-Shah \cite{Nardin}'s $\infty$-category of $\cT$-$\infty$-categories.
\end{proposition} 
\begin{proof}[Proof of \cref{Pointed prop}]
  For the first statement, note by \cref{Segal condition lemma} that a Segal $\Tot \uFF^{\cP}_{\cT,*}$-object in $\cC$ is equivalent to a functor 
  \[
    M\colon \Tot \uFF^{\cP}_{\cT,*} \rightarrow \cC
  \]
  satisfying $M(\bigoplus_i {U_i}) \simeq \prod_i M(U_i)$;
  taking adjunct maps yields a fully faithful embedding
  \[
    \Seg_{\Tot \uFF_{\cT,*}^{\cP}}(\cC) \hookrightarrow \Fun_{\cT}(\uFF_{\cT,*}^{\cP}, \uCoFr^{\cT}(\cC)),
  \]
  so it suffices to identify which $\cT$-functors $\uFF_{\cT,*}^{\cP} \rightarrow \uCoFr^{\cT}(\cC)$ satisfy the above condition. 
  In fact, this follows from the identification of $\cT$-(co)limits in $\uCoFr^{\cT}(\cC)$ of \cref{Fixed points of colimit}.
  The second statement follows by comparing definitions with \cite[Prop~4.1.7]{Barkan} in view of \cref{Segal condition lemma}.
\end{proof}

We now turn to the remaining statements of \cref{Operads are fibrous theorem} making use of the following theorem, whose main content is due to Shaul Barkan in \cite[Cor~2.64]{Barkan-arity}.
\begin{theorem}[{\cite[Prop~3.1.16,Thm~5.1.1]{Barkan}}]\label{Equivalence theorem}
  Suppose $f\colon \cO \rightarrow \cP$ is a strong Segal morphism of algebraic patterns such that the following conditions hold:
  \begin{enumerate}
    \item $f^{\el}\colon \cO^{\el} \rightarrow \cP^{\el}$ is an equivalence, and
    \item for every $O \in \cO$, the functor $\prn{\cO^{\act}_{/O}}^{\simeq} \rightarrow \prn{\cP^{\act}_{/f(O)}}^{\simeq}$ is an equivalence.
  \end{enumerate}
  Then, the functor $f^*\colon \Seg_{\cP}(\cC) \rightarrow \Seg_{\cO}(\cC)$ is an equivalence.
  Furthermore, if $\cP$ is soundly extendable, then $f^*\colon \Fbrs(\cP) \rightarrow \Fbrs(\cO)$ is an equivalence, and it suffices to check condition (2) on $O \in \cO^{\el}$.
\end{theorem}

\subsubsection{Global effective burnside patterns}
Fix $I \subset \FF_{\cT}$ a weakly extensive subcategory.
There is a span pattern analog to \cref{Segal condition lemma} which is proved identically.
\begin{lemma}\label{Span segal condition lemma}
  The full subcategory of $\prn{\Span_I(\FF_{\cT})^{\el}_{/S}}^{\op} \simeq \cT \times_{\FF_{\cT}} \FF_{\cT, /S}$ consisting of morphisms $f\colon T \rightarrow S$ such that $f$ is a summand inclusion is an initial subcategory equivalent to the set $\Orb(S)$.
\end{lemma}
Unwinding definitions, this demonstrates the following.
\begin{corollary}\label{CMon I cat corollary}
    The forgetful functor
    \[
        \Seg_{\Span_I(\FF_{\cT})}(\cC) \rightarrow \Fun(\Span_I(\FF_{\cT}), \cC)
    \]
    is fully faithful with image spanned by the product preserving functors.
\end{corollary}
    
We call these \emph{global effective Burnside patterns}.
They are generally well behaved:
\begin{lemma}\label{Soundly extendable lemma}
  The pattern $\Span_I(\FF_{\cT})$ is soundly extendable.
\end{lemma}
\begin{proof}
  It is sound by \cite[Cor~3.3.24]{Barkan}.
  To see that $\Span(\FF_{\cT})$ is extendable, it is equivalent to prove that $\sA_{\Span(\FF_{\cT})}$ is a Segal $\Span_I(\FF_{\cT})$-$\infty$-category, i.e. for every $S \in \Span_I(\FF_{\cT})$, the associated functor $\varphi$ of 
  \[
    \begin{tikzcd}
      \Span_I(\FF_{\cT})^{\act}_{/S} \arrow[d] \arrow[r,"\sim"]
      & I_{/S} \arrow[r,"\sim"] \arrow[d]
      & \dprod_{V \in \mathrm{Orb}(S)} I_{/V} \arrow[dl, "\varphi"]\\ 
      \dlim_{V \in \Span(\FF_{\cT})^{\el}_{S/}} \Span(\FF_{\cT})^{\act}_{/V} \arrow[r,"\sim"]
      & \dlim_{V \in \cT \times_{\FF_{\cT}} \FF_{\cT,/S}} I_{/V}
    \end{tikzcd}
  \]
  is an equivalence.
  In fact, it is an equivalence by \cref{Span segal condition lemma}.
\end{proof}

\subsubsection{The equivalence}
We resume our original generality with $\cP \subset \cT$ an atomic orbital subcategory.
\begin{corollary}\label{Span fibrous corollary}
  $\varphi\colon \uFF^{\cP}_{\cT,*} \hookrightarrow \Span_{\cP}(\FF_{\cT})$ induces equivalences of categories 
  \begin{align*}
    \Seg_{\Span_{\cP}(\FF_{\cT})}(\cC) 
    &\simeq \Seg_{\uFF^{\cP}_{\cT,*}}(\cC);\\
    \Fbrs(\Span_{\cP}(\FF_{\cT})) 
    &\simeq \Fbrs(\uFF^{\cP}_{\cT,*}).
  \end{align*}
\end{corollary}
\begin{proof}
  The pattern $\Span(\FF_{\cT})$ is soundly extendable by \cref{Soundly extendable lemma}.
  In order to verify that $\varphi$ is a strong Segal morphism, we must verify that $\varphi^{\el}_{\brk{S \rightarrow V}/}$ is initial;
  in fact, it is an equivalence by 
  \cref{Segal condition lemma,Span segal condition lemma}.
  
  It remains to check that $\varphi$ satisfies the conditions of \cref{Equivalence theorem}.
  First, note that $\varphi^{\el}$ is an equivalence by construction.
  Second, note that there is a factorization
  \[
    \begin{tikzcd}[column sep = small]
      \Tot \uFF_{\cT,*, /\brk{V = V}}^{\cP,\act} \arrow[d, "s^{\act}"] \arrow[r,"\simeq", phantom]
      & \FF^{\cP}_{\cT,/V} \arrow[d, equals] \\
      \Span_{\cP}(\FF_{\cT})^{\act}_{/V} \arrow[r,"\simeq", phantom]
      & \FF^{\cP}_{\cT, /V} 
    \end{tikzcd}
  \]
  so $\varphi^{\act}_{/V}$ is an equivalence for all $V \in \cT^{\op} = \Tot \uFF_{\cT,*, /\brk{V=V}}^{\cP,\el}$.  
\end{proof}
\cref{Span fibrous theorem} follows by combining \cref{Pointed prop,CMon I cat corollary,Span fibrous corollary}.
\subsubsection{The $\cO$-monoidal case}
We refer to $\Fbrs(\Span_{\cP}(\FF_{\cT}))$ as the \emph{$\infty$-category of $\cP$-operads}.
\cref{Span fibrous theorem} yields two algebraic patterns underlying a $\cP$-operad:
\begin{align*}
  \Tot&\colon \Fbrs(\Span_{\cP}(\FF_{\cT})) \rightarrow \AlgPatt;\\
  \Tot \Tot_{\cT}&\colon \Fbrs(\Span_{\cP}(\FF_{\cT})) \simeq \Fbrs(\Tot \uFF_{\cT,*}^{\cP}) \rightarrow \AlgPatt.
\end{align*}
in fact, these yield the same algebraic theories.
\begin{corollary}\label{Two models for segal objects}
  Let $\cO^{\otimes}$ be a $\cP$-operad.
  Then, $\varphi^*$ induces equivalences
  \begin{align*}
    \Seg_{\Tot \cO^{\otimes}} (\cC) &\simeq \Seg_{\Tot \Tot_{\cT} \cO^{\otimes}}(\cC);\\
    \Fbrs(\Tot \cO^{\otimes}) &\simeq \Fbrs(\Tot \Tot_{\cT} \cO^{\otimes}).
  \end{align*}
\end{corollary}
This will follow immediately from the following proposition.
\begin{proposition}\label{Pullback morita prop}
  Suppose $\varphi\colon \fO \rightarrow \fP$ is a strong Segal morphism of algebraic patterns satisfying the conditions of \cref{Equivalence theorem} and which is a $\pi_0$-isomorphism and let $\fQ \rightarrow \fP$ be a fibrous pattern.
  Then, the pullback map 
  \[
    \varphi'\colon  \varphi^* \fQ \rightarrow \fQ
  \]
  satisfies the conditions of \cref{Equivalence theorem};
  moreover, if $\fP$ is soundly extendable, then $\fQ$ is soundly extendable.
\end{proposition}
\begin{proof}
  First note that strong Segal morphisms are closed under pullback, since initial functors are closed under pullback.
  Furthermore, fibrous patterns over soundly extendable patterns are soundly extendable \cite[Lem~4.1.15]{Barkan}, so we're left with verifying the conditions.
  Note that we acquire pullback diagrams
  \[
    \begin{tikzcd}[ampersand replacement=\&]
      {\varphi^*\fQ^{\el}} \& {\fQ^{\el}} \&\& {\varphi^*\fQ^{\act}} \& {\fQ^{\act}} \\
      {\fO^{\el}} \& {\fP^{\el}} \&\& {\fO^{\act}} \& {\fP^{\act}}
      \arrow[from=1-1, to=1-2]
      \arrow[from=1-1, to=2-1]
      \arrow["\lrcorner"{anchor=center, pos=0.125}, draw=none, from=1-1, to=2-2]
      \arrow[from=1-2, to=2-2]
      \arrow[from=1-4, to=1-5]
      \arrow[from=1-4, to=2-4]
      \arrow["\lrcorner"{anchor=center, pos=0.125}, draw=none, from=1-4, to=2-5]
      \arrow[from=1-5, to=2-5]
      \arrow["\sim"', from=2-1, to=2-2]
      \arrow[from=2-4, to=2-5]
    \end{tikzcd}
  \]
  which imply that $\varphi^* \fQ^{\el} \rightarrow \fQ^{\el}$ is an equivalence.
  Pick some $X \in \varphi^* \fQ$;
  then, we acquire pullback diagrams
  \[
    \begin{tikzcd}[ampersand replacement=\&]
	{\varphi^*\fQ^{\act}_{/X}} \& {\fQ^{\act}_{/\varphi' X}} \&\& {\prn{\varphi^*\fQ^{\act}_{/X}}^{\simeq}} \& {\prn{\fQ^{\act}_{/\varphi' X}}^{\simeq}} \\
	{\fO^{\act}_{/\pi X}} \& {\fP^{\act}_{/\varphi \pi X}} \&\& {\prn{\fO^{\act}_{/\pi X}}^{\simeq}} \& {\prn{\fP^{\act}_{/\varphi \pi X}}^{\simeq}}
	\arrow[from=1-1, to=1-2]
	\arrow[from=1-1, to=2-1]
	\arrow["\lrcorner"{anchor=center, pos=0.125}, draw=none, from=1-1, to=2-2]
	\arrow[""{name=0, anchor=center, inner sep=0}, from=1-2, to=2-2]
	\arrow[from=1-4, to=1-5]
	\arrow[""{name=1, anchor=center, inner sep=0}, from=1-4, to=2-4]
	\arrow["\lrcorner"{anchor=center, pos=0.125}, draw=none, from=1-4, to=2-5]
	\arrow[from=1-5, to=2-5]
	\arrow[from=2-1, to=2-2]
	\arrow[from=2-4, to=2-5, "\simeq"]
\end{tikzcd}
  \]
where the right is the core of the left, implying that $\prn{\varphi^* \fQ^{\act}_{/X}}^{\simeq} \rightarrow \prn{\fQ^{\act}_{/\varphi' X}}^{\simeq}$ is an equivalence, as desired.
\end{proof}

For example, we can quickly acquire a model for $I$-operads akin to \cite{Nardin}.
The global version of this uses the following proposition, whose proof is identical to that of \cref{Definition of NIinfty proposition}.
\begin{proposition}
  Let $I \subset \FF_{\cT}^{\cP}$ be a replete wide subcategory.
  Then, $\Span_I(\FF_{\cT}) \rightarrow \Span_{\cP}(\FF_{\cT})$ presents a $\cP$-operad if and only if $I \subset \FF_{\cT}^{\cP}$ is a weakly extensive subcategory.
\end{proposition}
Define the pullback pattern $\Tot \uFF_{I,*} \deq \Tot \uFF_{\cT,*}^{\cP} \times_{\Span_{\cP}(\FF_{\cT})} \Span_I(\FF_{\cT})$
\begin{corollary}
  $\varphi^*$ induces equivalences
  \begin{align*}
    \Seg_{\Span_I(\FF_{\cT})}(\cC) &\simeq \Seg_{\Tot \uFF_{I,*}}(\cC);\\
    \Fbrs(\Span_I(\FF_{\cT})) &\simeq \Fbrs(\Tot \uFF_{I,*}).
  \end{align*}
\end{corollary}

\begin{remark}\label{Glo remark}
  Let $\Orb \subset \Glo$ be the global orbit category including into the global indexing category (see e.g. \cite[Ex~4.3.3]{Cnossen_stable}).
  As remarked in \cite[Rmk~4.3.4]{Cnossen_stable}, atomic orbital subcategories of $\Glo$ correspond to \emph{global transfer systems} in the sense of \cite{Barrero};
  since $\Orb$ is the maximal atomic orbital subcategory of $\Glo$, these correspond canonically with extensive subcategories of $\FF_{\Glo}^{\Orb}$.
  If we interpret weakly extensive subcategories $I \subset \FF_{\Glo}^{\Orb}$ as \emph{global weak indexing categories}, the above work thus constructs \emph{global weak $\cN_\infty$-operads} and an equivalence between two models for \emph{global $I$-operads}.
\end{remark}

\subsubsection{$I^\infty$-operads as operadic coefficient systems}
We now prove the following result.
\begin{proposition}\label{I-infinity operad prop}
  The map $\varphi\colon \cT^{\op} \times \Span(\FF) \rightarrow \Span_{I^\infty}(\FF_{\cT})$ induces equivalences
  \begin{align*}
    \Seg_{\Span_{I^\infty}}(\FF_{\cT})(\cC) &\simeq \Fun(\cT^{\op}, \CMon(\cC));\\ 
      \Fbrs\prn{\Span_{I^\infty}(\FF_{\cT})} &\simeq \Fun(\cT^{\op}, \Op). 
  \end{align*}
\end{proposition}
\begin{proof}
  The right hand sides correspond with $\Seg_{\cT^{\op} \times \Span(\FF)}(\cC)$ and $\Fbrs(\cT^{\op} \times \Span(\FF))$, so it suffices to verify the conditions of \cref{Equivalence theorem} for $\varphi$.
  We already know that the codomain is soundly extendable \cref{Soundly extendable lemma}, and it is easy to see that $\varphi^{\el}$ and $\prn{\varphi^{\act}_{/O}}^{\simeq}$ are equivalences.
  Moreover, the fact that $\varphi$ is a Segal morphism follows from \cref{Commutative monoids corollary}, so we are done.
\end{proof}

\subsection{Pullback of fibrous patterns along Segal morphisms and sound extendability}
\begin{proposition}\label{Soundly extendable pull-push prop}
  Suppose $\varphi\colon \fO \rightarrow \fP$ is functor which is compatible with the inert-active factorization system, and $\fP$ is soundly extendable..
  Then,
  \begin{enumerate}
    \item If the precomposition functor
           $\varphi^*\colon \Fun(\fP,\Cat) \rightarrow \Fun(\fO,\Cat)$
       preserves Segal objects, then the pullback functor $\varphi^*\colon \Cat_{/\fP} \rightarrow \Cat_{/\fO}$
       preserves fibrous patterns.
      \item If $\varphi$ is an inert-cocartesian fibration and the left Kan extension functor
          $\varphi_!\colon \Fun(\fO,\Cat) \rightarrow \Fun(\cP,\Cat)$
        preserves Segal objects, then postcomposition 
            $\varphi_!\colon \Cat_{/\fO} \rightarrow \Cat_{/\fP}$
        preserves fibrous patterns.
  \end{enumerate}
  In particular, if $\varphi$ is an inert-cocartesian Segal morphism with soundly extendable codomain whose left Kan extension preserves Segal categories, then pullback and postcomposition restrict to an adjunction on fibrous patterns
  \[
    \varphi_!\colon \Fbrs(\fO) \rightleftarrows \Fbrs(\fP)\colon \varphi^*
  \]
\end{proposition}
\begin{proof}
  Our argument is only a minor variation of \cite[Lem~4.1.19]{Barkan}.
  In either case, the property of being an inert-cocartesian fibration is always preserved, either by assumption or by \cite[Obs~2.2.6]{Barkan}.
  
  We prove (1) first.
  Fixing $\sF \in \Fbrs(\fP)$, by \cite[Obs~4.1.3]{Barkan}, it suffices to prove that the left vertical arrow in the following pullback diagram is a relative Segal $\fO$-$\infty$-category.
  \[
    \begin{tikzcd}
      \St_{\fO}^{\Int}(\varphi^* \sF)
      \arrow[r] \arrow[d] \arrow[rd, "\lrcorner" very near start, phantom]
      & \varphi^* \St_{\fP}^{\Int} \sF \arrow[d]\\
      \sA_{\fO} \arrow[r]
      & \varphi^* \sA_{\fP}
    \end{tikzcd}
  \]
  By \cite[Lem~3.1.10]{Barkan}, relative Segal $\fO$-$\infty$-categories are pullback-stable, so it suffices to prove that the right vertical arrow is a relative Segal $\fO$-$\infty$-category.
  By sound extendability $\sA_{\fP}$ is a Segal $\fP$-$\infty$-category, and since $\varphi^*$ preserves Segal $\infty$-categories, $\varphi^* \sA_{\fP}$ is a Segal $\fO$-$\infty$-category;
  by \cite[Obs~3.1.8]{Barkan} it then suffices to prove that $\varphi^* \St_{\fP}^{\Int} \sF$ is a Segal $\fO$-$\infty$-category.
  Since $\varphi^*$ preserves Segal $\infty$-categories, it suffices to prove that $\St_{\fP}^{\Int} \sF$ is a Segal $\fP$-category, which follows by the assumption that $\sF$ is a fibrous pattern.

  (2) is similar;
  this time, by taking left adjoints to the commutative square of \cite[Prop~4.2.5]{Barkan}, it suffices to prove that the composition
  \[
    \varphi_! \St_{\fO}^{\Int} \sF \rightarrow \varphi_! \sA_{\fO} \rightarrow \sA_{\fP}
  \]
  is relative Segal;
  since $\fP$ is soundly extendable, \cite[Obs~3.1.8]{Barkan} again reduces this to verifying that $\varphi_! \St_{\fO}^{\Int} \sF$ is Segal;
  this follows from the facts that $\sF$ is a fibrous pattern and $\varphi_!$ preserves Segal $\infty$-categories.
\end{proof}

\subsection{Segal morphisms between effective Burnside patterns}\label{Segal morphisms subsection}
We now fill our grab bag with a wide variety of Segal morphisms between effective Burnside patterns.
\begin{proposition}\label{Segal morphisms between span patterns prop}
    Suppose $I \subset J \subset \FF_{\cT}$ are weakly extensive subcategories.
    Then, the inclusion
    \[
        \iota\colon \Span_{I}(\FF_{\cT}) \rightarrow \Span_{J}(\FF_{\cT})
    \]
    is a Segal morphism.
\end{proposition}
\begin{proof}
    We are tasked with verifying that precomposition with $\iota$ preserves product-preserving functors, i.e. that $\iota$ is a product-preserving functor.
    In fact, this is immediate, since a functor $\Span_I(\FF_{\cT}) \rightarrow \cC$ is product-preserving if and only if the backwards maps $(S \leftarrow U)_{U \in \mathrm{Orb}(S)}$ together map to a product diagram, which is obviously true of $\iota$.
\end{proof}

\begin{proposition}\label{Ind is Segal}
  Suppose $\varphi\colon V \rightarrow W$ is a morphism in $\cT$.
  Then, the associated functor 
  \[
    \Span_I(\Ind_V^W)\colon \Span_I(\FF_V) \rightarrow \Span_I(\FF_W)
  \]
  is a Segal morphism.
\end{proposition}
\begin{proof}
  We're tasked with proving that precomposition along $\Span(\Ind_V^W)$ preserves product-preserving functors, i.e. it is a product-preserving functor.
  Since $\Span_I(\FF_V)$ and $\Span_I(\FF_W)$ are semiadditive, it is equivalent to prove that $\Span(\Ind_V^W)$ is coproduct-preserving;
  since coproducts in $\Span_I(\FF_V)$ are computed in $\FF_V$, it's equivalent to prove that $\Ind_V^W\colon \FF_V \rightarrow \FF_W$ is coproduct-preserving, which follows from the fact that it's a left adjoint. 
\end{proof}

\begin{proposition}\label{Segal morphism}\label{Segal morphisms between span patterns prop 2}
  If $f\colon \cT' \rightarrow \cT$ is a functor of $\infty$-categories sending an atomic orbital subcategory $\cP' \subset \cT'$ into an atomic orbital subcategory $\cP \subset \cT$, then the associated functor $\Span_{\cP'}(\FF_{\cT'}) \rightarrow \Span_{\cP}(\FF_{\cT})$ is a Segal morphism.
\end{proposition}
\begin{proof}
  By \cite[Rem~4.3]{Chu}, it suffices to verify that $f^{\el}_{X/}$ induces an equivalence on the left vertical arrow
  \[
    \begin{tikzcd}
      \dlim_{\Span_{\cP}(\cT)^{\el}_{f(X)/}} F \arrow[d] \arrow[r,"\simeq", phantom]
      & \dprod_{U \in \mathrm{Orb}(f(X))} F(U) \arrow[d]\\
      \dlim_{\Span_{\cP'}(\cT')^{\el}_{X/}} F \circ f^{\el} \arrow[r,"\simeq", phantom]
      & \dprod_{V \in \mathrm{Orb}(X)} F f(V)  
    \end{tikzcd}
  \]
  whenever $F$ is restricted from a Segal $\Span_{\cP}(\FF_{\cT})$ space.
  This follows by noting that the horizontal arrows are equivalences and $\Span(f)$ sends the set of orbits of $X$ bijectively onto the set of orbits of $f(X)$.
\end{proof}

\begin{proposition}\label{Cartesian product is Segal}
  If $\cP \subset \cT$ is an atomic orbital subcategory such that $\cP,\cT$ have compatible terminal objects, then the induced functor 
  \[
    \wedge \deq \Span(\times)\colon \Span_{\cP}(\FF_{\cT}) \times \Span_{\cP}(\FF_{\cP}) \xrightarrow{\wedge} \Span_{\cP}(\FF_{\cT})
  \]
  is compatible with Segal objects.
\end{proposition}
\begin{proof}
  By \cite[Ex~5.7]{Chu}, a functor $\Span_{\cP}(\FF_{\cT}) \times \Span_{\cP}(\FF_{\cT}) \rightarrow \cC$ is a Segal object if and only if it preserves products separately in each variable.
  Hence we're tasked with verifying that $\wedge^*F$ preserves products separately in each variable whenever $F$ preserves products.
  In fact, this follows by distributivity of products and coproducs in $\FF^{\cP}_{\cT}$;
  indeed, we have
  \begin{align*}
    \wedge^* F\prn{\prn{X_+ \oplus Z_+,Y_+}} 
      &\simeq F\prn{\prn{X \sqcup X'} \times Y}_+\\
      &\simeq F\prn{\prn{X \times Y} \sqcup \prn{X' \times Y}}_+\\
      &\simeq F\prn{\prn{X_+ \wedge Y_+} \oplus \prn{X'_+ \wedge Y_+}}\\
      &\simeq F\prn{X_+ \wedge Y_+} \oplus F\prn{X'_+ \wedge Y_+}\\
      &\simeq \wedge^*F\prn{X_+,Y_+} \oplus \wedge^*F\prn{X'_+,Y_+}.\qedhere
    \end{align*}
\end{proof}

\resumetocwriting
\end{appendix}
\resumetocwriting

\printbibliography

\end{document}